\documentclass[
12pt, 
a4paper, 
oneside, 
headinclude,footinclude, 
]{article}

\usepackage[
nochapters, 
beramono, 
eulermath,
pdfspacing, 
dottedtoc 
]{classicthesis} 

\usepackage{arsclassica} 

\usepackage[T1]{fontenc} 

\usepackage[utf8]{inputenc} 

\usepackage{graphicx} 

\usepackage{enumitem} 

\usepackage{amsmath,amssymb,amsthm} 

\usepackage{varioref} 

\hypersetup{
colorlinks=true, breaklinks=true, bookmarks=true,bookmarksnumbered,
urlcolor=webbrown, linkcolor=RoyalBlue, citecolor=webgreen, 
pdftitle={}, 
pdfauthor={\textcopyright}, 
pdfsubject={}, 
pdfkeywords={}, 
pdfcreator={pdfLaTeX}, 
pdfproducer={LaTeX with hyperref and ClassicThesis} 
} 

\usepackage{tikz}
\usetikzlibrary{decorations.pathreplacing, calligraphy}

\usepackage{hyperref}
\usepackage{bm, stmaryrd}											
\usepackage{a4wide}
\usepackage{mathdots}
\usepackage[english]{babel}
\usepackage{xcolor}
\usepackage{float}
\usepackage{bigints}
\usepackage{subcaption}
\usepackage{accents}

\allowdisplaybreaks

\newtheorem{thm}{Theorem}[section]
\newtheorem{lem}[thm]{Lemma}
\newtheorem{prop}[thm]{Proposition}
\newtheorem{conj}[thm]{Conjecture}
\newtheorem{coro}[thm]{Corollary}

\newtheorem*{thma}{First Main Theorem}
\newtheorem*{thmb}{Second Main Theorem}
\newtheorem*{thmc}{Tubular Neighbourhood Area Theorem}
\newtheorem*{thmaux}{Auxiliary Theorem}

\newtheorem*{vh}{Area Heuristic}
\newtheorem*{rvh}{Regular Area Heuristic}

\theoremstyle{definition}				
\newtheorem*{rem}{Remark}

\numberwithin{equation}{section}

\newcommand{\R}{\mathbb{R}}

\newcommand{\N}{\mathbb{N}}
\newcommand{\Z}{\mathbb{Z}}
\newcommand{\Q}{\mathbb{Q}}

\newcommand{\bP}{\mathbb{P}}

\newcommand{\E}{\mathbb{E}}
\newcommand{\V}{\mathbb{V}}

\newcommand{\dist}{\detokenize{dist}}

\newcommand{\erfc}{\detokenize{erfc}}

\newcommand{\ppi}{\overline{\pi}}

\title{\normalfont\spacedallcaps{Rational Points and Brownian Motion}} 

\author{\spacedlowsmallcaps{Faustin ADICEAM\textsuperscript{1} \& \spacedlowsmallcaps{Volodymyr PAVLENKOV\textsuperscript{2}} \& Evgeniy ZORIN\textsuperscript{3} } }

\date{} 

\begin{document}

\maketitle


\renewcommand{\sectionmark}[1]{\markright{\spacedlowsmallcaps{#1}}} 
\lehead{\mbox{\llap{\small\thepage\kern1em\color{halfgray} \vline}\color{halfgray}\hspace{0.5em}\rightmark\hfil}} 

\pagestyle{scrheadings} 

\begin{flushright}
\textit{Vasily Ivanovich Bernik (1947---2026) \& Michael Maurice Dodson (1937---2026) \\ in memoriam}
\end{flushright} 

\vspace{2mm}

\begin{abstract}
\noindent Given a real--valued function $f$ defined over a bounded interval with nonempty interior, let $\mathcal{N}_f\left(\delta, Q\right)$ denote the number of rational points in the plane with deno\-minators at most $Q\ge 1$ lying in the $\left(\delta/Q\right)$--tubular neighbourhood of the graph of the function $f$. A heuristic predicts that the number of such points should grow like the area of the  neighbourhood provided that $\delta$ (the parameter determining the thickness of the domain) is big enough compared with some negative power of the bound $Q$. \\

\noindent Considerable efforts have been committed to prove this heuristic in the case of sufficiently regular curves. This culminated in the works by Vaughan \& Velani~(2006) and  by Huang~(2015) establishing an asymptotic expansion for the counting function $\mathcal{N}_f\left(\delta, Q\right)$ provided that $\delta\gg Q^{-1+\epsilon}$ for some $\epsilon>0$, and provided that the map $f$ is twice continuously differentiable with non--vanishing, Lipschitz continuous second derivative.\\

\noindent The present work deals with the thus--far unexplored complementary regime where minimal regularity conditions are imposed on the curve. More precisely, it is concerned with the case where the function $f$  is an almost sure realisation of the graph of  Brownian motion (and, more generally, of any stochastic process which has an absolutely continuous law with respect to that of Brownian motion).\\

\noindent The main result establishes the existence of an almost sure asymptotic expansion for the counting function for all possible values of $\delta$, with the sole exception of an isolated critical regime, thereby going well beyond the theory currently available for regular curves. This exceptional regime is conjectured to correspond to a ge\-nui\-ne transition point in the asymptotic behavior. A key ingredient in the proof is the derivation of the area heuristic, which relies on establishing the almost sure asymptotics of the area of the tubular neighbourhood of the graph of Brownian motion. \\

\noindent This result, and indeed the theory developed towards its proof, have two main consequences~: firstly, they complete the counting aspect of the theory of Diophantine approximation on the graph of  Brownian motion ini\-tia\-ted by Sprind\v{z}uk (1979). 
Secondly, they hint at the existence of a theory unifying, through the area heuristic, the analysis of rational points near curves on the one hand and, on the other,  the local Hölder regularity and fine-scale oscillations of the curve. They thus build a seemingly new bridge between Number Theory and Multifractal Analysis.

\end{abstract}

\tableofcontents

\let\thefootnote\relax\footnotetext{\textsuperscript{1} {Laboratoire d’analyse et de mathématiques appliquées (LAMA), Université Paris-Est Créteil, France,} \texttt{faustin.adiceam@u-pec.fr}}
\let\thefootnote\relax\footnotetext{\textsuperscript{2}Department of Mathematics, University of York,
Heslington, York, YO10 5DD, England, {\texttt{volodymr.pavlenkov@york.ac.uk}}.}
\let\thefootnote\relax\footnotetext{\textsuperscript{3} Department of Mathematics, University of York,
Heslington, York, YO10 5DD, England, \texttt{evgeniy.zorin@york.ac.uk}.\\}

\section{Introduction}

\subsection{The General Setup}

\noindent Let $I$ be a bounded interval in the real line with nonempty interior. Let $f~: I\rightarrow \R$ be a bounded map. Given an integer $Q\ge 1$ and a real $\delta=\delta(Q)>0$ (possibly depending on $Q$), denote  by $\mathcal{N}_f\left(\delta, Q\right)$ the function counting   the number of rational points with deno\-mi\-nators at most $Q$ lying in 
a ($\delta/Q$)--neighbourhood of the curve $$\mathcal{C}_f\;=\;\left\{(x,f(x)) : x\in I\right\}.$$ Explicitly, employing here and throughout the symbol $\sharp$ to denote the cardinality of a finite set,
\begin{equation}\label{defcount}
\mathcal{N}_f\!\left(\delta, Q\right)\;=\; \sharp\left\{ \left(p,q,r\right)\in\Z^3\;:\; 1\le q\le Q \quad \textrm{and}\quad  \left(\frac{p}{q}, \frac{r}{q}\right)\in \mathfrak{T}_f\!\left(\frac{\delta}{Q}\right) \right\}.
\end{equation} 
Here\textsuperscript{4}\let\thefootnote\relax\footnotetext{\textsuperscript{4}In probability theory, and in particular in the study of the properties of Brownian motion, the notation $\Q$ is commonly used to denote a probability measure (especially in the context of the well--known Girsanov Theorem --- see, e.g., \cite[Theorem~6.4]{B}). In the entirety of this paper,  this notation is  employed only in its Number Theoretical meaning to denote the set of rational numbers. \\ Also,  the definition of the counting function  $\mathcal{N}_f\!\left(\delta, Q\right)$ in~\eqref{defcount} accounts for the fact that rationals are not assumed to be irreducible. A standard argument based on M\"obius inversion enables one to recover the analogues of all counting results established hereafter in the case of primitive rational points, that is, under the additional assumption that $\gcd(p,q,r)=1$. See, e.g., \cite[\S2]{H} for further details.}, given a parameter $\eta>0$, $\mathfrak{T}_f\left(\eta\right)$ denotes the $\eta$--tubular neighbourhood of the curve  $\mathcal{C}_f$ with respect to the sup--norm. This is defined as 
\begin{equation*}
\mathfrak{T}_f\!\left(\eta\right)= \left\{\bm{y}\in  \R^2\;:\; \dist\left(\mathcal{C}_f, \bm{y}\right)\le \eta\right\}, 
\end{equation*} 
where, given $\bm{y}=\left(y_1, y_2\right)\in\R^2$, 
\begin{equation}\label{defdist}
\dist\left(\mathcal{C}_f, \bm{y}\right)\;=\; \inf_{x\in I}\;\max\left\{\left|x-y_1\right|, \left|f(x)-y_2\right|\right\}.
\end{equation} 

\noindent The asymptotic behavior as $Q$ tends to infinity of the counting function $\mathcal{N}_f\left(\delta, Q\right)$  is go\-ver\-ned by the \emph{area  heuristic} based on the assumption that the set of rationals  
\begin{equation}\label{boundrat}
\mathfrak{S}_{\Q}(Q)\;=\; \left\{\left(\frac{p}{q}, \frac{r}{q}\right)\in\Q^2\;:\; p,r\in\Z \quad \textrm{and}\quad 1\le q\le Q\right\}^*,
\end{equation}
where the star indicates that each rational point is counted with multiplicity, is equidistributed in suitable neighbourhoods of the curve $\mathcal{C}_f$ (in a sense to be specified). This is saying that the number of rationals in $\mathfrak{S}_{\Q}(Q)$ falling in    $\mathfrak{T}_f\left(\delta/Q\right)$ should be proportional to the area of this set (assumed to be measurable). This can be formalised as follows~:


\begin{vh}
Assume that for any given $\eta>0$, the tubular neighbourhood  $\mathfrak{T}_f\!\left(\eta\right)$ is measurable and denote by $\mathcal{A}_f\!\left(\eta\right)$ its two--dimensional Lebesgue measure. Then, for suitable ranges of the thickness measure $\delta=\delta(Q)$, 
\begin{equation}\label{ah0}
\mathcal{N}_f\!\left(\delta, Q\right)\;\asymp\; \mathcal{A}_f\!\left(\frac{\delta}{Q}\right)\cdot Q^3
\end{equation}
\end{vh}

\noindent In the above and throughout, given positive reals $x$ and $y$ depending on some para\-meters, the notations $x \asymp y$, $x\ll y$ and $x\gg y$ mean the existence of a constant $c>0$ independent of these parameters such that $c^{-1}\cdot y \le x \le c\cdot y$, $\; x\le c\cdot y$ and $ x\ge c\cdot y$, respectively. The quantity $c$  is then referred to as the \emph{implicit constant}.\\

\noindent A few points deserve to be further justified and specified in the formalisation of the above heuristic. Firstly, the factor $Q^3$ should be seen as accounting for the number of rationals in the set $\mathfrak{S}_{\Q}(Q)$ falling in a given unit square. Under the assumption of equidistribution, the intersection of $\mathfrak{S}_{\Q}(Q)$ with the tubular neighbourhood $\mathfrak{T}_f\!\left(\delta/Q\right)$ is comparable to the product of the area of this neighbourhood with $Q^3$, which explains the form of the relation in~\eqref{ah0}. \\

\noindent Secondly, the equidistribution assumption should be expected to hold only  if the set $\mathfrak{T}_f\!\left(\delta/Q\right)$ is thick enough compared with the number of rational points available. This observation is usually formalised in the literature by  characterising the ranges of the thickness measure $\delta$ over which the heuristic is valid as a   lower bound of the form
\begin{equation}\label{ah}
\delta\gg Q^{-\alpha} \qquad \textrm{for some exponent}\qquad  \alpha>0.
\end{equation}
{Denote by $\alpha(f)$  the {(possibly infinite)} supremum of the admissible values of $\alpha$ for which the heuristic holds (which can happen only in the regime where the counting function $\mathcal{N}_f\!\left(\delta, Q\right)$ does not vanish).  
This quantity  can} then be referred to as the \emph{Area-Heuristic exponent}. It captures in particular the balance and  interplay between the two types of rational points involved in the counting function $\mathcal{N}_f\!\left(\delta, Q\right)$ : on the one hand those lying \emph{on} the curve (these are always included in the count and might ivalidate the \emph{area} heuristic, if comparatively too numerous, as they lie on a one-dimensional structure) and those genuinely  lying in a neighbourhood of the curve (but not on it, in such a way that their contribution depends on the value of $\delta$ and thus on the area of the domain they lie in). \\

\noindent The former are referred to as the \emph{intrinsic rational points}; they can be expected to contribute most to the counting function $\mathcal{N}_f\!\left(\delta, Q\right)$ when {$\delta\ll Q^{-\alpha}$} for some $\alpha>\alpha(f)$ (i.e.~when the tubular {neighbourhood} is ``small''), and depend on intrinsic properties of the curve determined by $f$ (for instance the arithmetic constraints that may be imposed in the equations defining it). The latter are referred to as the \emph{extrinsic rational points}; they can be expected to contribute most to the counting function $\mathcal{N}_f\!\left(\delta, Q\right)$ when {$\delta\gg Q^{-\alpha}$} for some $\alpha<\alpha(f)$  (i.e.~when the tubular neighboorhood is ``large''),  and depend on the local structure of the curve in the plane (which may depend, for instance, on its regularity or its curvature).

\subsection{The Case of Regular Curves}\label{subsec2}

\noindent In practice, the Area Heuristic  has been dealt with in the case where the function $f$ is regular enough, and more specifically in the case where it is at least continuously differentiable over the interval $I$ with a bounded derivative. It  is then easily seen that  $$\mathcal{A}_f\!\left(\frac{\delta}{Q}\right)\;\asymp\; \left|I\right|\cdot \frac{\delta}{Q}\, ,$$ where $\left|I\right|$ denotes the length of the interval $I$ and where the implicit constant is allowed to depend on the function $f$. The heuristic then specialises to the following form~: 

\begin{rvh}
Assume that the map $f~: I\rightarrow \R$ is continuously differentiable over $I$ with a bounded derivative. Then, 
\begin{equation}\label{rah}
\mathcal{N}_f\left(\delta, Q\right)\;\asymp\;  \left|I\right| \cdot \delta\cdot Q^2\qquad \textrm{when}\qquad \delta\gg Q^{-\alpha} \qquad \textrm{for some}\qquad  \alpha>0.
\end{equation}
\end{rvh}

\noindent The specification of an admissible value for the exponent $\alpha$ in this case, and indeed of the Area-Heuristic exponent, involves the consideration of the following type of obstructions. Assume that the function $f$ determines an algebraic curve defined over the rationals of degree $d\ge 2$. A simple argument based on a Taylor expansion  (illustrated in~\cite{A2} and proved in~\cite[Lemma~1]{BDL}) shows that a gap phenomenon occurs~:  when $\delta=o\!\left(Q^{1-d}\right)$ (in particular, when $\delta=o\!\left(Q^{-1}\right)$ if $d=2$), the  function $\mathcal{N}_f\!\left(\delta, Q\right)$ reduces to solely counting rational points \emph{on} the curve $\mathcal{C}_f$, and therefore solely to the intrinsic contribution. As previously pointed out, one can then not hope for some \emph{area} heuristic to hold.  \\

\noindent In the case that the map $f$ is at least twice differentiable, the Regular Area Heuristic  can be made more precise to take into account the possibility that the curve  $\mathcal{C}_f$ may accumulate an abnormally large number of rationals around points where it has a high order of contact with its tangent, in particular when the tangent is, furthermore, a rational line. To avoid such a degenerate situation, it is then assumed that its second derivative remains bounded  in absolute value both from above and away from zero  (that is, that the curve is convex or concave on its domain, and that the curvature is uniformly bounded above and  below).\\

\noindent Establishing the Regular Area Heuristic under these curvature assumptions has attracted a considerable amount of effort involving a large variety of  methods and theories. 
To cite but a few milestones in this quest, building on the earlier work  by Jarn\'ik~\cite{J}, by Swinnerton--Dyer~\cite{S-D} and by Bombieri \& Pila~\cite{BP}, Huxley~\cite{Hx} employed the so--called determinant method to establish the upper bound in the first relation in~\eqref{rah} in the weaker form $$\mathcal{N}_f\left(\delta, Q\right)\ll \left|I\right| \cdot \delta^{1-\eta}\cdot Q^2\quad \textrm{for any\quad }\eta>0\quad  \textrm{whenever}\quad \delta\gg Q^{-(1-\varepsilon)}\quad \textrm{for some}\quad \varepsilon>0.$$ Parallelly, the exact lower bound (without the $\eta$ factor) has been obtained in the same range of $\delta$'s by Beresnevich, Dickinson \& Velani~\cite{BDV} and then generalised to a larger class of functions by Beresnevich and Zorin~\cite{BZ} with a view towards applications to Metric Number Theory. \\

\noindent Getting back to the upper bound, the additional exponent $\eta>0$ in Huxley's estimate has subsequently been remarkably removed by Vaughan \& Velani~\cite{VV} through a mixture of ideas involving Fourier analytic methods and the Legendre transform of the function $f$ (which exists under the assumption of uniform bounds on the curvature). A refinement of their method culminated in the work by Huang~\cite{H}, 
where a precise asymptotic estimate for the counting function is obtained~:  

\begin{thm}[Huang, 2015]\label{huang}
Let $f~: I\rightarrow \R$ be a twice continuously differentiable function over a bounded interval $I$ such that $\left|f''\right|\asymp 1$ and such that the second derivative is, furthermore, Lipschitz continuous. Then, as $Q$ tends to infinity, 
$$\mathcal{N}_f\left(\delta, Q\right)\; =\;  \left(1+o(1)\right)\cdot \left|I\right| \cdot \delta\cdot Q^2\qquad \textrm{whenever}\qquad \delta\gg Q^{-(1-\varepsilon)} \qquad \textrm{for some}\qquad  \varepsilon>0. $$
\end{thm}

\noindent One of the {striking} features of this result is its uniformity~:  provided that the tubular neighbourhood is thick enough, and under the considered curvature assumptions, all dependence on the specific function $f$ is relegated to the error term. This error term takes an explicit form which varies in the work by Huang~\cite{H} and in its subsequent extension by Gafni~\cite{G}. Also, Huang~\cite{H} obtains the above asymptotic for a larger class of functions which are only differentiable once with a Lipschitz continuous derivative; the cost to pay is that it is valid only in the smaller range $\delta\gg Q^{-2/3}$.\\

\noindent Recent developments in this theory deal with the even more refined and difficult case where one considers very small tubular neighbourhoods of a regular enough curve, for instance when one assumes that $\delta=o(Q^{-1})$. As noted above, no uniform counting estimates can then be expected as the behavior of the function $\mathcal{N}_f\!\left(\delta, Q\right)$ heavily depends on the intrinsic Diophantine properties of the curve $\mathcal{C}_f$. Most of the current approaches have dealt with the case where the function $f$ is algebraic. Explicit (but probably not optimal) upper bounds for $\mathcal{N}_f\!\left(\delta, Q\right)$ are then obtained with the help of the determinant method  in some particular cases in~\cite{H-B} and in the general case in~\cite{AM2}. \\

\noindent Finally, it should be noted that the theory of counting rational points lying near higher dimensional regular enough manifolds  has also been   very actively developed. 
The reader is referred, e.g., to~\cite{A1, BVVZ_d, BVVZ_e} and to the references within for an account on the related state of the art.

\subsection{Counting Rational Points near a Trajectory of  Brownian Motion}

\noindent The main goal of the present work is to tackle the Area Heuristic in the so--far unexplored case where minimal regularity conditions are required for the function $f$. This is achieved upon working with one of the most canonical examples of such a function, namely an almost sure realisation of the standard linear Brownian motion. It defines a (class of) continuous but nowhere differentiable function(s).  In doing so, a parallel goal is to  complete  the counting aspect of the  programme set out by {Sprind\v{z}uk} in~\cite[Chap.~2, Section~13]{S}, where the theory of Diophantine approximation on the graph of  Brownian motion is initiated. \\

\noindent Specifically, {Sprind\v{z}uk} is concerned  with the closely related but weaker problem of establishing the \emph{extremality} of an almost sure realisation of the graph of the  linear Brow\-nian motion. This is saying that almost surely,  the Lebesgue measure of the set of  times $t>0$ such that the point on the graph of Brownian motion at time $t$ 
should satisfy a certain Diophantine condition is zero. This condition is that for some $\varepsilon>0$,  the point should  lie infinitely often within a distance $q^{-3/2-\varepsilon}$ of  a rational vector, the two coordinates of which have the same denominator $q\ge 1$~(\textsuperscript{5}\let\thefootnote\relax\footnotetext{\textsuperscript{5}To be more precise, {Sprind\v{z}uk} proves in~~\cite[Chap.~2, Theorem~12]{S} the so--called  dual form of the stated property of extremality. The formulation adopted here is equivalent to Sprid\v{z}uk's from Khintchine's classical Transference Theorem --- see, e.g., in~\cite[Chap.~V]{Ca}.}).\\

\noindent Some notation and terminology are now introduced to state the main results. For a standard reference on the theory of Brownian motion, the reader is referred to~\cite{MP}.\\

\noindent Let $\left(\Omega, \mathcal{F}, \mathbb{P}\right)$ be a probability space. Recall that the \emph{standard linear Brownian motion}  is a stochastic process (i.e.~a family of random functions) which to each $\omega\in\Omega$ associates a map $t\ge 0\mapsto B(t,\omega)$ meeting the following four properties~: \textbf{(1)} almost surely (i.e.~for  $\mathbb{P}$--almost all $\omega\in\Omega$), the graph of Brownian motion starts at the origin in the usual sense   that $B(0,\omega)=0$; \textbf{(2)} almost surely, the map $t\ge 0\mapsto B(t,\omega)$ is continuous; \textbf{(3)} the property of \emph{independence of increments} is satisfied; that is, for all $t_0 = 0\le t_1\le t_2\le \dots \le t_n$, the random variables $B(t_n,\;\cdot\;)-B(t_{n-1},\;\cdot)\;, \;\;\dots\;\; , B(t_1,\;\cdot\;)-B(t_{0}, \;\cdot\;)$ are independent; \textbf{(4)} given any $t\ge 0$ and $h>0$, the increments  $B(t+h,\;\cdot\;)-B(t,\;\cdot\;)$ are normally distributed with expectation zero and variance $h$.\\

\noindent As usual, the dependence on $\omega\in\Omega$ is  omitted and the stochastic process denoted by $\left\{B(t) \right\}_{t\ge 0}$. Its graph is the random set $\mathcal{G}\subset  \R_{\ge 0}\times \R$ whose restriction to times $T_1\le t\le T_2$, where $T_2>T_1\ge 0$, is denoted by $\mathcal{G}\!\left(\bm{\mathcal{T}}\right)$ (here and throughout,  $\bm{\mathcal{T}}$ is shorthand notation for the vector $\left(T_1, T_2\right)$). In other words, 
\begin{equation}\label{greph}
\mathcal{G}\!\left(\bm{\mathcal{T}}\right)\; :=\; \left\{\left(t, B(t)\right)\;:\; T_1\le t\le T_2\right\}\;\subset\;  \R_{\ge 0}\times \R.
\end{equation}  
It is standard, see~\cite[Theorem~1.27]{MP}, that almost surely, the graph of Brownian motion defines a continuous but nowhere differentiable curve. Furthermore, Taylor's Theorem (see~\cite[Theorem~4.29]{MP}) states that the almost sure Hausdorff dimension of the graph~\eqref{greph} is 3/2.  \\

\noindent The Area Heuristic is addressed in this context by generalising the definition of the counting function in~\eqref{defcount} to allow a much {larger} class of  approximation functions. Let then $\xi : x\ge 1\mapsto \xi(x)\in(0, 1]$ be any map. Given times $0\le T_1<T_2$, define the random counting function  
\begin{align}\label{defcountrand0}
\widehat{\mathcal{N}}_{\xi}\!\left(\bm{\mathcal{T}}; Q\right)\;=\;
&\sharp\Bigg\{ \left(p,q,r\right)\in\Z^3\;:\;  \dist\!\left(\bm{\rho}, \mathcal{G}\!\left(\bm{\mathcal{T}}\right)\right)\le \frac{\xi(q)}{2q},\nonumber \\
& \qquad\qquad\qquad\qquad \; \textrm{where}\qquad  \bm{\rho}=\left(\frac{p}{q}, \frac{r}{q}\right)\qquad\textrm{and}\qquad 1\le q\le Q\Bigg\}
\end{align}
(recall that the $\dist$ function is the one introduced in~\eqref{defdist}). In the particular case where $ \xi$ is the  approximation function
\begin{equation}\label{defpsideltaq} 
\xi^{(\delta)}_{Q}~: \; q\ge 1\;\mapsto\; {2}\delta\cdot \min\left\{\frac{q}{Q}, 1\right\}\qquad \textrm{for some} \;\qquad  \delta\; =\; \delta(Q)\in\left(0, {1/4}\right),
\end{equation} 
set 
\begin{align}\label{defcountrand1}
\mathcal{N}^{\,\flat}\!\left(\bm{\mathcal{T}};\left(\delta, Q\right)\right)\;=\; \widehat{\mathcal{N}}_{\xi^{(\delta)}_{Q}}\!\left(\bm{\mathcal{T}}; Q\right).
\end{align}
This is the exact analogue of the   deterministic counting function~\eqref{defcount} in the sense that the function $\mathcal{N}^{\,\flat}\!\left(\bm{\mathcal{T}};\left(\delta, Q\right)\right)$ counts the number of elements in the intersection of the set of rational points $\mathfrak{S}_{\Q}(Q)$  defined in~\eqref{boundrat} with the set $\widehat{\mathfrak{T}}\left(\bm{\mathcal{T}};\delta/Q\right)$. Here,
given $\eta>0$, 
\begin{equation}\label{count2}
\widehat{\mathfrak{T}}\!\left(\bm{\mathcal{T}}; \eta\right)\;=\; \left\{\bm{s}\in   {\R^2}\;:\; \dist\left(\bm{s}, \mathcal{G}\!\left(\bm{\mathcal{T}}\right)\right)\le \eta\right\}
\end{equation} 
is the $\eta$--tubular neighbourhood of the graph of Brownian motion between the two times $T_1$ and $T_2$.  \\
 
\noindent Determining the asymptotic behavior of the counting function~\eqref{defcountrand0} is the natural level of generality in  Diophantine approximation, where the  size of the target set is allowed to   depend on the denominators of the rational approximants. This, in particular, corresponds to the framework considered by Sprind\v{z}uk in~\cite{S} when giving impetus to the theory of  Diophantine approximation of trajectories of stochastic processes. \\

\noindent It should then be expected that the counting problem under consideration, usually referred to as \emph{asymptotic counting},   is governed by the cumulative expected contributions of all denominators.  The theorem below confirms this heuristic by determining the almost sure explicit shape of the individual contributions of each denominator under a mild assumption of monotonicity, natural in the Diophantine theory, on the  approximation function. 

\begin{thma}[Asymptotic Counting]\label{mainthm}
Fix an approximation function $\xi~: \left[1, \infty\right)\rightarrow \left(0,1/2\right)$ such that the map $q\mapsto \xi(q)/q$ is nonincreasing. Let $T_2>T_1>0$ and $\eta>0$ be reals, and let 
\begin{equation*}
V\!\left(\bm{\mathcal{T}}\right) \;=\; \max\left\{1,\;  \frac{1}{\sqrt{T_1}}, \; T_2-T_1,\; \sqrt{T_2}-\sqrt{T_1},\; \log\left(\frac{T_2}{T_1}\right),\; \frac{1}{T_2-T_1}\right\}.
\end{equation*}
Define the partial sums 
\begin{equation}\label{partial-sum}
\Xi\left(Q\right)\;=\; \sum_{q=1}^Qq^{3/2}\cdot\sqrt{\xi(q)}.
\end{equation} 
 Then, almost surely as $Q$ tends to infinity,
\begin{align}\label{maincount}
 \widehat{\mathcal{N}}_{ \xi}\left(\bm{\mathcal{T}}; Q\right)\;\;=\; &\frac{2\sqrt{2}}{\sqrt{\pi}}\cdot \left(T_2-T_1\right)\cdot\Xi\left(Q\right)\;+\; O\!\left(\mathcal{E}_\eta\!\left(\bm{\mathcal{T}}, Q\right)\right),
 \end{align}
where the error  term $\mathcal{E}_\eta\!\left(\bm{\mathcal{T}}, Q\right)$ is 
\begin{align*}
\mathcal{E}_\eta\!\left(\bm{\mathcal{T}}, Q\right)\;=\; & V\!\left(\bm{\mathcal{T}}\right)  \cdot \max_{1\le q\le Q}\sqrt{q^3\cdot \xi(q)}
\;+\;V\!\left(\bm{\mathcal{T}}\right)\cdot\sum_{q=1}^Q\sqrt{q\cdot\xi(q)}\cdot\left(\log q +\sqrt{q\cdot \xi(q)}\right)\\
  &+\; \sqrt{V\!\left(\bm{\mathcal{T}}\right)\cdot\Xi\left(Q\right)}\cdot \left(1+ \Xi\left(Q\right)^{2/5+\eta}\right)\cdot \left|\log\left(V\!\left(\bm{\mathcal{T}}\right)\cdot\Xi\left(Q\right)\right)\right|^{3/2+\eta}.
\end{align*}
The implicit constant  in~\eqref{maincount} depends at most on the specific trajectory of Brownian motion and on the choice of the exponent $\eta>0$ provided that the counting function $ \widehat{\mathcal{N}}_{ \xi}\left(\bm{\mathcal{T}}; Q\right)$ is, furthermore, restricted to those rationals with denominators
\begin{equation*}
q\;> \; \max\left\{\frac{4}{T_1} , \;  \frac{1}{T_2-T_1}, \;  \frac{2}{T_2} , \;  \left(\frac{V\!\left(\bm{\mathcal{T}}\right)}{T_2-T_1}\right)^2\right\}.
\end{equation*}
\end{thma}

\noindent It should be noted that the statement remains valid when the series defined by the partial sums~\eqref{partial-sum} converges. Also, the assumption $T_1>0$  is justified by the fact that Brownian motion admits a singular behavior near the  origin as it  takes there, by definition, an almost sure fixed value. See Figure~\ref{fig1} for some illustrations of graphs of realisations of  Brownian motion. 

\begin{figure}[H]
\includegraphics[width=\linewidth]{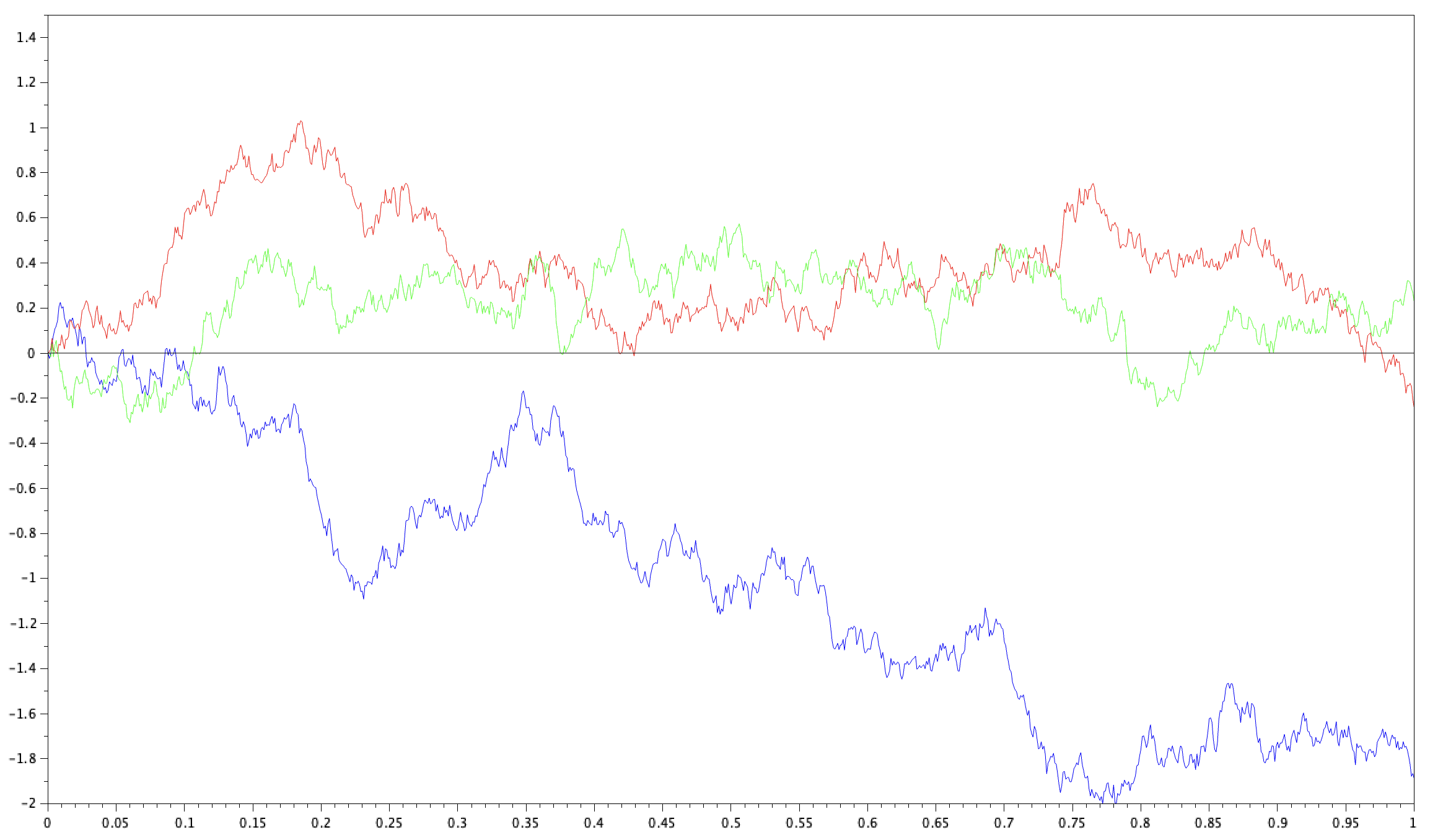}
\caption{Trajectories of the standard linear Brownian motion between  times $T_1=0$ and $T_2=1$.}
\label{fig1}
\end{figure}

\noindent The use of probabilistic methods is crucial in the proof of the above statement to overcome the impossibility to rely on Taylor expansions in the case of irregular curves. This idea was in some sense previously alluded in the context of establishing effective counting results in the case of random shifts of regular curves in~\cite{A2} or in the case of probabilistic generalisations of the Oppenheim  Conjecture in Diophantine Approximation in~\cite{AM1}.\\

\noindent In the context of the Area Heuristic, the relevant counting function is the one defined in~\eqref{defcountrand1}, namely $\mathcal{N}^{\,\flat}\!\left(\bm{\mathcal{T}};\left(\delta, Q\right)\right)$. The target set is then taken as having a size uniform in the bound imposed on the rational denominators (in the sense that the size depends only on this bound). 

%
%

\begin{thmb}[Uniform Counting]\label{coro1} Keep the notation of the First Main Theorem. Assume furthermore that $\delta~: \N\rightarrow (0, 1/4)$ is a map such that
\begin{align*}
\limsup_{Q\rightarrow \infty} \delta(Q)<\frac{1}{4}\cdotp
\end{align*} 
Then, almost surely as $Q$ tends to infinity,
\begin{align}\label{asymunifestim}
\mathcal{N}^{\,\flat}\!\left(\bm{\mathcal{T}};\left(\delta, Q\right)\right)\;=\; \frac{{4}}{3\sqrt{\pi}}\cdot\left(T_2-T_1\right)\cdot \sqrt{\delta(Q) Q^5}\cdot\left(1+o(1)\right)
\end{align}
whenever 
\begin{align}\label{lowboudelB}
\delta(Q)\cdot Q^5\gg\left(\log Q\right)^{2+\eta} \qquad \textrm{for some} \qquad \eta>0.
\end{align}
Furthermore, in the symmetric regime where 
\begin{align}\label{lowboudelBbis}
\delta(Q)\cdot Q^5\ll\left(\log Q\right)^{-2-\eta} \qquad \textrm{for some} \qquad \eta>0,
\end{align}
one has that almost surely, for all large enough integers $Q$, 
\begin{align}\label{asymunifestimbis}
\mathcal{N}^{\,\flat}\!\left(\bm{\mathcal{T}};\left(\delta, Q\right)\right)\;=\; 0. 
\end{align}
\end{thmb} 

\noindent A feature of the above result is that it imposes very little   constraints on the  range of admissible exponents $\alpha$  for the counting estimate to be valid under the assumption that either $\delta\ll Q^{-\alpha}$ or $\delta\gg Q^{-\alpha}$ (namely, \emph{any} $\alpha\neq 5$ is admissible). This stands in sharp contrast to the case of classes of regular curves where all known general results hold only when $\alpha<1$ and $\delta\gg Q^{-\alpha}$. \\

\noindent The only regime where the probabilistic  behavior of the counting function $\mathcal{N}^{\,\flat}\!\left(\bm{\mathcal{T}};\left(\delta, Q\right)\right)$ remains undetermined is when $\delta$ equals, up to logarithmic powers, $Q^{-5}$. The authors do not necessarily expect that an almost sure behavior (in the likes of~\eqref{asymunifestim} and~\eqref{asymunifestimbis}) should then be satisfied. The complete determination of the probability distribution of the random variable $\mathcal{N}^{\,\flat}\!\left(\bm{\mathcal{T}};\left(\delta, Q\right)\right)$ in this critical regime is left open  (perhaps is it Poissonian when $\delta Q^5\asymp 1$). A contribution towards this problem is contained in the auxiliary  statement below showing that, in the sense of convergence in probability, the dichotomy of capturing either no rational point or infinitely many of them in a tubular neighbourhood of the graph of  Brownian motion is determined by, respectively, the vanishing or the divergence to infinity of the sequence $\left(\delta(Q)\cdot Q^5\right)_{Q\ge 1}$.


\begin{thmaux} Keep the notation of the Second Main Theorem. Then, the following dichotomy holds~: 
\begin{itemize}
\item when $\delta(Q)\cdot Q^5\longrightarrow \infty$ as $Q\longrightarrow\infty$, 
\begin{align}\label{prob1}
\mathcal{N}^{\,\flat}\!\left(\bm{\mathcal{T}};\left(\delta, Q\right)\right)\; \overset{\mathbb{P}}{\mathrel{\scalebox{2}{$\sim$}}} \; \frac{{4}}{3\sqrt{\pi}}\cdot\left(T_2-T_1\right)\cdot \sqrt{\delta Q^5} \qquad \textrm{as}\qquad  Q\longrightarrow\infty.
\end{align}
This means that the ratio of the right and left--hand sides of the first relation tends to 1 in probability, and implies in particular that 
\begin{align*}
\mathcal{N}^{\,\flat}\!\left(\bm{\mathcal{T}};\left(\delta, Q\right)\right) \overset{\mathbb{P}}{\longrightarrow} \infty \qquad \textrm{as}\qquad  Q\rightarrow\infty.
\end{align*}
\item when $\delta(Q)\cdot Q^5\longrightarrow 0$ as $Q\longrightarrow\infty$, 
\begin{align}\label{prob2}
\mathbb{P}\left(\mathcal{N}^{\,\flat}\!\left(\bm{\mathcal{T}};\left(\delta, Q\right)\right)\ge 1\right)\; \ll \; 
\left(T_2-T_1\right)\cdot \sqrt{\delta Q^5},
\end{align}
implying in particular that 
\begin{align*}
\mathcal{N}^{\,\flat}\!\left(\bm{\mathcal{T}};\left(\delta, Q\right)\right) \overset{\mathbb{P}}{\longrightarrow} 0 \qquad \textrm{as}\qquad  Q\rightarrow\infty.\\
\end{align*}
\end{itemize}
\end{thmaux}

\noindent Up to a multiplicative absolute constant, the leading term in the Second Main Theorem is $$ \left(T_2-T_1\right)\cdot \sqrt{\frac{\delta}{Q}}\cdot Q^{3}.$$ For  the Area Heuristic to be satisfied, the factor $\left(T_2-T_1\right)\cdot \sqrt{\delta/Q}$ should match (up to multiplicative constants) the almost sure area of a $(\delta/Q)$--tubular neighbourhood of the graph of  Brownian motion. This is confirmed by the fol\-lowing statement {which, together with the Second Main Theorem, shows in particular  that the heuristic holds in a strong, extended sense   provided one stays away from the critical value of the exponent $\alpha=5$ (as introduced in~\eqref{ah})~: for any $\eta>0$, the heuristic is indeed valid when $\delta\gg Q^{-5+\eta}$ (this is the regime where the area of the tubular neighbourhood tends with the counting function $\mathcal{N}^{\,\flat}\!\left(\bm{\mathcal{T}};\left(\delta, Q\right)\right) $  to infinity) and, in the symmetric regime where $\delta\ll Q^{-5-\eta}$, the counting function $\mathcal{N}^{\,\flat}\!\left(\bm{\mathcal{T}};\left(\delta, Q\right)\right) $ vanishes at infinity in the same way as the area of the tubular neighbourhood.  In view of the discussion held when introducing the Area-Heuristic exponent (see after~\eqref{ah}), the latter case should be related to the elementarily verified property that Brownian motion captures, almost surely, no rational point~: in particular, the intrinsic contribution plays no role in the counting theory at any scale.}\\

\noindent Surprisingly, while the area of the tubular neighbourhood of a planar Brownian motion is well--studied (this is the so--called theory of Wiener Sausages  --- see~\cite{Sz}), no similar result could be found for the case of the tubular neighbourhood of the graph of the linear Brownian motion. 

\begin{thmc}[Almost sure behavior of the area of a tubular neighbourhood of the graph of Brownian motion]\label{thmvol}
Let $0\le T_1< T_2$ and $\eta>0$ be reals. Recall the definition~\eqref{count2} of the $\eta$--neighbourhood $\widehat{\mathfrak{T}}\!\left(\bm{\mathcal{T}}; \eta\right)$ of the graph of Brownian motion between times $T_1$ and $T_2$. Denote by $\widehat{\mathcal{A}}\!\left(\bm{\mathcal{T}}; \eta\right)$ the area of $\widehat{\mathfrak{T}}\!\left(\bm{\mathcal{T}};\eta\right)$, whenever it exists. \\

\noindent Then, $ \widehat{\mathcal{A}}\!\left(\bm{\mathcal{T}};\eta\right)$ is a well--defined random variable over the probability space $\left(\Omega, \mathcal{F}, \mathbb{P}\right)$ where Brownian motion is defined. Furthermore, the expectation of $\widehat{\mathcal{A}}\!\left(\bm{\mathcal{T}};\eta\right)$ admits the exact expression 
\begin{equation}\label{expvol}
\E\left[\widehat{\mathcal{A}}\!\left(\bm{\mathcal{T}};\eta\right)\right]\;=\; \frac{4}{\sqrt{\pi}}\cdot\left(T_2-T_1\right)\cdot \sqrt{\eta}\;+\; 2{\cdot\left(T_2-T_1\right)}\cdot \eta\;-\; \frac{8\left(\sqrt{2}-1\right)}{3\sqrt{\pi}}\cdot \eta^{3/2}
\end{equation} 
provided that $\eta<\left(T_2-T_1\right)/2$. Also, almost surely as $\eta\rightarrow 0^+$, the asymptotic estimate 
\begin{equation*}
\widehat{\mathcal{A}}\!\left(\bm{\mathcal{T}};\eta\right)\;=\; \left(\frac{4}{\sqrt{\pi}}+o(1)\right)\cdot \left(T_2-T_1\right) \cdot \sqrt{\eta}
\end{equation*}
is satisfied.
\end{thmc}

\noindent An interesting feature of the above estimates is that the leading terms do not contain any logarithmic corrections. This stands in contrast to  the asymptotic behavior of the expectation  of the area of a Wiener sausage in the plane which, as proved in~\cite[p.79]{Sz}, involves such an additional logarithmic correction.\\

\begin{rem} It is immediate that the conclusions of the First and Second Main Theorems actually hold for any almost surely continuous stochastic process $Y=\left(Y_t\right)_{T_1\le t\le T_2}$ whose law is absolutely continuous with respect to Brownian motion over the space of continuous functions over the interval $\left[T_1, T_2\right]$ (here, $0<T_1<T_2$). A natural way to construct such processes is to apply Girsanov's Theorem such as stated, e.g.,  in~\cite[Theorem~6.4]{B}. For the sake of illustration, this statement  implies that the  two theorems under consideration also hold for the following processes~:
\begin{itemize}
\item $Y$ is a \emph{Brownian motion with constant drift} $Y_t=B_t+\mu t$, where $\mu\in\R$, and more generally when $Y$ is a perturbed Brownian motion of the form $$Y_t\;=\; B_t+ \int_{T_1}^tu_s\cdot \textrm{d}s.$$ Here, $\left(u_s\right)_{T_1\le s\le T_2}$ can be taken as  an adapted process meeting the condition $$\mathbb{E}\left[\exp\left(\frac{1}{2}\cdot \int_{T_1}^{T_2}u_s^2\cdot\textrm{d}s\right)\right]<\infty.$$ 
\item $Y$ satisfies the stochastic differential equation $\textrm{d}Y_t= b\!\left(t, Y_t\right)\cdot\textrm{d}t+\textrm{d}B_t$, where the drift $b$ is, e.g., either bounded or satisfies the so-called Novikov condition. A particular case of this situation is the \emph{Ornstein–Uhlenbeck process} obtained upon setting $b(t,x)=-\theta x$, where $\theta>0$.
\item $Y$ is a \emph{Brownian bridge} away from the pinning time, say $T_2$. In the notation of the previous example, the drift is, in this case,  given by $b(t,x)=-x/(T_2-t)$. The law of this process is absolutely continuous with respect to that of Brownian motion on every interval of the form $\left[T_1, T_2-\varepsilon\right]$, where $\varepsilon>0$.
\end{itemize}
Other examples include Bessel processes away from the origin or a Brownian motion conditioned by sufficiently ``nice'' events. Details are left to the interested reader. 
\end{rem}

\subsection{Rational {Points} and Multifractal Analysis}

\noindent The appearance of  square-root factors in the  Second Main Theorem and in the Tubular Neighbourhood Area Theorem is reminiscent {of} the well--known fact that Brownian motion is almost surely $\gamma$-Hölderian everywhere  when $\gamma<1/2$, and that this threshold is optimal. 
More precisely, the Hölder--regularity of  Brownian motion is captured in a strong form in  Lévy's famous Mo\-du\-lus of Continuity Theorem~\cite[Theorem~1.14 \& Remark~1.18]{MP}. It claims that, almost surely and after normalisation to work over the interval $\left[T_1, T_2\right]=[0,1]$, 
\begin{equation}\label{levy}
\lim_{\eta\rightarrow 0^+}\;\sup_{\underset{\left|t-s\right|\le \eta}{0\le t, s\le 1}}\;\frac{\left|B(s)-B(t)\right|}{\sqrt{2\eta\log(1/\eta)}}\;=\; 1.
\end{equation}

\noindent Consider now a continuous function $f$ defined over a bounded interval $I$. Given $x\in I$ and $\eta>0$, define its local modulus of continuity at $x$ by
\begin{equation}\label{localmodulus}
\omega_f(x,\eta)\;=\;
\sup_{\substack{y\in I\\0<|x-y|\leq\eta}}
|f(x)-f(y)|.
\end{equation}
The local Hölder exponent of $f$ at $x$ is defined as
\begin{equation}\label{hoelderexp}
h_f(x)\;=\;
\liminf_{\eta\rightarrow 0^+}\;
\frac{\log\omega_f(x,\eta)}{\log\eta}
\end{equation}
with the convention that if the function $f$ is constant in a neighbourhood of $x$, then $h_f(x)=\infty$. \\

\noindent The study of the \emph{multifractal analysis} of the function $f$, as initiated by Frisch and Parisi~\cite{FP}, consists in determining its so--called \emph{singularity spectrum}; that is, the va\-lues taken by the map
$$
\gamma\geq 0\;\mapsto\;\dim_H L_f(\gamma),
\qquad \textrm{where}\qquad
L_f(\gamma)\;=\;\left\{x\in I\;:\;h_f(x)=\gamma\right\}.
$$
Here, the set $L_f(\gamma)$ is referred to as the \emph{iso--Hölder set} of level $\gamma$ of the function $f$, and $\dim_H$ denotes the Hausdorff dimension. The function $f$ is called \emph{monofractal with Hölder exponent $\gamma$} if
\begin{equation}\label{monofractal}
h_f(x)\;=\;\gamma
\qquad\textrm{for every}\qquad x\in I.
\end{equation}
In this terminology, Brownian motion is almost surely monofractal with Hölder exponent $1/2$. 
In view of this observation, the Second Main Theorem suggests that the local regularity of a function plays a pivotal role in determining the number of rational points with bounded denominators lying in a shrinking tubular neighbourhood of its graph. That being said, the Hölder exponent~\eqref{hoelderexp} only provides local information whereas  for the Area Heuristic to hold, one needs a stronger uniform control on the variations of the function under consideration.  \\

\noindent To capture this additional information, let $\mathfrak{D}(I)$ denote the collection of all those intervals obtained by repeated  bisections of the interval $I$, with the requirement that $I\in \mathfrak{D}(I)$. Given $J\in\mathfrak{D}(I)$ and an integer $k\geq 1$, let also
$$
J_{k}^{(1)},\;\dots,\; J_{k}^{(2^k)}
$$
be the partition of $J$   into $2^k$ subintervals of equal length  (up to boundary points). Set
\begin{equation}\label{averagedosc}
S_{f}\!\left(J,k\right)\;=\;
\frac{1}{2^k}\sum_{j=1}^{2^k}
\sup_{x,y\in J_{k}^{(j)}}|f(x)-f(y)|.
\end{equation}
The function $f$ is said to be \emph{uniformly oscillating with exponent $\gamma\in(0,1]$} if, for every $J\in\mathfrak{D}(I)$,
\begin{equation}\label{unifosc}
S_{f}\!\left(J,k\right)\;\asymp\;
\left(\frac{|J|}{2^k}\right)^\gamma
\qquad\textrm{as}\qquad k\longrightarrow\infty,
\end{equation}
where $|J|$ denotes the length of the interval $J$ and where the implicit constants may depend on $f$ and $J$, but not on $k$~(\textsuperscript{6}\let\thefootnote\relax\footnotetext{\textsuperscript{6}It can be shown that, whenever a monofractal function is uniformly oscillating, its uniformly oscillating exponent~\eqref{unifosc} coincides with its Hölder exponent~\eqref{monofractal}. Since the proof of this claim is nontrivial and lies beyond the scope of this paper, the equality between these two exponents   is  explicitly imposed in the statement of Conjecture~\ref{conj} below.. \\})

\noindent Condition~\eqref{unifosc} determines the order of magnitude of the area of the tubular neighbourhood up to multiplicative constants depending on $f$ and on $I$. Indeed, if $\eta>0$ is given and if the integer $k$ is chosen so that
\begin{equation}\label{dyad}
\eta\;\leq\;\frac{|I|}{2^k}\;<\;2\eta,
\end{equation}
then an elementary packing and covering argument\textsuperscript{7}\let\thefootnote\relax\footnotetext{\textsuperscript{7}This argument is sketched for completeness~: by continuity, the area of the $\eta$-neighbourhood of the graph of the restriction of $f$ to each subinterval  $I_{k}^{(\ell)}$ is of the order of $\eta$ times the maximum in~\eqref{areaavosc}. To deduce from this observation the  upper bound in the first relation  in~\eqref{areaavosc}, it suffices to cover the part of the graph above each subinterval $I_{k}^{(\ell)}$ by a rectangle enlarged by $\eta$ in both directions (to absorb the boundary effects). As far as  the lower bound is concerned, note first that the set
$
\left\{(x,y)\in\R^2\;:\;x\in I,\quad |y-f(x)|\leq\eta\right\}
$
is contained in the tubular neighbourood $\mathfrak{T}_f(\eta)$. Then, choose a residue class of indices $\ell$ modulo 4 for which the sum of the oscillations over the corresponding intervals $I_{k}^{(\ell)}$ is maximal. This sum is at least one quarter of the total sum, while the $\eta$--neighbourhoods of the corresponding graph pieces are disjoint in the horizontal direction. The relation~\eqref{areaavosc} follows.} yields that 
\begin{equation}\label{areaavosc}
\mathcal{A}_f(\eta)\;\asymp\;
\eta\cdot \sum_{\ell=1}^{2^k}
\max\left\{\eta, \;\sup_{x,y\in I_{k}^{(\ell)}}|f(x)-f(y)|\right\}
\;\underset{\eqref{averagedosc} \& \eqref{dyad} }{\asymp}\;
\eta+ S_{f}\!\left(I,k\right)
\end{equation}
in such a way that 
\begin{equation}\label{areaavoscconsequence}
\mathcal{A}_f(\eta)\;\underset{\eqref{unifosc}}{\asymp}\;\eta^\gamma
\qquad\textrm{as}\qquad\eta\longrightarrow 0^+.
\end{equation}

\noindent This leads one to the following conjecture, which builds a natural bridge between Number Theory and Multifractal Analysis~:

\begin{conj}[Oscillation Principle]\label{conj}
Let $f:I\rightarrow\R$ be continuous and uniformly oscillating with exponent $\gamma\in(0,1)$. Assume furthermore that 
the function $f$ is monofractal with Hölder exponent $\gamma$. 
Then,  there exists an exponent $\alpha>0$ such that
\begin{equation}\label{ahb}
\mathcal{N}_f\left(\delta,Q\right)\;\asymp\;
\left(\frac{\delta}{Q}\right)^\gamma\cdot Q^3
\qquad\textrm{whenever}\qquad
\delta\gg Q^{-\alpha}.
\end{equation}
\end{conj}

\noindent This Oscillation Principle  is intended to describe the expected behavior of a large class of irregular curves by conjecturally identifying some analytic features determi\-ning the asymptotic order of magnitude of the counting function under consideration. From~\eqref{areaavoscconsequence}, its assumptions already yield the asymptotic order of magnitude for the area of the tubular neighbourhood. For the Area Heuristic to hold, it thus remains to show that in some suitable range of the measure of thickness $\delta$,  rational points are sufficiently well--distributed in such tubular neighbourhood. \\

\noindent The ``regular case'' $\gamma=1$ is excluded from the assumptions of Conjecture~\ref{conj} as it is not hard to construct in this case a function which satisfies both the monofractality and uniform oscillation assumptions, which is not affine on any nondegenerate interval, and which nevertheless does not satisfy the estimate~\eqref{ahb}. 
Additional assumptions are thus needed to force the conclusion. As mentioned in the above Section~\ref{subsec2},   a result by Huang~\cite{H} provides such additional, non--degeneracy  conditions and establishes a form of the Oscillation Principle for a large class of continuously differentiable curves covering the case $\gamma=1$ --- an admissible exponent is then $\alpha=2/3$. \\

\noindent From the Second Main Theorem, the Oscillation Principle also holds almost surely  in the case of Brownian motion when $\alpha<5$. To see this, it suffices to show that the assumptions of Conjecture~\ref{conj}   are met. This is indeed the case since,  as noted above, the monofractal property holds with the exponent $\gamma=1/2$. Furthermore,  the Tubular Neighbourhood Area Theorem applied to the restriction of Brownian motion to each subinterval $J\in\mathfrak{D}(I)$, and combined with the area estimate~\eqref{areaavosc} yields the uniform oscillating property~\eqref{unifosc}. Since $\mathfrak{D}(I)$ is countable, these relations hold simultaneously for all $J\in\mathfrak{D}(I)$ on an event of full probability, whence the claim.  \\

\noindent In the case of continuously differentiable functions, the proof of the above--mentioned result by Huang~\cite{H} (under assumptions stronger than those in the Oscillation Principle) relies extensively on Taylor expansions. In the Brownian case however, the proof in Section~\ref{sec3} below is purely probabilistic. It thus remains wide open to establish Conjecture~\ref{conj} for  \emph{deterministic} monofractal functions satisfying the uniform oscillation condition~: neither the probabilistic  method nor the one relying on Taylor expansions is then available. Canonical examples of interest  are given by Weierstrass' nowhere differentiable  functions, whose regularity is studied in~\cite{Se} and  in the references within. \\

\noindent It would also be interesting to formulate an analogue of Conjecture~\ref{conj} in the more challenging case of multifractal maps. A most famous example is Riemann's non--differentiable function, the singularity spectrum of which is fully determined by Jaffard in his seminal work~\cite{Ja}. Such a statement would  need to capture  the distribution of the oscillations  of the multifractal map  both in localisation and in  scale.


\paragraph{Organisation of the paper.} The proofs of the  statements rely on the Second Moment Method in Probability Theory. More precisely, this method is combined with a slicing technique introduced in Section~\ref{sec2} to  first establish therein the Tubular Neighbourhood Area Theorem. The First  and Second Main Theorems, and also the Auxiliary  Theorem are then derived  in Section~\ref{sec3} from moment estimates proved in Sections~\ref{asymdiag}, \ref{asymoffdiag} and~\ref{secfin}. They rely on the sharp estimations of the probability that  Brownian motion hits a given rectangle proved in the earlier Section~\ref{probsquare}. \\

\noindent Each of the seven sections is introduced with a concise paragraph explaining  its articulation with the other parts of the paper. These introductory paragraphs together with the table of contents are meant to facilitate further the navigation within the various proofs and sections.

\paragraph{Acknowledgments.} The first--named author would like to thank Xiaolin Zeng (Université de Strasbourg) for various discussions pertaining to this paper. His work was supported by the French \emph{Agence Nationale de la Recherche} through grant ANR-25-CE40-1961-01. The other two co-authors were supported by EPSRC grant UKRI 2768. \color{black}

\paragraph{Notation and standard properties of Brownian motion used throughout.} The notation is fairly standard, and the basic properties of  Brownian motion stated here can be found, e.g., in~\cite{MP}. \\

\noindent Given two random variables $X$ and $Y$, the relation $X\overset{\mathcal{L}}{=}Y$ means that they share the same law. In the case where $X$ follows the normal law with expectation $\mu\in\R$ and variance $\sigma^2>0$, this is denoted by $X\overset{\mathcal{L}}{\thicksim} \mathcal{N}\left(\mu, \sigma^2\right)$. The independence of $X$ and $Y$ is represented by $X\amalg Y$. Only the standard, linear Brownian motion is considered throughout. It is defined over a probability space $\left(\Omega, \mathcal{F}, \mathbb{P}\right)$, and the expectation and variance of a random variable $X$   defined over this space are denoted by $\E\left[X\right]$ and $\V\left[X\right]$, respectively. \\

\noindent The maximum and the minimum of Brownian motion up to time $t\ge 0$ are the random variables 
\begin{equation*}\label{minmax}
M(t)\;=\;\max_{0\leq s\leq t}\;B(s)\qquad \quad\textrm{and}\quad\qquad m(t)\;=\;\min_{0\leq s\leq t}\;B(s)\;=\; -\max_{0\leq s\leq t}\;(-B(s)).
\end{equation*}
They meet the so-called \emph{Law of Maximum} claiming that for all  $L>0$, 
\begin{equation} \label{Mt}
\bP\left(M(t)>L\right)\;=\;2\cdot \bP\left(B(t)>L\right) \quad \textrm{and}\quad \bP\left(m(t)<-L\right)=2\cdot \bP\left(B(t)<-L\right).
\end{equation}

\noindent The following standard properties of the process $\left\{B(t)\right\}_{t\ge 0}$, some of them already mentioned before, shall be used throughout and referred to with the labels introduced below~: 
\begin{itemize}[leftmargin=20mm]
\item[$\bm{(P1)}$] \emph{Symmetry~:}  $B(t)\overset{\mathcal{L}}{=}-B(t)$  for all $t\ge 0$;
\item[$\bm{(P2)}$] \emph{Scaling invariance~:}  $\sqrt{c}\cdot B(t)\overset{\mathcal{L}}{=}B(ct)$ for all $t\ge 0$ and $c>0$;
\item[$\bm{(P3)}$] \emph{Law of increments~:}  $B(t)-B(s) \overset{\mathcal{L}}{{=}} B(t-s) \overset{\mathcal{L}}{\thicksim}  \mathcal{N}\left(0, t-s\right)$ for all $0\le s<t$;
\item[$\bm{(P4)}$] \emph{Independence of increments~:} $\left(B(t)-B(s)\right)\amalg B(s)$  for all $t>s\ge  0$.\\
\end{itemize}

\noindent Finally, given  $t>0$, the density of the random variable $B(t)$ is  the Gaussian function
\begin{equation} \label{def_p}
p_t~: x\in\R\;\mapsto\; p_t(x)\;=\; \frac{1}{\sqrt{2\pi t}}\cdot \exp\left(-\frac{x^2}{2t}\right).\\
\end{equation}


\section{Area of the Tubular Neighbourhood of a Trajectory of Brownian Motion}\label{sec2}

\noindent The goal of this section is to establish the Tubular Neighbourhood Area Theorem. Its content is independent of the remaining part of the paper.\\

\noindent First, given $\eta>0$, the elegant and simple argument outlined in~\cite[Remark~2.5]{CPS} shows that the area $\widehat{\mathcal{A}}\!\left(\bm{\mathcal{T}};\eta\right)$ of the tubular neighbourhood $\widehat{\mathfrak{T}}\!\left(\bm{\mathcal{T}}; \eta\right)$ is a random variable. \\

\noindent Given $t>0$, define the \emph{spread} of  Brownian motion up to time $t$ as the random variable 
\begin{equation*}
R(t)\;=\; M(t)-m(t). 
\end{equation*}
\noindent It is known from~\cite{Fe} that it admits the density function
\begin{align}
\rho_{t}(x)\; =\; \frac{8}{\sqrt{2\pi t}}\cdot \left(\sum_{k=1}^{\infty}\left(-1\right)^{k-1}\cdot k^2\cdot \exp\left(-\frac{\left(kx\right)^2}{2t}\right)\right)
\;=\; \sqrt{\frac{2}{\pi}}\cdot \frac{1}{x}\cdot F'\left(\frac{x}{2\sqrt{t}}\right),\label{densityR}
\end{align}
where, given $u>0$, 
\begin{align}\label{densityF}
F~: u\;\mapsto\; \frac{\sqrt{2\pi}}{u}\cdot\left(\sum_{k=1}^{\infty}\exp\left(-\frac{\left(2k-1\right)\cdot\pi^2}{8u^2}\right)\right)
\end{align}
is a cumulative distribution function.\\

\noindent Recall that $T_2>T_1\ge 0$ and $\eta>0$ are fixed reals, and set $\bm{\mathcal{T}}=\left(T_1, T_2\right)$. Define the \emph{localised maximum and minimum at time $t\in \left[T_1, T_2\right]$ at the scale $\eta>0$} as the random variables 
\begin{equation}\label{minmaxloc}
M(\bm{\mathcal{T}};\left(t, \eta\right))\;=\;\max_{\underset{T_1\le s\le T_2}{t-\eta\leq s\leq t+\eta}}\;B(s)\qquad \quad\textrm{and}\quad\qquad m(\bm{\mathcal{T}};\left(t, \eta\right))\;=\;\min_{\underset{T_1\le s\le T_2}{t-\eta\leq s\leq t+\eta}}\;B(s),
\end{equation}
respectively. Correspondingly, the \emph{localised range at time $t\in \left[T_1, T_2\right]$ at the scale $\eta>0$} is the interval 
\begin{equation}\label{rangeloc}
\mathcal{I}(\bm{\mathcal{T}};\left(t, \eta\right))\;=\; \left[m(\bm{\mathcal{T}};\left(t, \eta\right))-\eta, \; M(\bm{\mathcal{T}};\left(t, \eta\right))+\eta\right],
\end{equation}
and its length is denoted by $\left|\mathcal{I}(\bm{\mathcal{T}};\left(t, \eta\right))\right|$. From Fubini's Theorem, one has almost surely that 
\begin{equation}\label{volfub}
\widehat{\mathcal{A}}\!\left(\bm{\mathcal{T}}; \eta\right)\;=\; \bigintssss_{T_1}^{T_2} \left|\mathcal{I}(\bm{\mathcal{T}};\left(t, \eta\right))\right|\cdot\textrm{d}t.
\end{equation}
\noindent This decomposition is the key step in the proof of the Tubular Neighbourhood Area Theorem~: it reduces the volume estimate to the determination of the trace left by a ball of radius $\eta$ (with respect to the sup--norm) on a given vertical line at time $t$ when the center of the ball sweeps the graph of Brownian motion. This slicing technique allows for  the estimation of the first and second moments of the random variable $\widehat{\mathcal{A}}\!\left(\bm{\mathcal{T}};\eta\right)$, from which the statement of the Tubular Neighbourhood Area Theorem follows.

\begin{prop}[First moment estimate]\label{firstmom}
With the above notation, assume that $\eta<\left(T_2-T_1\right)/2$. Then, the expectation estimate~\eqref{expvol} holds; that is, 
\begin{equation*}
\E\left[\widehat{\mathcal{A}}\!\left(\bm{\mathcal{T}};\eta\right)\right]\;=\; \frac{4}{\sqrt{\pi}}\cdot\left(T_2-T_1\right)\cdot \sqrt{\eta}\;+\; 2\eta{\cdot\left(T_2-T_1\right)}\;-\; \frac{8\left(\sqrt{2}-1\right)}{3\sqrt{\pi}}\cdot \eta^{3/2}.
\end{equation*} 
\end{prop} 

\begin{proof}
From the decomposition~\eqref{volfub}, Fubini's Theorem yields that 
\begin{equation}\label{fubdem}
\E\left[\widehat{\mathcal{A}}\!\left(\bm{\mathcal{T}};\eta\right)\right]\;=\; \bigintssss_{T_1}^{T_2} \E\left[\,\left|\mathcal{I}(\bm{\mathcal{T}};\left(t, \eta\right))\right|\,\right]\cdot\textrm{d}t, 
\end{equation} 
where upon relying on the definition of the interval $\mathcal{I}(\bm{\mathcal{T}}; \left(t, \eta\right))$ in~\eqref{rangeloc} and on the property of symmetry (P1), the integrand can be expressed as 
\begin{align}\label{expint}
 \E\left[\,\left|\mathcal{I}(\bm{\mathcal{T}};\left(t, \eta\right))\right|\,\right]\;&=\; 2\eta+  \E\left[M(\bm{\mathcal{T}};\left(t, \eta\right))\right] - \E\left[m(\bm{\mathcal{T}};\left(t, \eta\right))\right] \; =\; 2\cdot \left(\eta+ \E\left[M(\bm{\mathcal{T}};\left(t, \eta\right))\right]  \right).
\end{align}
Set for convenience 
\begin{equation}\label{notatiem}
t^-= t^-\!\left(\bm{\mathcal{T}}; \left(t,\eta\right)\right)\;:=\; \max\left\{t-\eta, T_1\right\}\qquad \textrm{and}\qquad t^+= t^+\!\left(\bm{\mathcal{T}}; \left(t,\eta\right)\right):= \min\left\{t+\eta, T_2\right\}. 
\end{equation}
Then, 
\allowbreak{
\begin{align*}
\E\left[M(\bm{\mathcal{T}}; \left(t, \eta\right))\right] \;&=\; \E\left[\max_{t^-\le s\le t^+} B(s)\right]\\
& =\;  \E\left[\E\left[\max_{t^-\le s\le t^+} B(s)\middle|  B_{t^-}\right]\right] \\
& =\;   \E\left[B\left(t^-\right)\right]\;+\; \E\left[\E\left[\max_{0\le u\le t^+-t^-} \left(B\left(t^-+u\right)-B\left(t^-\right)\right)\middle|  B_{t^-}\right]\right].
\end{align*}
}The first term in this sum vanishes from Property~(P3). 
From the property of independence of increments~(P4), this yields that 
\allowbreak{
\begin{align*}
\E\left[M(\bm{\mathcal{T}}; \left(t, \eta\right))\right] \;&=\;   \E\left[\max_{0\le u\le t^+-t^-} \left(B\left(t^-+u\right)-B\left(t^-\right)\right)\right]\\
&\underset{(P2)\&(P4)}{=}\; \sqrt{t^+-t^-}\cdot \E\left[M(1)\right]\\
&=\; \sqrt{t^+-t^-}\cdot\int_0^\infty \bP\left(M(1)\ge y\right)\cdot\textrm{d}y\qquad \qquad \qquad\;\, \textrm{(by Fubini)}\\
&\underset{\eqref{Mt}}{=}\;  2\cdot \sqrt{t^+-t^-}\cdot\int_0^\infty \bP\left(B(1)> y\right)\cdot\textrm{d}y\\
&\underset{\eqref{def_p}}{=}\; 2\cdot \sqrt{t^+-t^-}\cdot\frac{1}{\sqrt{2\pi}}\cdot \underbrace{\int_0^\infty x\cdot \exp\left(-\frac{x^2}{2}\right)\cdot\textrm{d}x}_{=1}\quad \textrm{(by Fubini)}\\
&=\; \sqrt{\frac{2\cdot\left(t^+-t^-\right)}{\pi}}\cdot
\end{align*}
}Recall that it is assumed in the Tubular Neighbourhood Area Theorem  that $T_2-T_1>2\eta$. As a consequence, reverting to the full notation introduced in~\eqref{notatiem}, it follows from the decomposition~\eqref{fubdem} and from the second equation in~\eqref{expint} that 
\allowbreak{
\begin{align*}
\E\left[\widehat{\mathcal{A}}\!\left(\bm{\mathcal{T}};\eta\right)\right]\;&=\; 2\cdot\left(\eta{\cdot\left(T_2-T_1\right)}+\sqrt{\frac{2}{\pi}}\cdot \int_{T_1}^{T_2}\sqrt{ t^+\!\left(\bm{\mathcal{T}}; \left(t,\eta\right)\right)- t^-\!\left(\bm{\mathcal{T}}; \left(t,\eta\right)\right)}\cdot \textrm{d}t\right)\\
&=\; 2\cdot\left(\eta{\cdot\left(T_2-T_1\right)}+\sqrt{\frac{2}{\pi}}\cdot \left(\int_{T_1}^{T_1+\eta}\sqrt{t+\eta-T_1}\cdot\textrm{d}t\right)+\int_{T_1+\eta}^{T_2-\eta}\sqrt{2\eta}\cdot\textrm{d}t\right.\\
&\left.\qquad\qquad\qquad\qquad\qquad\qquad\qquad\quad\qquad\,\;\quad +\int_{T_2-\eta}^{T_2}\sqrt{T_2-t+\eta}\cdot\textrm{d}t\right)\\
&=\; 2\cdot\left(\eta{\cdot\left(T_2-T_1\right)}+\sqrt{\frac{2}{\pi}}\cdot \left( 2\cdot\int_{\eta}^{2\eta}\sqrt{v}\cdot\textrm{d}v+\int_{T_1+\eta}^{T_2-\eta}\sqrt{2\eta}\cdot\textrm{d}v\right)\right).
\end{align*}
}A direct calculation of the integrals yields the formula stated in the proposition. 
\end{proof}

\begin{lem}[Second moment estimate]\label{secmom}
Keep the above notation. Then, 
\begin{equation*}
\E\left[\left(\widehat{\mathcal{A}}\!\left(\bm{\mathcal{T}};\eta\right)\right)^2\right]-\left(\E\left[\widehat{\mathcal{A}}\left(\bm{\mathcal{T}}; \eta\right)\right]\right)^2\;\ll\; \eta^2\cdot\left(T_2-T_1\right),
\end{equation*} 
where the implied constant is absolute.
\end{lem} 

\noindent The proof is more elaborate than for the first moment. 

\begin{proof}
Relying on the decomposition~\eqref{volfub},  Fubini's Theorem implies that 
\begin{align*}
\E\left[\left(\widehat{\mathcal{A}}\left(\bm{\mathcal{T}};\eta\right)\right)^2\right]\;&=\; \bigintssss_{T_1}^{T_2} \bigintssss_{T_1}^{T_2} \E\left[\, \left|\mathcal{I}(\bm{\mathcal{T}}; \left(t, \eta\right))\right|\cdot \left|\mathcal{I}(\bm{\mathcal{T}}; \left(s, \eta\right))\right|\,  \right]\cdot\textrm{d}t\cdot\textrm{d}s\\
&=\; 2\cdot\left(\bigintssss_{T_1\le s\le t\le T_2} \E\left[\, \left|\mathcal{I}(\bm{\mathcal{T}}; \left(t, \eta\right))\right|\cdot \left|\mathcal{I}(\bm{\mathcal{T}}; \left(s, \eta\right))\right| \, \right]\cdot \textrm{d}t\cdot\textrm{d}s\right),
\end{align*} 
where the second equation follows by symmetry. As in the previous proof, let 
\begin{equation}\label{notatiembis}
t^-= t^-\!\left(\bm{\mathcal{T}}; \left(t,\eta\right)\right):= \max\left\{t-\eta, T_1\right\}\qquad \textrm{and}\qquad t^+= t^+\!\left(\bm{\mathcal{T}}; \left(t,\eta\right)\right):= \min\left\{t+\eta, T_2\right\}
\end{equation}
and, similarly,
\begin{equation*}
s^-= s^-\!\left(\bm{\mathcal{T}}; \left(s,\eta\right)\right):= \max\left\{s-\eta, T_1\right\}\qquad \textrm{and}\qquad s^+= s^+\!\left(\bm{\mathcal{T}}; \left(s,\eta\right)\right)\;:=\; \min\left\{s+\eta, T_2\right\}. 
\end{equation*}
Decompose  also the domain of integration into the following two subsets~: 
\begin{equation*}
\Sigma^{(1)}\!\left(\bm{\mathcal{T}}; \eta\right)\;=\; \left\{T_1\le s\le t\le T_2\;:\; s^-\le s^+\le t^-\le t^+\right\}
\end{equation*}
and
\begin{equation}\label{sigma2}
\Sigma^{(2)}\!\left(\bm{\mathcal{T}};\eta\right)\;=\; \left\{T_1\le s\le t\le T_2\;:\; s^-\le t^-\le s^+\le t^+\right\}.
\end{equation}
Thus, 
\begin{align}\label{decompint}
\E\left[\left(\widehat{\mathcal{A}}\!\left(\bm{\mathcal{T}}; \eta\right)\right)^2\right]\;&=\; 2\cdot\left(\bigintssss_{\Sigma^{(1)}\!\left(\bm{\mathcal{T}};\eta\right)}+\bigintssss_{ \Sigma^{(2)}\!\left(\bm{\mathcal{T}};\eta\right)} \right)\left(\E\left[\, \left|\mathcal{I}(\bm{\mathcal{T}}; \left(t, \eta\right))\right|\cdot \left|\mathcal{I}(\bm{\mathcal{T}}; \left(s, \eta\right))\right| \, \right]\cdot \textrm{d}t\cdot\textrm{d}s\right).\\ \nonumber
\end{align} 

\noindent Consider first the integral over the set $\Sigma^{(1)}\!\left(\bm{\mathcal{T}}; \eta\right)$. Fixing $(s,t)\in \Sigma^{(1)}\!\left(\bm{\mathcal{T}};\eta\right)$,
\begin{equation}\label{expcondi}
\E\left[ \,\left|\mathcal{I}(\bm{\mathcal{T}}; \left(t, \eta\right))\right|\cdot \left|\mathcal{I}(\bm{\mathcal{T}}; \left(s, \eta\right))\right| \, \right]\;=\; \E\left[ \E\left[\,\left|\mathcal{I}(\bm{\mathcal{T}}; \left(t, \eta\right))\right|\cdot \left|\mathcal{I}_{\bm{T}}(s, \eta)\right| \;\middle|\;  B\left(t^-\right) \right] \right],
\end{equation}
where 
\begin{align*}
\left|\mathcal{I}(\bm{\mathcal{T}}; \left(t, \eta\right))\right|\;&\underset{\eqref{rangeloc}}{=}\; 2\eta+M(\bm{\mathcal{T}}; \left(t, \eta\right))-m(\bm{\mathcal{T}}; \left(t, \eta\right))\\
&\underset{\eqref{minmaxloc}}{=} \; 2\eta+ \max_{\underset{T_1\le s\le T_2}{t-\eta\leq s\leq t+\eta}}\;\left(B(s)-B\left(t^-\right)\right)-\min_{\underset{T_1\le s\le T_2}{t-\eta\leq s\leq t+\eta}}\;\left(B(s)-B\left(t^-\right)\right).
\end{align*}
From the property of independence of increments (P4), the above maximum and mi\-ni\-mum, and therefore the random variable $\left|\mathcal{I}_{\bm{T}}(t, \eta)\right|$ itself,  are independent from  $B\left(t^-\right)$. As a consequence, equation~\eqref{expcondi} reads 
\begin{equation*}
\E\left[\, \left|\mathcal{I}(\bm{\mathcal{T}}; \left(t, \eta\right))\right|\cdot \left|\mathcal{I}(\bm{\mathcal{T}}; \left(s, \eta\right))\right| \, \right]\;=\; \E\left[\, \left|\mathcal{I}(\bm{\mathcal{T}}; \left(t, \eta\right))\right|\cdot  \E\left[ \, \left|\mathcal{I}(\bm{\mathcal{T}}; \left(s, \eta\right))\right| \;\middle|\;  B\left(t^-\right) \right] \right].
\end{equation*}
Here, $ \E\left[ \,\left|\mathcal{I}(\bm{\mathcal{T}}; \left(s, \eta\right))\right| \;\middle|\;  B\left(t^-\right) \right] $ lies in the sigma--algebra generated by $\left\{B(s)\right\}_{0\le s\le t^-}$ whereas,  again from property (P4), the random variable $\left|\mathcal{I}(\bm{\mathcal{T}}; \left(t, \eta\right))\right|$ is independent from it. This yields that 
\begin{align*}
\E\left[ \,\left|\mathcal{I}(\bm{\mathcal{T}}; \left(t, \eta\right))\right|\cdot \left|\mathcal{I}(\bm{\mathcal{T}}; \left(s, \eta\right))\right| \, \right]\;&=\; \E\left[\,\left|\mathcal{I}(\bm{\mathcal{T}}; \left(t, \eta\right))\right|\,\right]\cdot \E\left[ \E\left[\, \left|\mathcal{I}\left(\bm{\mathcal{T}}; \left(s, \eta\right)\right)\right| \;\middle|\;  B\left(t^-\right) \right] \right]\\
&=\; \E\left[\,\left|\mathcal{I}(\bm{\mathcal{T}}; \left(t, \eta\right))\right|\,\right]\cdot \E\left[ \,\left|\mathcal{I}(\bm{\mathcal{T}}; \left(s, \eta\right))\right| \,\right] .
\end{align*}
The integral over $\Sigma^{(1)}\!\left(\bm{\mathcal{T}}; \eta\right)$ thus becomes 
\begin{align}
& 2\cdot\left(\bigintssss_{\Sigma^{(1)}\!\left(\bm{\mathcal{T}}; \eta\right)}\E\left[\, \left|\mathcal{I}(\bm{\mathcal{T}}; \left(t, \eta\right))\right|\cdot \left|\mathcal{I}(\bm{\mathcal{T}}; \left(s, \eta\right))\right| \, \right]\cdot \textrm{d}t\cdot\textrm{d}s\right)\nonumber\\
& \qquad\qquad\qquad\qquad=\; 2\cdot\left(\bigintssss_{\Sigma_{\bm{T}}^{(1)}\left(\eta\right)} \E\left[\,\left|\mathcal{I}(\bm{\mathcal{T}}; \left(t, \eta\right))\right|\,\right]\cdot \E\left[\, \left|\mathcal{I}(\bm{\mathcal{T}}; \left(s, \eta\right))\right| \,\right] \cdot \textrm{d}t\cdot\textrm{d}s\right)\nonumber\\
&\qquad\qquad\qquad\qquad\le\; \bigintssss_{\left[T_1,T_2\right]^2} \E\left[\,\left|\mathcal{I}(\bm{\mathcal{T}}; \left(t, \eta\right))\right|\,\right]\cdot \E\left[ \,\left|\mathcal{I}(\bm{\mathcal{T}}; \left(s, \eta\right))\right| \,\right] \cdot \textrm{d}t\cdot\textrm{d}s\nonumber\\
&\qquad\qquad\qquad\qquad =\; \left(\bigintssss_{T_1}^{T_2} \E\left[\,\left|\mathcal{I}(\bm{\mathcal{T}}; \left(t, \eta\right))\right|\,\right] \textrm{d}t\right)^2\;=\;  \left(\E\left[\widehat{\mathcal{A}}\!\left(\bm{\mathcal{T}}; \eta\right)\right]\right)^2,\label{errorterm1}
\end{align}
where the last equation is again a consequence of Fubini's Theorem. This gives the first and main term in the estimate in the conclusion of the  statement. The error term comes from the integral over $\Sigma^{(2)}\!\left(\bm{\mathcal{T}}; \eta\right)$. To see this, fix $(s,t)\in \Sigma_{\bm{T}}\!\left(\bm{\mathcal{T}}; \eta\right)$. Then, from the Cauchy--Schwarz inequality, 
\begin{equation}\label{CSexp}
\E\left[  \,\left|\mathcal{I}(\bm{\mathcal{T}}; \left(t, \eta\right))\right| \cdot \left|\mathcal{I}(\bm{\mathcal{T}}; \left(s, \eta\right))\right|  \,\right]\;\le\; \sqrt{\E\left[  \,\left|\mathcal{I}(\bm{\mathcal{T}}; \left(t, \eta\right))\right|^2 \,\right]}\cdot\sqrt{\E\left[  \,\left|\mathcal{I}(\bm{\mathcal{T}}; \left(s, \eta\right))\right|^2 \,\right]}.
 \end{equation}
 Here, 
\begin{align} 
\E\left[  \,\left|\mathcal{I}(\bm{\mathcal{T}}; \left(t, \eta\right))\right|^2 \,\right]\;&\underset{\eqref{rangeloc}}{=}\; \E\left[\left(M(\bm{\mathcal{T}}; \left(t, \eta\right))-m(\bm{\mathcal{T}}; \left(t, \eta\right))+2\eta\right)^2\right]\nonumber\\
&\underset{\eqref{minmaxloc}}{=}\;\E\left[\left(\max_{0\le u\le t^+-t^-}\left(B\left(t^-+u\right)-B\left(t^-\right)\right) \right.\right.\nonumber\\
&\qquad \qquad \qquad\qquad\qquad\qquad\left. \left.- \min_{0\le u\le t^+-t^-}\left(B\left(t^-+u\right)-B\left(t^-\right)\right)+2\eta\right)^2\right] \nonumber\\
&\underset{(P3)\&(P4)}{=}\;\E\left[\left(\underbrace{\left(\max_{0\le v\le t^+-t^-}B\left(v\right)\right) - \left(\min_{0\le v\le t^+-t^-}B\left(v\right)\right)}_{= \;R\left(t^+-t^-\right)}+2\eta\right)^2\right] \nonumber\\
&=\; \E\left[R^2\left(t^+-t^-\right)\right]+ 4\eta\cdot \E\left[R\left(t^+-t^-\right)\right]+ 4\eta^2.\label{secmomfirst}
\end{align} 
Each of these expectations is estimated separately thanks to the explicit expression~\eqref{densityR} for the density $\rho_t$ of the random variable $R(t)$. Thus, 
\begin{align*}
 \E\left[R\left(t^+-t^-\right)\right]\;&=\; \int_0^{\infty}\sqrt{\frac{2}{\pi}}\cdot F'\left(\frac{x}{2\sqrt{t^+-t^-}}\right)\cdot\textrm{d}x\\
 &=\; 2\sqrt{\frac{2}{\pi}}\cdot\sqrt{t^+-t^-}\cdot\left(\int_0^\infty F'\right)\; =\; 2\sqrt{\frac{2}{\pi}}\cdot\sqrt{t^+-t^-}
\end{align*}
(recall that $F$ is a smooth cumulative distribution function). Similarly, 
\begin{align*}
 \E\left[R^2\left(t^+-t^-\right)\right]\;&=\; \int_0^{\infty}\sqrt{\frac{2}{\pi}}\cdot x\cdot F'\left(\frac{x}{2\sqrt{t^+-t^-}}\right)\cdot\textrm{d}x\\
 &=\; 4\sqrt{\frac{2}{\pi}}\cdot\left(t^+-t^-\right)\cdot\left(\int_0^\infty y\cdot F'(y)\cdot\textrm{d}y\right)\; \ll\;t^+-t^-,
\end{align*}
where the last inequality is easily deduced from the explicit form for $F'$ in~\eqref{densityF}. As a consequence, one infers from equation~\eqref{secmomfirst} that 
\begin{equation*}
\E\left[  \,\left|\mathcal{I}(\bm{\mathcal{T}}; \left(t, \eta\right))\right|^2 \,\right]\;\ll\; t^+-t^-+\eta\cdot \sqrt{t^+-t^-}+\eta^2\;\underset{\eqref{notatiembis}}{\ll}\; \eta.
\end{equation*}
With the help of inequality~\eqref{CSexp} and upon denoting by $\left|\; \cdot\; \right|_2$ the two-dimensional Lebesgue measure, this yields the following upper bound for the integral over the domain $\Sigma^{(2)}\!\left(\bm{\mathcal{T}}; \eta\right)$~: 
\begin{align}
&\bigintssss_{\Sigma^{(2)}\!\left(\bm{\mathcal{T}};\eta\right)}\E\left[\, \left|\mathcal{I}(\bm{\mathcal{T}}; \left(t, \eta\right))\right|\cdot \left|\mathcal{I}(\bm{\mathcal{T}}; \left(s, \eta\right))\right| \, \right]\cdot \textrm{d}t\cdot\textrm{d}s\nonumber\\
&\qquad \qquad \qquad \le \bigintssss_{\Sigma^{(2)}\!\left(\bm{\mathcal{T}}; \eta\right)} \sqrt{\E\left[  \,\left|\mathcal{I}(\bm{\mathcal{T}}; \left(t, \eta\right))\right|^2 \,\right]\cdot\E\left[  \,\left|\mathcal{I}(\bm{\mathcal{T}}; \left(s, \eta\right))\right|^2 \,\right]} \cdot \textrm{d}t\cdot\textrm{d}s \nonumber\\
&\qquad \qquad \qquad\ll\; \eta\cdot \left|\Sigma^{(2)}\!\left(\bm{\mathcal{T}; \eta}\right)\!\right|_2\nonumber\\
&\qquad \qquad \qquad\underset{\eqref{sigma2}}{\ll}\; \eta\cdot \left|\left\{T_1\le s\le t\le T_2\; :\; s\le t\le s+2\eta \right\}\right|_2\nonumber\\
&\qquad \qquad \qquad\ll\; \left(T_2-T_1\right)\cdot \eta^2. \label{errorterm2}
\end{align} 
The lemma then just follows upon putting together the integral decomposition~\eqref{decompint}, the upper bound~\eqref{errorterm1}  for the subintegral over the domain $\Sigma^{(1)}\!\left(\bm{\mathcal{T}; \eta}\right)$, and the above estimate~\eqref{errorterm2} for the subintegral over the domain $\Sigma^{(2)}\!\left(\bm{\mathcal{T}}; \eta\right)$ . 
\end{proof}
$\quad$

\begin{proof}[Completion of the proof of the Tubular Neighbourhood Area Theorem] The first moment estimate in the Tubular Neighbourhood Area Theorem is precisely the content of Proposition~\ref{firstmom}. As for the volume estimate, fix a real $\eta\in (0,1)$   and let $p\ge 2$ and $k\ge 0$ be integers such that $$\left(1-\frac{1}{p}\right)^{k+1}\;\le\; \eta\;<\; \left(1-\frac{1}{p}\right)^{k}.$$ For the sake of simplicity of notation, define the random variable $$\widehat{\mathcal{A}}^{(p)}\!\left(\bm{\mathcal{T}};k\right)\;=\; \widehat{\mathcal{A}}\!\left(\bm{\mathcal{T}}; \left(1-\frac{1}{p}\right)^{k}\right)$$ so that 
\begin{equation}\label{encavol}
\widehat{\mathcal{A}}^{(p)}\!\left(\bm{\mathcal{T}}; k+1\right) \;\le\; \widehat{\mathcal{A}}\!\left(\bm{\mathcal{T}};\eta\right)\;\le\; \widehat{\mathcal{A}}^{(p)}\!\left(\bm{\mathcal{T}};k\right).
\end{equation}
Fix $\delta>0$. It is then a consequence of Chebyshev's inequality that 
\begin{align*}
&\bP\left(\left|\widehat{\mathcal{A}}^{(p)}\!\left(\bm{\mathcal{T}};k\right)-\E\left[\widehat{\mathcal{A}}^{(p)}\!\left(\bm{\mathcal{T}};k\right)\right]\right|\ge \delta\cdot \E\left[\widehat{\mathcal{A}}^{(p)}\!\left(\bm{\mathcal{T}};k\right)\right] \right)\\
&\qquad \qquad\qquad\qquad \qquad \qquad\le\; \frac{1}{\delta^2}\cdot\left(\frac{\E\left[\left(\widehat{\mathcal{A}}^{(p)}\!\left(\bm{\mathcal{T}};k\right)\right)^2\right]}{\left(\E\left[\widehat{\mathcal{A}}^{(p)}\!\left(\bm{\mathcal{T}};k\right)\right]\right)^2}-1\right)\;\ll\; \frac{1}{\delta^2}\cdot \left(1-\frac{1}{p}\right)^{k},
\end{align*}
where the last inequality is  derived from Proposition~\ref{firstmom} and Lemma~\ref{secmom}. Since the right--hand side of this relation is the general term of a convergent series when summing over $k\ge 1$, the Borel--Cantelli lemma is easily seen to imply  that the sequence of random variables $\left(\widehat{\mathcal{A}}^{(p)}_{\left[T_1, T_2\right]}\left(k\right)/ \E\left[\widehat{\mathcal{A}}^{(p)}_{\left[T_1, T_2\right]}\left(k\right)\right]\right)_{k\ge 0}$ converges almost surely  to $1$ as $k$ tends to infinity.\\

\noindent Notice now that inequalities~\eqref{encavol} yield that
\begin{align*}
\frac{\widehat{\mathcal{A}}^{(p)}\!\left(\bm{\mathcal{T}};k+1\right)}{\E\left[\widehat{\mathcal{A}}^{(p)}\left(\bm{\mathcal{T}};k\right)\right]}\;\le\;  \frac{\widehat{\mathcal{A}}\!\left(\bm{\mathcal{T}};\eta\right)}{\E\left[\widehat{\mathcal{A}}\left(\bm{\mathcal{T}};\eta\right)\right]}\;\le\; \frac{\widehat{\mathcal{A}}^{(p)}\!\left(\bm{\mathcal{T}};k\right)}{\E\left[\widehat{\mathcal{A}}^{(p)}\!\left(\bm{\mathcal{T}};k+1\right)\right]}\cdotp
\end{align*}
It then follows from Proposition~\ref{firstmom} that, almost surely as $k$ tend to infinity,  
\begin{align*}
\sqrt{\frac{p-1}{p}}\;\le\; \liminf_{\eta\rightarrow 0^+}\left(\frac{\widehat{\mathcal{A}}\!\left(\bm{\mathcal{T}};\eta\right)}{\E\left[\widehat{\mathcal{A}}\!\left(\bm{\mathcal{T}};\eta\right)\right]}\right) \;\le\; \limsup_{\eta\rightarrow 0^+}\left(\frac{\widehat{\mathcal{A}}\!\left(\bm{\mathcal{T}};\eta\right)}{\E\left[\widehat{\mathcal{A}}\!\left(\bm{\mathcal{T}};\eta\right)\right]}\right) \;\le\;\sqrt{\frac{p}{p-1}}\cdotp
\end{align*}
Upon letting $p$ tend to infinity, one retrieves that, almost surely as $\eta$ tends to $0^+$, $$\widehat{\mathcal{A}}\!\left(\bm{\mathcal{T}};\eta\right)\;=\; \left(1+o(1)\right)\cdot \E\left[\widehat{\mathcal{A}}\;\left(\bm{\mathcal{T}};\eta\right)\right]\;=\; \left(\frac{4}{\sqrt{\pi}}+o(1)\right)\cdot \left(T_2-T_1\right)\cdot\sqrt{\eta},$$ where the last equation is again a consequence of Proposition~\ref{firstmom}.  This completes the proof of the Tubular Neighbourhood Area Theorem. 
\end{proof}

\section{Counting Rational Points}\label{sec3}

\noindent This section is devoted to the proofs of the First and Second Main Theorems, and also of the Auxiliary Theorem. This is achieved modulo three statements (Propositions~\ref{diagestimterm}, \ref{offdiagestimterm} and~\ref{smeuf}) concerned with moment estimates established in Sections~\ref{asymdiag}, \ref{asymoffdiag} and~\ref{secfin}, respectively (with the help of hitting probability estimates proved in the earlier Section~\ref{probsquare}).

\subsection{Mean--Variance Lemma and Probability Estimates}\label{sbscMVLPE}

\paragraph{Statements of the lemma.} The proofs of the First and Second Main Theorems, and also of the Auxiliary Theorem, rely on the Second Moment Method as formalised in the following Mean-Variance Lemma. This result, which dates back to the work by Rademacher on orthogonal series 
(see~\cite[Chap.~1]{Ha} and the references within for details), is well--known in the context of Metric Diophantine Approximation. It is then typically applied upon taking the probability space as the unit interval equipped with the Lebesgue measure. The novelty in the present situation is to depart from this framework by working with a  "genuine" probability space, namely the one which Brownian motion is defined over. This involves a wealth of new considerations different in nature from the standard Diophantine framework.

\begin{lem}[Mean--Variance Lemma] \label{ebc}
Fix a probability space $(X,\mathcal{B},\mu)$. Let $\left(f_q\right)_{q \ge 0}$ be a sequence of real--valued, nonnegative and $\mu$-measurable functions defined on $X$. Consider two sequences of real numbers $\left(\widehat{f}_q\right)_{q \ge 0 }$ and $\left(\varphi_q\right)_{q  \ge 0}$  such that for all $q\ge 0$, 
\begin{equation}\label{ineqfphiMVE} 
0\leq \widehat{f}_q \leq \varphi_q.  
\end{equation}
\noindent 
Assume also that for   any integers $0\le  Q_1<Q_2 $,  
\begin{equation} \label{ebc_condition1} 
\int_{X} \left(\sum_{q=Q_1}^{Q_2} \big( f_q(x) -  \widehat{f}_q \big) \right)^2\cdot \mathrm{d}\mu(x)\, \ll\,   \Phi(Q_2)-\Phi(Q_1-1),
\end{equation}
where $\Phi(-1)=0$ and where given an integer $Q\ge 0$, 
\begin{equation}\label{partiasumMVE}
\Phi(Q)\; :=\; \sum\limits_{q=0}^{Q}\varphi_q .
\end{equation} 
Then,  for every $\eta>0$ and for $\mu$-almost all $x\in X$,
\begin{equation} \label{ebc_conclusion}
\sum_{q=0}^Q f_q(x)\; =\; \sum_{q=0}^{Q}\widehat{f}_q\, +\, O\left(\Phi(Q)^{1/2}\cdot \log^{3/2+\eta}\Phi(Q)+\max_{1\leq q\leq Q}\widehat{f}_q\right).
\end{equation}
\end{lem}
 
\noindent The standard form of this statement --- see~\cite[Lemma~1.5]{Ha} for a proof --- requires the extra assumption that the partial sum~\eqref{partiasumMVE} should diverge to infinity. As it turns out, this assumption is unnecessary, and when  
\begin{equation}\label{convcaseMVE}
C\;:=\; \lim_{Q\rightarrow\infty} \Phi(Q)\;<\; \infty, 
\end{equation} 
the conclusion means that, for $\mu$--almost all $x\in X$, 
\begin{equation}\label{convconclu}
0\;\le\;  \sum_{q=1}^{\infty} f_q(x)\;<\;\infty.
\end{equation}

\begin{proof}[Proof of the Mean--Variance Lemma in the convergence case] Assume that the relation~\eqref{convcaseMVE} holds. From the assumption~\eqref{ineqfphiMVE}, one has then also that 
\begin{equation} \label{result_convergence}
0\;\le\; \sum_{q=0}^Q \widehat{f}_q\;\le\; C
\end{equation}
for all $Q\ge 0$. The first step in the proof is to show for all $Q\ge 0$,
\begin{equation} \label{int_result}
 \int_{X} \left(\sum_{q=0}^{Q} f_q(x)\right)\cdot  \mathrm{d}\mu(x)\; \ll \; 1.
\end{equation}
To see this, note that when the relation~\eqref{convcaseMVE} holds, the assumption~\eqref{ebc_condition1} becomes 
\[
\int_{X} \left(\sum_{q=0}^{Q} \big( f_q(x) -  \widehat{f}_q \big) \right)^2\cdot \mathrm{d}\mu(x)\;\ll\; 1.
\]
Upon opening the square, this is readily seen to yield from~\eqref{result_convergence} that 
\begin{equation} \label{int_X_sum_f_is_bounded}
\int_{X} \left(\left(\sum_{q=0}^{Q} f_q(x)  \right)^2-2C\cdot \left(\sum_{q=0}^{Q} f_q(x)\right)\right)\cdot \mathrm{d}\mu(x) \ll 1.
\end{equation}
Furthermore, the following inequality is easily verified by making a distinction of cases depending on whether the left--hand side is smaller or bigger than $2C+1$~: 
\begin{equation*} \label{ie_sum_fn}
\sum_{q=0}^{Q} f_q(x) \;\leq\; \max\left\{2C+1,\; \left(\sum_{q=0}^{Q} f_q(x)  \right)^2-2C\cdot  \left(\sum_{q=0}^{Q} f_q(x)  \right)\right\}.
\end{equation*}
The  bound~\eqref{int_result} is then obtained upon integrating this inequality and upon calling on the relation~\eqref{int_X_sum_f_is_bounded}.\\

\noindent Finally, the sought conclusion~\eqref{convconclu} is derived from the upper bound~\eqref{int_result} as a consequence of the Monotone Convergence Theorem. Indeed, this theorem yields that $$  \int_{X} \left(\sum_{q=0}^{\infty} f_q(x)\right)\cdot  \mathrm{d}\mu(x)\; \underset{\eqref{int_result}}{\ll} \; 1,$$ implying in particular that the nonnegative function $x\mapsto \sum_{q=1}^{\infty} f_q(x)$ is finite for $\mu$--almost every $x\in X$.\\
\end{proof}

\noindent The First and Second Main Theorems,  and also the Auxiliary Theorem, are deduced in the following subsections with the help of the Mean--Variance Lemma~\ref{ebc}. To this end, some further notation is  introduced. 

\paragraph{Notation.} A rectangle $\mathcal{R}$ in the plane is understood as a closed  set with nonempty interior of the form $\mathcal{R}=\left[a, b\right]\times\left[c, d\right]$ for some reals $a<b$ and $c<d$. It is then convenient to set $x_{\mathcal{R}}^{-}=a$, $x_{\mathcal{R}}^{+}=b$, $y_{\mathcal{R}}^{-}=c$ and $y_{\mathcal{R}}^{+}=d$ and, when there is no ambiguity in the choice of the rectangle $\mathcal{R}$, to further simplify these notations to $x^{-}= x_{\mathcal{R}}^{-}$, $x^{+}= x_{\mathcal{R}}^{+}$, $y^{-}= y_{\mathcal{R}}^{-}$ and $y^{+}= y_{\mathcal{R}}^{+}$. Thus, 
\begin{equation*}
\mathcal{R}\;=\; \left[x_{\mathcal{R}}^{-}, \;x_{\mathcal{R}}^{+}\right]\,\times\,\left[y_{\mathcal{R}}^{-}, \; y_{\mathcal{R}}^{+}\right]\;=\; \left[x^{-}, x^{+}\right]\,\times\,\left[y^{-}, y^{+}\right].
\end{equation*}
The characteristic function of the event that a Brownian motion defined over the proba\-bi\-lity space $\left(\Omega, \mathcal{F}, \mathbb{P}\right)$   should intersect $\mathcal{R}$ is the map
\begin{equation*}
\chi_{\mathcal{R}}~: \omega\in\Omega\;\mapsto\; 
\begin{cases}
1 &\textrm{if there exists }   t\in \left[x^{-}, x^{+}\right]\textrm{such that }  B(t)\in \left[y^{-},  y^{+}\right]; \\
0 & \textrm{otherwise.}
\end{cases}
\end{equation*}
The probability that Brownian motion should intersect  $\mathcal{R}$ is then   the quantity denoted by
\begin{equation}\label{defproR}
\ppi\left(\mathcal{R}\right)\;=\; \E\left[ \chi_{\mathcal{R}}\right]\;=\; \int_{\omega} \chi_{\mathcal{R}}(\omega)\cdot \textrm{d}\bP(\omega).
\end{equation}
Similarly, the notation $\ppi\left(\mathcal{R}, \mathcal{R}'\right)$ stands for the probability that Brownian motion should pass through the two given rectangles $\mathcal{R}$ and $\mathcal{R}'$; that is, 
\begin{equation*}
\ppi\left(\mathcal{R}, \mathcal{R}'\right) = \E\left[ \chi_{\mathcal{R}}\cdot \chi_{\mathcal{R'}}\right].
\end{equation*}

\noindent  In view of the counting results to establish, of particular interest is the case where rectangles are squares centered at rational points with side lengths determined by the appro\-ximation function $\xi$ considered in the statement of the First Main Theorem. With this in mind, given reals $T_2>T_1>0$, set $\bm{\mathcal{T}}=(T_1, T_2)$ and define 
\begin{equation}\label{defstrip}
\Q_{\xi}( \bm{\mathcal{T}})\;=\; \left\{\frac{\bm{s}}{q}=\left(\frac{s_1}{q}, \frac{s_2}{q}\right)\in\Q^2\;:\; \left[\frac{s_1}{q}-\frac{\xi(q)}{2q}, \frac{s_1}{q}+\frac{\xi(q)}{2q}\right]\cap\left[T_1, T_2\right]\;\neq\; \emptyset\right\}. 
\end{equation}
When $\bm{s}/q=(s_1/q, s_2/q)$ is a rational point in the plane with denominator $q\ge 1$ and numerators $s_1, s_2\in\Z$, let also 
\begin{equation}\label{defract}
\mathcal{R}_\xi\left(\bm{\mathcal{T}}; \frac{\bm{s}}{q}\right)\;=\; \left(\left[\frac{s_1}{q}-\frac{\xi(q)}{2q}, \; \frac{s_1}{q}+\frac{\xi(q)}{2q}\right]\cap \left[T_1, T_2\right] \right)\,\times\,\left[\frac{s_2}{q}-\frac{\xi(q)}{2q}, \; \frac{s_2}{q}+\frac{\xi(q)}{2q}\right].
\end{equation}
If $\bm{s}/q$ and $\bm{s}'/q'$ are two rational points with positive denominators, set for the sake of simplicity of notations
\begin{equation}\label{defpiprob}
\ppi_\xi\left(\bm{\mathcal{T}}; \frac{\bm{s}}{q}\right)\;=\; \ppi\left(\mathcal{R}_\xi\left(\bm{\mathcal{T}}; \frac{\bm{s}}{q}\right)\right) \end{equation}
and
\begin{equation}\label{defpiprobbis}
 \ppi_\xi\left(\bm{\mathcal{T}}; \left(\frac{\bm{s}}{q}, \frac{\bm{s}'}{q'}\right)\right)\;=\; \ppi\left(\mathcal{R}_\xi\left(\bm{\mathcal{T}}; \frac{\bm{s}}{q}\right), \mathcal{R}_\xi\left(\bm{\mathcal{T}}; \frac{\bm{s}'}{q'}\right)\right).
\end{equation}
$\quad$

\paragraph{Probability Estimates.} The First Main Theorem, and parts of the Second and the Auxiliary ones,  are derived from the Mean--Variance  Lemma~\ref{ebc}. The key results allowing one to show that the assumptions of the lemma  hold consist in determining the asymptotic behavior of two quantities,  the so--called \emph{diagonal} and \emph{off--diagonal pro\-ba\-bi\-lity terms}. Given integers $Q_2\ge Q_1\ge 1$, upon setting $\bm{\mathcal{Q}}=\left(Q_1, Q_2\right)$, they are respectively defined as

\begin{equation}\label{DPT}
\Delta_{\xi}\left(\bm{\mathcal{Q}} ; \bm{\mathcal{T}}  \right)
\;\underset{\eqref{defpiprob}}{=}\;
\sum_{q=Q_1}^{Q_2}
\sum_{\frac{\bm{s}}{q}\in \Q_{\xi}(\bm{\mathcal{T}})}
\ppi_\xi\!\left(\bm{\mathcal{T}}; \frac{\bm{s}}{q}\right).
\end{equation}
and
\begin{equation}\label{ODPT}
\Theta_{\xi}\left(\bm{\mathcal{Q}}; \bm{\mathcal{T}}\right)
\;\underset{\eqref{defpiprob}}{=}\;
\sum_{Q_1\le q,q'\le Q_2}
\sum_{\frac{\bm{s}}{q},\,\frac{\bm{s}'}{q'}\in \Q_{\xi}(\bm{\mathcal{T}})}
\ppi_\xi\left(
\bm{\mathcal{T}};
\left(\frac{\bm{s}}{q},\frac{\bm{s}'}{q'}\right)
\right).
\end{equation}

\noindent To this end, the probability that a realisation of a Brownian motion should pass through a given rectangle $\mathcal{R}$ is first determined accurately in Section~\ref{probsquare}. This yields the following sharp estimates for the above--introduced quantities. 
Their expressions depend on the parameter $V\!\left(\bm{\mathcal{T}}\right)$ introduced in the statement of the First Main Theorem, namely
\begin{equation}\label{defVT} 
V\!\left(\bm{\mathcal{T}}\right)\;=\; \max\left\{1,\;  \frac{1}{\sqrt{T_1}},\; T_2-T_1,\; \sqrt{T_2}-\sqrt{T_1},\; \log\left(\frac{T_2}{T_1}\right),\; \frac{1}{T_2-T_1}\right\},
\end{equation}
and are valid when 
\begin{equation}\label{asymptlageT_1T_2} 
Q_1\;\ge\; \max\left\{\frac{4}{T_1},\; \frac{1}{T_2-T_1},\; \frac{2}{T_2}\right\}.
\end{equation}

\begin{prop}[Estimation of the Diagonal Probability Term]\label{diagestimterm}
With the above notation and assumptions, in particular under the assumption that the inequality~\eqref{asymptlageT_1T_2} holds,  one has that 
\begin{align*}
\Delta_{\xi}\left(\bm{\mathcal{Q}}; \bm{\mathcal{T}}\right)\;=\; \frac{2\sqrt{2}}{\sqrt{\pi}}\cdot\left(T_2-T_1\right)&\cdot \sum_{q=Q_1}^{Q_2}q^{3/2}\cdot\sqrt{\xi(q)}\\
&+ O\!\left(V\!\left(\bm{\mathcal{T}}\right)\cdot\left(\sum_{q=Q_1}^{Q_2}\sqrt{q\cdot\xi(q)} \cdot\left(1+\sqrt{q\cdot\xi(q)}+\log q\right)\right)\right).
\end{align*}
\end{prop}


\begin{prop}[Estimation of the Off--Diagonal Probability Term]\label{offdiagestimterm} Let $\mu\in (0,1/27)$. Assume here also that the lower bound~\eqref{asymptlageT_1T_2} holds.  Then, under the additional assumptions that 
\begin{equation*}
Q_1\;\ge\; 12^{1/(2\mu)}
\end{equation*}
and that the map $q\mapsto \xi(q)/q$ is nonincreasing, 
\begin{align*}
\Theta_{\xi}\left(\bm{\mathcal{Q}}; \bm{\mathcal{T}}\right)\;=\; &\left(\Delta_{\xi}\left(\bm{\mathcal{Q}}; \bm{\mathcal{T}}\right)\right)^2\\
&+O\left(V\!\left(\bm{\mathcal{T}}\right)^3\cdot \Delta_{\xi}\left(\bm{\mathcal{Q}}; \bm{\mathcal{T}}\right)\cdot\left(1+\sum_{q=Q_1}^{Q_2}\!\left(q\xi(q) +q^{\frac{1}{2}+\mu}\cdot\xi(q)^{\frac{1}{2}-\mu}\right)\right)\right)\\
&+O\left(V\!\left(\bm{\mathcal{T}}\right)\cdot \sum_{q=Q_1}^{Q_2}\sqrt{q\cdot \xi(q)}\cdot \sigma(q)\right).
\end{align*} 
Here, the implicit constant in the first error term depends on $\mu$, and  
\begin{equation}\label{sumdiv0}
\sigma(q)\;=\; \sum_{d|q}d
\end{equation}
denotes the sum of the divisors of an integer $q\ge 1$. Also, as $Q$ tends to infinity, 
\begin{equation}\label{defMxibis0}
\sum_{q=1}^{Q}\sqrt{q\cdot \xi(q)}\cdot \sigma(q)\;=\;   \frac{\pi^2}{6}\cdotp \Xi(Q)\;+\; O\left(\sum_{q=1}^{Q}\sqrt{q\xi(q)}\cdot\log q\right),
\end{equation}
where 
\begin{equation}\label{estimsub4asym0}
 \Xi(Q)\;=\; \sum_{q=1}^{Q}\sqrt{q^3\cdot \xi(q)}
\end{equation}
is the partial sum introduced in the First Main Theorem.
\end{prop}
 
 \noindent As detailed in Section~\ref{asymoffdiag} (see the comment just following Proposition~\ref{propestimregi} therein), the assumption of  monotonicity of the map $q\mapsto \xi(q)/q$, although it cannot be entirely removed, is used very weakly in the proof of this statement.\\
 
 \noindent To establish the First Main Theorem, the Mean--Variance Lemma~\ref{ebc} is applied with the help of the above two propositions to a sequence of measurable functions $\left(f_q\right)_{q\ge 1}$  depending on the approximation function $\xi$ part of the statement of theorem. The proof of the Mean--Variance Lemma~\ref{ebc}  in~\cite[Lemma~1.5]{Ha} and the related work on effectivity carried out by Lee and Scoones~\cite{LS} make it clear that the implicit constant in the error term~\eqref{ebc_conclusion} then heavily depends on $\xi$. This constitutes a major drawback when dealing with  the almost sure estimates in the Second Main Theorem. Indeed, one then considers the family $\left(\xi_Q^{(\delta)}\right)_{Q\ge 1}$ of approximation functions as introduced in~\eqref{defpsideltaq}. The implicit constant in the above--mentioned error term becomes in this case a function of (a large power of) $Q$, meaning that the conclusion of the Mean--Variance Lemma~\ref{ebc}  becomes irrelevant. \\

\noindent To overcome this difficulty, the proof of the Second Main Theorem relies on separate  moment estimates captured in the following statement~:

%

\begin{prop}[First and Second Moment Estimates for Uniform Counting]\label{smeuf}
Keep the notation of the Second Main Theorem, and recall the definition of the map $\xi^{(\delta)}_{Q}$ in~\eqref{defpsideltaq} and of the corresponding counting function $\mathcal{N}^{\,\flat}\!\left(\bm{\mathcal{T}};\left(\delta, Q\right)\right)$ in~\eqref{defcountrand1}. Assume that 
\begin{equation}\label{assumpsecmom}
\delta\cdot Q^5=\delta(Q)\cdot Q^5\ge 1.
\end{equation} 
Then, there exist an integer $Q_{\bm{\mathcal{T}}}\ge 2$ and constants $A_{\bm{\mathcal{T}}}, K_{\bm{\mathcal{T}}}>0$ such that
\begin{equation}\label{exp0}
\left|\E\left[\mathcal{N}^{\,\flat}\!\left(\bm{\mathcal{T}};\left(\delta, Q\right)\right)\right] - \frac{{4}}{3\sqrt{\pi}}\cdot\left(T_2-T_1\right)\cdot \sqrt{\delta Q^5}\right| \le A_{\bm{\mathcal{T}}}\cdot \left(\sqrt{\delta Q^3}\cdot \log\left(Q\right)+\delta Q^2\right)
\end{equation}
and
\begin{equation}\label{var0}
\V\left[\mathcal{N}^{\,\flat}\!\left(\bm{\mathcal{T}};\left(\delta, Q\right)\right)\right]\le K_{\bm{\mathcal{T}}}\cdot\left(\sqrt{\delta Q^5}+\delta Q^4\log\left(2Q\right)+\sqrt{\delta^3 Q^9}\right).\\
\end{equation}
\end{prop}

$\quad$

\noindent The  First  Main Theorem is derived from the Mean--Variance  Lemma~\ref{ebc} in the next Subsection~\ref{completionproof} with the help of Propositions~\ref{diagestimterm} and~\ref{offdiagestimterm} (the (off)--diagonal probability terms estimates). The  Second  Main Theorem and the Auxiliary Theorem are then inferred from the above Proposition~\ref{smeuf} in Subsection~\ref{completionproofbis}.\\ 

\noindent Propositions~\ref{diagestimterm} and~\ref{offdiagestimterm} are  established in Sections~\ref{asymdiag} and~\ref{asymoffdiag}, respectively. As for Proposition~\ref{smeuf}, it is proved in the final Section~\ref{secfin}.

\subsection{Derivation of the First Main Theorem  from the (Off--)Diagonal Pro\-ba\-bi\-li\-ty Terms Estimates}\label{completionproof}

\subsubsection{Choice of the Parameters.}  The Mean--Variance Lemma~\ref{ebc} is applied with the fol\-lo\-wing choice of parameters~: 

\begin{itemize}
\item the probability space $(X,\mathcal{B},\mu)$ is the space $\left(\Omega, \mathcal{F}, \mathbb{P}\right)$ which Brownian motion is defined over ;

\item all sequences in the statement are set to vanish unless the index $q$ is such that $q\;\ge\; Q_1^*-1$, where
\begin{equation}\label{defQ*1}
\quad Q_1^*\;:=\; \max\left\{\left\lceil \frac{4}{T_1} \right\rceil, \left\lceil  \frac{1}{T_2-T_1}\right\rceil, \left\lceil \frac{2}{T_2} \right\rceil, \left\lceil 12^{1/(2\mu)}\right\rceil, \left\lceil \frac{V\!\left(\bm{\mathcal{T}}\right)^2}{\left(T_2-T_1\right)^2}\right\rceil \right\}.
\end{equation}
When $q=Q_1^*-1$,  the map $f_{\left(Q_1^*-1\right)}$ is set to be constant equal to 1; when  $q\ge Q_1^*$, the map $f_q~: \Omega\rightarrow\R_{\ge 0}$ counts the number of times  a realisation of  Brownian motion passes through a square centered at a rational point with denominator $q$. More precisely, the squares under consideration are  prescribed to be centered at rational points $\bm{s}/q=(s_1/q, s_2/q) \in \Q_{\xi}(\bm{\mathcal{T}})$  
and to have side length $\xi(q)/q$. 
With the above--introduced notation, this is saying that 
\begin{align}\label{deffq}
f_{\left(Q_1^*-1\right)}\;=\; 1\qquad \textrm{and that}\qquad  f_q\;=\; \sum_{\frac{\bm{s}}{q}\in \Q_{\xi}( \bm{\mathcal{T}})} \chi_{\mathcal{R}_\xi\left(\bm{s}/q\right)} \quad\textrm{when } q\ge Q_1^*. 
\end{align}

\item when $q=Q_1^*-1$, set $ \widehat{f}_{\left(Q_1^*-1\right)}=1 $; when   $q\ge Q_1^*$, the real number $ \widehat{f}_q $ is taken as the expectation of $f_q$. In other words, 
\begin{equation}\label{dehtfq}
\widehat{f}_{\left(Q_1^*-1\right)}\;=\; 1 \qquad \textrm{and}\qquad  \widehat{f}_q\;\underset{\eqref{deffq}}{=}\; \E\left[f_q\right] \;\underset{\eqref{defpiprob}}{=}\;  \sum_{\frac{\bm{s}}{q}\in \Q_{\xi}(\bm{\mathcal{T}})} \ppi_\xi\!\left(\frac{\bm{s}}{q}\right) \quad\textrm{when } q\ge Q_1^*.
\end{equation}

\item given a parameter $\mu>0$ and  an integer $q\ge Q_1^*-1$, the real $\varphi_q$ is defined by setting 
\begin{equation}\label{condiboundary}
\varphi_{\left(Q_1^*-1\right)}= \widehat{f}_{\left(Q_1^*-1\right)}\;\underset{\eqref{dehtfq}}{=}\; 1
\end{equation} 
and, when $q\ge Q_1^*$, 
\begin{align}\label{defvarphiq}
\varphi_{q}\;=\; &\sqrt{q\cdot \xi(q)}\cdot \sigma(q)  \;+\;  \left(\Delta^*_{\xi}\left(q, \bm{\mathcal{T}}\right) - \Delta^*_{\xi}\left(q-1, \bm{\mathcal{T}}\right)\right) \nonumber \\
& + \left(\left(\Delta^*_{\xi}\left(q, \bm{\mathcal{T}}\right)\right)^{9(1+3\mu)/5} -\left(\Delta^*_{\xi}\left(q-1, \bm{\mathcal{T}}\right)\right)^{9(1+3\mu)/5}\right).
\end{align}
Here,
\begin{equation} \label{defdelta*} 
\Delta^*_{\xi}\left(Q, \bm{\mathcal{T}}\right)\;=\;\sum_{q=Q_1^*}^{Q}\widehat{f}_q
\end{equation}
(in other words, $ \Delta^*_{\xi}\left(Q, \bm{\mathcal{T}}\right)$ is the Diagonal Probability Term $\Delta_{\xi}\left(\bm{\mathcal{Q}}; \bm{\mathcal{T}}\right)$ when $\bm{\mathcal{Q}}=\left(Q_1^*, Q\right)$). 
\end{itemize}

\subsubsection{Checking the First Domination  Assumption in the Mean--Variance Lemma.} Since all quantities vanish when $q\le Q_1^*-2$, it is enough to check the assumption~\eqref{ineqfphiMVE} when $q\ge Q_1^*-1$.  It clearly holds when $q=Q_1^*-1$ since, by definition in~\eqref{condiboundary},  $\widehat{f}_{(Q_1^*-1)}=1$ and $\varphi_{(Q_1^*-1)}=1$. Assume therefore that $q\ge Q_1^*$. From the Mean Value Theorem applied to the map $x\mapsto x^{9/5}$,
\begin{align*}
\varphi_q\;&\underset{\eqref{defvarphiq}}{\ge}\; \left(\widehat{f}_q+\sum_{k=0}^{q-1}\widehat{f}_k\right)^{9/5}-\left(\sum_{k=0}^{q-1}\widehat{f}_k\right)^{9/5}\\
&\ge\; \frac{9}{5}\cdot \widehat{f}_q\cdot \min_{q\ge 0}\left(\sum_{k=0}^{q}\widehat{f}_k\right)^{4/5}\;\ge\; \frac{9}{5}\cdot \widehat{f}_q\cdot  \widehat{f}_{\left(Q_1^*-1\right)} \;\underset{\eqref{dehtfq}}{\ge}\;  \widehat{f}_q,
\end{align*}
which completes the checking.

\subsubsection{Checking the Second Domination  Assumption  in the Mean--Variance Lemma.}\label{sbsccheck}

\paragraph{$\bullet$}  The proof of the inequality~\eqref{ebc_condition1} is achieved in several steps starting with two auxiliary results. The first one is an estimate proved  in Section~\ref{asymdiag}~:

\begin{prop}\label{diagestimtermbis}
Given an integer $q\ge 1$, one has that $$\left(T_2-T_1\right)\cdot  \sqrt{q^3\xi(q)}\;\ll\; \widehat{f}_q\;\underset{\eqref{defVT}}{\ll}\; V\!\left(\bm{\mathcal{T}}\right)\cdot  \sqrt{q^3\xi(q)}.$$ Here,   the implicit constant in the upper bound is absolute and the one in the lower bound becomes so provided that $q\ge  \left(V\!\left(\bm{\mathcal{T}}\right)/(T_2-T_1)\right)^2 $.
\end{prop} 

\noindent The second auxiliary statement is a lemma on real sequences~: 

\begin{lem}[Elementary Lemma on Real Sequences]\label{elemlem}
Let $\left(a_q\right)_{q\ge 1}$ be a sequence of positive reals defining the partial sums $A_n=\sum_{1\le q\le n}a_q$ when $n\ge 1$. Set conventionally $A_0=0$. Then, given a real parameter $0\le \eta\le 1/2$ and integers $1\le u\le v$, $$0\;\le\; \sum_{q=u}^{v}\frac{a_q}{A_q^\eta} \;\le\;2\cdot \left(A_v^{1-\eta}-A_{u-1}^{1-\eta}\right).$$
\end{lem}

\begin{proof}
Note first that $$\sum_{q=u}^{v}\frac{a_q}{A_q^\eta}\;=\; \sum_{q=u}^{v}\frac{A_q-A_{q-1}}{A_q^\eta}\;=\; A_v^{1-\eta}-A_{u-1}^{1-\eta}\;+\; \sum_{q=u}^vA_{q-1}^{1-\eta}\cdot\left(1-\frac{A_{q-1}^\eta}{A_q^\eta}\right).$$ Furthermore, 
it is elementary to check that 
$$1-\frac{1}{x^\eta}\;\le\; x^{1-\eta}-1\qquad \textrm{when}\qquad x\ge 1\qquad \textrm{ and }\qquad 0\le \eta\le \frac{1}{2}\cdotp$$ Specialising this identity to the case where $x=A_q/A_{q-1}\ge 1$ when $q\ge 2$, a term by term comparison of the sums below yields that $$0\;\le\; \sum_{q=u}^vA_{q-1}^{1-\eta}\cdot\left(1-\frac{A_{q-1}^\eta}{A_q^\eta}\right)\;\le\; \sum_{q=u}^v \left(A_q^{1-\eta}-A_{q-1}^{1-\eta}\right)\;=\; A_v^{1-\eta}-A_{u-1}^{1-\eta}.$$ This completes the proof. 
\end{proof}

\paragraph{$\bullet$} To show that the Domination  Assumption~\eqref{ebc_condition1} holds, fix integers  $Q_2>Q_1\ge Q_1^*$, where the quantity  $Q_1^*$ is defined in~\eqref{defQ*1}. Then, with the above choice of parameters and upon opening the square,
\begin{align}\label{condimeanvardom2}
\int_{\Omega} \left(\sum_{q=Q_1}^{Q_2} \big( f_q(\omega) -  \widehat{f}_q \big) \right)^2\cdot \mathrm{d}\mu(\omega)\;=\; \int_{\Omega} \left(\sum_{q=Q_1}^{Q_2} f_q(\omega)  \right)^2\cdot \mathrm{d}\mu(\omega)-\left(\sum_{q=Q_1}^{Q_2}   \widehat{f}_q \right)^2.
\end{align}
Here, 
\begin{align}\label{deltaexpec}
\sum_{q=Q_1}^{Q_2}   \widehat{f}_q\;\underset{\eqref{dehtfq}}{=}\; \sum_{q=Q_1}^{Q_2}   \sum_{\frac{\bm{s}}{q}\in \Q_{\xi}(\bm{\mathcal{T}})} \ppi_\xi\!\left(\frac{\bm{s}}{q}\right) \;\underset{\eqref{DPT}}{=}\; \Delta_{\xi}\left(\bm{\mathcal{Q}} ; \bm{\mathcal{T}}  \right). 
\end{align}
Furthermore, given that 
two squares centered at rationals with the same denominator $q$ and with side length $\xi(q)/q\le1/q$ can only overlap on their boundaries, one easily checks upon successively opening the various squares that
\begin{align}\label{thetafomeg}
 \int_{\Omega} \left(\sum_{q=Q_1}^{Q_2} f_q(\omega)  \right)^2\cdot \mathrm{d}\mu(\omega)\;&=\; \sum_{Q_1\le q,  q'\le Q_2}\; \sum_{\frac{\bm{s}}{q}, \frac{\bm{s}'}{q'}\in \Q_{\xi}( \bm{\mathcal{T}})} \ppi_\xi\left(\bm{\mathcal{T}}; \left(\frac{\bm{s}}{q}, \frac{\bm{s}'}{q'}\right)\right)\nonumber \\
 &\underset{\eqref{ODPT}}{=}\; \Theta_{\xi}\left(\bm{\mathcal{Q}}; \bm{\mathcal{T}}\right).
 \end{align}
 
\paragraph{$\bullet$} From Propositions~\ref{offdiagestimterm}, the relation~\eqref{condimeanvardom2} thus yields that 
\begin{align}
&\int_{\Omega} \left(\sum_{q=Q_1}^{Q_2} \big( f_q(\omega) -  \widehat{f}_q \big) \right)^2\cdot \mathrm{d}\mu(\omega)\nonumber \\
&\ll V\!\left(\bm{\mathcal{T}}\right)^3\cdot \left( \Delta_{\xi}\left(\bm{\mathcal{Q}}; \bm{\mathcal{T}}\right)\cdot\left(1+\sum_{q=Q_1}^{Q_2}\!\!\left(q\xi(q) +q^{\frac{1}{2}+\mu}\cdot\xi(q)^{\frac{1}{2}-\mu}\right)\right)+ \sum_{q=Q_1}^{Q_2}\!\!\sqrt{q\cdot \xi(q)}\cdot \sigma(q)\right)\nonumber\\
&\ll V\!\left(\bm{\mathcal{T}}\right)^3\cdot \left( \Delta_{\xi}\left(\bm{\mathcal{Q}}; \bm{\mathcal{T}}\right)\cdot\left(1+\sum_{q=Q_1}^{Q_2} q\cdot\xi(q)^{\frac{1}{2}-\mu}\right)+ \sum_{q=Q_1}^{Q_2}\!\!\sqrt{q\cdot \xi(q)}\cdot \sigma(q)\right) \nonumber\\
&\ll V\!\left(\bm{\mathcal{T}}\right)^3\cdot \left( \Delta_{\xi}\left(\bm{\mathcal{Q}}; \bm{\mathcal{T}}\right)\cdot \left(1+  \sum_{q=Q_1}^{Q_2} q^{1+6\mu}\cdot\sqrt{\xi(q)}\right)+ \sum_{q=Q_1}^{Q_2}\!\!\sqrt{q\cdot \xi(q)}\cdot \sigma(q)\right).\label{upboun6mu}
\end{align}
To see the reason why this last inequality holds, partition the index of summation into the two subsets
\begin{align*}
\mathcal{I}_\xi^{(1)}\left(\bm{\mathcal{Q}}\right)= \left\{Q_1\le q\le Q_2 : \xi(q)\le q^{-6}\right\}\quad \textrm{and}\quad \mathcal{I}^{(2)}_\xi\left(\bm{\mathcal{Q}}\right)= \left\{Q_1\le q\le Q_2 : \xi(q)> q^{-6}\right\}.
\end{align*}
Then, 
\begin{align*}
\sum_{q=Q_1}^{Q_2} q\cdot\xi(q)^{\frac{1}{2}-\mu}\le \sum_{q\in \mathcal{I}_\xi^{(1)}\left(\bm{\mathcal{Q}}\right)} q\cdot\xi(q)^{\frac{1}{2}-\mu} + \sum_{q\in \mathcal{I}_\xi^{(2)}\left(\bm{\mathcal{Q}}\right)} q\cdot\xi(q)^{\frac{1}{2}-\mu} \ll 1+ \sum_{q=Q_1}^{Q_2}  q^{1+6\mu}\cdot\xi(q)^{\frac{1}{2}}.
\end{align*}

\paragraph{$\bullet$} Returning to the bound~\eqref{upboun6mu}, the next goal is to show that the first sum therein meets the inequality 
\begin{equation}\label{simpleerror}
 \sum_{q=Q_1}^{Q_2} q^{1+6\mu}\cdot\sqrt{\xi(q)}\;\ll\; V\!\left(\bm{\mathcal{T}}\right)^{6(1-2\mu)/5} \cdot \sum_{q=Q_1}^{Q_2}\frac{\widehat{f}_q}{\Delta^*_{\xi}\left(q, \bm{\mathcal{T}}\right)^{(1-12\mu)/5}}, 
 \end{equation}
where $\Delta^*_{\xi}\left(q, \bm{\mathcal{T}}\right)$ is defined in~\eqref{defdelta*}. To establish this relation, note that in view of Proposition~\ref{diagestimterm}, one has the trivial bound $$\Delta^*_{\xi}\left(q, \bm{\mathcal{T}}\right)\;\ll\; V\!\left(\bm{\mathcal{T}}\right)\cdot q^{5/2}$$ in such a way that $$\left(\frac{\Delta^*_{\xi}\left(q, \bm{\mathcal{T}}\right)}{V\!\left(\bm{\mathcal{T}}\right)}\right)^{(1-12\mu)/5}\;\ll\; q^{1/2-6\mu}.$$
As a consequence, 
\begin{align*}
 \sum_{q=Q_1}^{Q_2} q^{1+6\mu}\cdot\sqrt{\xi(q)}\;&\ll\; V\!\left(\bm{\mathcal{T}}\right)^{(1-12\mu)/5} \sum_{q=Q_1}^{Q_2}\frac{\sqrt{q^3\xi(q)}}{\Delta^*_{\xi}\left(q, \bm{\mathcal{T}}\right)^{(1-12\mu)/5}}\cdotp
\end{align*}
The bound~\eqref{simpleerror} then follows from the conclusion of Proposition~\ref{diagestimtermbis} and from the definition of the quantity $V\!\left(\bm{\mathcal{T}}\right)$ in~\eqref{defVT}.

\paragraph{$\bullet$} From the Elementary Lemma~\ref{elemlem} on Real Sequences applied with the choice of exponent $\eta=(1-12\mu)/5$, 
\begin{equation*}
\sum_{q=Q_1}^{Q_2}\frac{\widehat{f}_q}{\Delta^*_{\xi}\left(q, \bm{\mathcal{T}}\right)^{(1-12\mu)/5}} \;\ll\; \left(\Delta^*_{\xi}\left(Q_2, \bm{\mathcal{T}}\right)\right)^{4(1+3\mu)/5} - \left(\Delta^*_{\xi}\left(Q_1-1, \bm{\mathcal{T}}\right)\right)^{4(1+3\mu)/5}.
\end{equation*}
The bounds~\eqref{upboun6mu} and~\eqref{simpleerror} then imply that 
\begin{align}
&\int_{\Omega} \left(\sum_{q=Q_1}^{Q_2} \big( f_q(\omega) -  \widehat{f}_q \big) \right)^2\cdot \mathrm{d}\mu(\omega)\label{varint} \\
&\underset{\eqref{defdelta*}}{\ll} V\!\left(\bm{\mathcal{T}}\right)^{(21-12\mu)/5}\cdot \left(\sum_{q=Q_1}^{Q_2}\!\!\sqrt{q\cdot \xi(q)}\cdot \sigma(q) \;+ \right. \nonumber\\
&\left. \left(\Delta^*_{\xi}\left(Q_2, \bm{\mathcal{T}}\right)-\Delta^*_{\xi}\left(Q_1-1, \bm{\mathcal{T}}\right)\right)\cdot \left(1+  \left(\Delta^*_{\xi}\left(Q_2, \bm{\mathcal{T}}\right)\right)^{4(1+3\mu)/5} - \left(\Delta^*_{\xi}\left(Q_1-1, \bm{\mathcal{T}}\right)\right)^{4(1+3\mu)/5}\right)\right)\nonumber\\
&\ll\; \sum_{q=Q_1}^{Q_2} \varphi_q, \label{varintbis}
\end{align}
where this last inequality follows from the definition of the sequence $\left(\varphi_q\right)_{q\ge 1}$  in~\eqref{defvarphiq}. This confirms that the second Domination  Assumption~\eqref{ebc_condition1} indeed holds.

\subsubsection{Completion of the Proof of the First Main Theorem  from Propositions~\ref{diagestimterm} and~\ref{offdiagestimterm}}\label{compprFMT}

\begin{proof}[Proof of the First Main Theorem]
From the definition of the sequences $\left(f_q\right)_{q\ge 1}$, $\left(\widehat{f}_q\right)_{q\ge 1}$ and $\left(\varphi_q\right)_{q\ge 1}$ in \eqref{deffq}, \eqref{dehtfq} and~\eqref{defvarphiq},  respectively, it follows from the conclusion of the Mean--Variance Lemma~\ref{ebc} that, almost surely and for all $\mu\in (0, 1/27)$ and all integers $Q\ge 1$,
\begin{align*}
& \widehat{\mathcal{N}}_{ \xi}\left(\bm{\mathcal{T}}; Q\right)\;\underset{\eqref{defdelta*}}{=}\; \Delta^*_{\xi}\left(Q, \bm{\mathcal{T}}\right)\\
&+O\left(\left(\Delta^*_{\xi}\left(Q, \bm{\mathcal{T}}\right)\;+\; \Delta^*_{\xi}\left(Q, \bm{\mathcal{T}}\right)^{9(1+3\mu)/5}\;+\; \sum_{q=1}^{Q}\sqrt{q\cdot \xi(q)}\cdot \sigma(q)\right)^{1/2}\times\right.\\
&\qquad \qquad \quad \quad \left.\left(\log\left(\Delta^*_{\xi}\left(Q, \bm{\mathcal{T}}\right)\; +\; \Delta^*_{\xi}\left(Q, \bm{\mathcal{T}}\right)^{9(1+3\mu)/5}\; +\; \sum_{q=1}^{Q}\sqrt{q\cdot \xi(q)}\cdot \sigma(q)\right)\right)^{3/2+\mu}\right)\\
&+O\left(\max_{1\le q\le Q} \widehat{f}_q\right).
\end{align*}
From Propositions~\ref{diagestimterm} and~\ref{diagestimtermbis}, and from the equation~\eqref{defMxibis0} stated in Proposition~\ref{offdiagestimterm}, this reduces to the asymptotic expansion stated in the First Main Theorem under the assumption of the monotonicity of the map $q\mapsto \xi(q)/q$ when $q\ge Q_1^*$. Here, the integer $Q_1^*$ is  defined in~\eqref{defQ*1}, and the condition $q\ge Q_1^*$   is the one under which the assumptions of the Mean--Variance Lemma involve implicit constants not depending on $\bm{\mathcal{T}}$ (from Proposition~\ref{diagestimtermbis}).
\end{proof}

\subsection{Derivation of the Second Main Theorem and of the Auxiliary Theorem  from the Moments Estimate for Uniform Counting}\label{completionproofbis}

\noindent The proofs of the Second Main Theorem  and of the Auxiliary Theorem share several features and are thus proved together in this section. Each of the claims they are made of is proved separately and successively.

\begin{proof}[Proof of the Asymptotic Uniform Counting Estimate~\eqref{asymunifestim} valid for Large Values of $\delta Q^5$] Let $j{\ge 1}$ be  an exponent and $Q\ge 1$ be an integer in the dyadic block $2^j\le Q<2^{j+1}$. Given a parameter $\eta>0$, set throughout this proof
\begin{equation}\label{defvarthetaj}
\vartheta_j=j^{-\eta/16}.
\end{equation}

\noindent Fix a subset 
\begin{equation}\label{cardQj}
\mathcal{Q}_j\subset \left[2^j, {2^{j+1}}\right]\cap \N\qquad\textrm{with cardinality }\qquad \#\mathcal{Q}_j\ll \vartheta^{-1}_j
\end{equation}  
such that for every $Q\in  \left[2^j, 2^{j+1}-1\right]$, there exist integers $Q_\pm\in \mathcal{Q}_j$ meeting the relations 
\begin{align}\label{sqpm}
Q_-\le Q\le Q_+\qquad \textrm{and}\qquad \frac{Q_+}{Q_-}=1+O\left(\vartheta_j\right),
\end{align}
where the implicit constant is absolute. These conditions can for instance be achieved upon choosing 

\begin{equation*}
\mathcal{Q}_j:=\left\{
\left\lfloor 2^j(1+\vartheta_j)^k\right\rfloor \;:\;  0\le k\le K_j \right\} \cup \left\{2^j, {2^{j+1}}\right\}, 
\end{equation*}
where
\begin{equation*}
K_j=\max\left\{k\ge 0\; :\; 2^j(1+\vartheta_j)^k<2^{j+1}\right\}.
\end{equation*}
Given constants $c_1, c_2>0$, define also $\mathcal{R}_j$ as the finite geometric progression 
\begin{equation}\label{progrgeomrj}
\mathcal{R}_j=
\left\{
c_1j^{2+\eta}2^{-6(j+1)}(1+\vartheta_j)^\ell\;:\; 0\le \ell\le m_j
\right\},
\end{equation}
where $m_j$ is the largest integer $\ell$ for which $c_1j^{2+\eta}2^{-6(j+1)}(1+\vartheta_j)^\ell\le c_22^{-j}$. Thus, 
\begin{equation}\label{cardiinclurj}
\#\mathcal{R}_j\ll j\vartheta_j^{-1} \qquad \textrm{and}\qquad  \mathcal{R}_j\subset \left[c_1j^{2+\eta}2^{-6(j+1)}, \; c_22^{-j}\right].
\end{equation}

\noindent Here, in view of the assumption  $\delta<1/4$ and the assumption~\eqref{lowboudelB} that $\delta(Q)\cdot Q^5\gg\left(\log Q\right)^{2+\eta}$, the constants $c_1$ and $c_2$ can be and are chosen so that for any integer $j\ge 0$, 
\begin{equation}\label{ineqdelatqQ}
c_1j^{2+\eta}2^{-6(j+1)}\le \frac{\delta(Q)}{Q} \le c_22^{-j}\qquad \textrm{for any }\qquad Q\in  \left[2^j, 2^{j+1}-1\right].
\end{equation}

\noindent Given a sufficiently small parameter $c_3>0$, let 
\begin{equation*}
\mathfrak{F}_j=\left\{\left(Q', r'\right)\in \mathcal{Q}_j\times \mathcal{R}_j\;:\; 0<\delta'<\frac{1}{4}\quad\textrm{and}\quad\sqrt{\delta'\cdot\left(Q'\right)^5}\ge c_3j^{1+\eta/2}\right\},
\end{equation*}
where 
\begin{equation}\label{deltaQR}
\delta'=Q'r'. 
\end{equation}
It is immediate that when $Q$ lies in the range $\left[2^j, 2^{j+1}-1\right]$, the assumption~\eqref{lowboudelB} that $\delta(Q)\cdot Q^5\gg\left(\log Q\right)^{2+\eta}$ implies that $\sqrt{\delta\cdot Q^5}\gg j^{1+\eta/2}$. As a consequence, upon adjusting the parameter $c_3$ accordingly, the inclusion in~\eqref{cardiinclurj} combined with the inequalities in~\eqref{ineqdelatqQ}, the assumptions that $\limsup_{Q\rightarrow\infty}\delta(Q)<1/4$ and the clear fact that $\lim_{j\rightarrow \infty }\vartheta_j=0$ altogether yield that every planar point $\left(Q, \delta(Q)/Q\right)$ lies in a rectangle whose vertices are in the set $\mathfrak{F}_j$ provided that $Q$ is in the dyadic range under consideration. \\

\noindent Fix then a point $\left(Q', r'\right)\in \mathfrak{F}_j$. From Chebyshev's inequality and from the variance estimate for Uniform Counting stated in Proposition~\ref{smeuf},
\begin{align}
&\bP\left(\left| \mathcal{N}^{\,\flat}\!\left(\bm{\mathcal{T}};\left(\delta', Q'\right)\right)- \E\left[\mathcal{N}^{\,\flat}\!\left(\bm{\mathcal{T}};\left(\delta', Q'\right)\right)\right]\right|\ge \vartheta_j c_3^{-1}\sqrt{\delta'\cdot \left(Q'\right)^5}\right)\ll \frac{{\V\left[\mathcal{N}^{\,\flat}\left(\bm{\mathcal{T}};\left(\delta', Q'\right)\right)\right]}}{\vartheta_j^{2}\cdot\delta'\cdot \left(Q'\right)^5}\nonumber\\
&\ll \vartheta_j^{-2}\cdot \left(\frac{1}{\sqrt{\delta'\cdot \left(Q'\right)^5}}+\frac{\log\left(2Q'\right)}{Q'}+\sqrt{\frac{\delta'}{Q'}}\right).\label{majprobacountunif}
\end{align}
The goal is to show that this last quantity is summable when the pair $\left(Q', r'\right)=\left(Q', \delta'/Q'\right)$ varies in the set $\mathfrak{F}_j$ and the integer $j$ over all natural integers. To this end, taking into account the fact that for a fixed value of $Q'\in \mathcal{Q}_j$, the quantity $r'=\delta'/Q'$ varies in a geometric progression with ratio $(1+\theta_j)$, one first obtains that 
\begin{align*}
\sum_{\left(Q', r'\right)\in \mathfrak{F}_j} \! \frac{1}{\sqrt{\delta'\cdot \left(Q'\right)^5}}&\underset{\eqref{deltaQR}}{\le} \sum_{Q'\in \mathcal{Q}_j}\frac{1}{\left(Q'\right)^3}\sum_{r'\in  \mathcal{R}_j}\frac{1}{\sqrt{r'}}\underset{\eqref{progrgeomrj}}{\ll} \frac{1}{j^{1+\eta/2}\cdot\left(1-{\left(1+\vartheta_j\right)^{-\frac{1}{2}}}\right)} \sum_{Q'\in \mathcal{Q}_j}\frac{2^{3j}}{\left(Q'\right)^3}\\
&\underset{\eqref{cardQj}}{\ll}\; \!\frac{1}{\vartheta_j^2\cdot j^{1+\eta/2}}\cdotp
\end{align*}
Moreover, 
\begin{align*}
\sum_{\left(Q', r'\right)\in \mathfrak{F}_j}  \frac{\log\left(2Q'\right)}{Q'}\underset{\eqref{cardQj}}{\le} \frac{j+2}{2^j}\cdot  \#\mathcal{Q}_j\cdot  \#\mathcal{R}_j \underset{\eqref{cardQj} \& \eqref{cardiinclurj}}{\ll} \frac{j^2}{2^j\cdot \vartheta_j^2}
\end{align*}
and the geometric summation over $r'$ yields
\begin{align*}
\sum_{\left(Q', r'\right)\in \mathfrak{F}_j} \sqrt{r'}\underset{\eqref{progrgeomrj}}{\ll} \#\mathcal{Q}_j\cdot \frac{2^{-{j/2}}}{\vartheta_j} \underset{\eqref{cardQj}}{\ll} \frac{2^{-{j/2}}}{\vartheta_j^2}\cdotp 
\end{align*}
As a consequence, putting together the last three estimates, 
\begin{align*}
\sum_{j=0}^\infty \sum_{\left(Q', r'\right)\in \mathfrak{F}_j}  &\bP\left(\left| \mathcal{N}^{\,\flat}\!\left(\bm{\mathcal{T}};\left(\delta', Q'\right)\right)- {\E\left[\mathcal{N}^{\,\flat}\!\left(\bm{\mathcal{T}};\left(\delta', Q'\right)\right)\right]}\right|\ge \vartheta_j c_3^{-1}\sqrt{\delta'\cdot \left(Q'\right)^5}\right)\\
& \underset{\eqref{defvarthetaj} \& \eqref{majprobacountunif}}{\ll }\sum_{j=0}^\infty \frac{1}{j^{1+{\eta/4}}}<\infty. 
\end{align*}
From the Borel--Cantelli Lemma, for all large $j\ge 0$ and all $\left(Q', r'\right)= \left(Q', \delta'/Q'\right)\in \mathfrak{F}_j$, 
\begin{align}
\mathcal{N}^{\,\flat}\!\left(\bm{\mathcal{T}};\left(\delta', Q'\right)\right)&= \E\left[\mathcal{N}^{\,\flat}\!\left(\bm{\mathcal{T}};\left(\delta', Q'\right)\right)\right] +O\left(\vartheta_j \sqrt{\delta'\cdot \left(Q'\right)^5}\right)\nonumber\\
&=  \frac{{4}}{3\sqrt{\pi}}\cdot\left(T_2-T_1\right)\cdot \sqrt{\delta' \cdot\left(Q'\right)^5}+  O\left(\vartheta_j \sqrt{\delta'\cdot \left(Q'\right)^5}\right) \nonumber\\
&\qquad \qquad \qquad\qquad + O \left(\sqrt{\delta' \cdot\left(Q'\right)^3}\cdot \log\left(2Q'\right)+\delta'\cdot \left(Q'\right)^2\right),\label{asympnprdq}
\end{align}
where the last upper bound follows from the first moment estimate for Uniform Counting stated in Proposition~\ref{smeuf}. \\

\noindent The final goal is to remove the restriction that $\left(Q', r'\right)= \left(Q', \delta'/Q'\right)\in \mathfrak{F}_j$ and thus to  pass  to the set of all integers $Q\ge 1$ and all reals $\delta>0$. To this end, given a pair $\left(Q, \delta\right)$, set $r=\delta/Q$. Let $j\ge 0$ be chosen so that  $Q\in  \left[2^j, 2^{j+1}-1\right]$, and choose integers $Q_\pm$ such that the relations~\eqref{sqpm} are met. Similarly, let $r_\pm$ be reals such that $\left(Q_\pm, r_\pm\right)\in \mathfrak{F}_j$, 
\begin{align}\label{srpm}
r_-\le r\le r_+\qquad \textrm{and}\qquad \frac{r_+}{r_-}\underset{\eqref{progrgeomrj}}{=}1+\vartheta_j.
\end{align}
Setting here again $\delta_\pm=Q_\pm r_\pm$, it is immediate from the definition of the counting function $\mathcal{N}^{\,\flat}\!\left(\bm{\mathcal{T}};\left(\delta, Q\right)\right)$ that 
\begin{equation}\label{ineqcoununifbegrid}
\mathcal{N}^{\,\flat}\!\left(\bm{\mathcal{T}};\left(\delta_-, Q_-\right)\right)\;\le\; \mathcal{N}^{\,\flat}\!\left(\bm{\mathcal{T}};\left(\delta, Q\right)\right)\;\le\;  \mathcal{N}^{\,\flat}\!\left(\bm{\mathcal{T}};\left(\delta_+, Q_+\right)\right). 
\end{equation}
Furthermore, from the asymptotic bound~\eqref{asympnprdq}, 
\begin{equation*}
\mathcal{N}^{\,\flat}\!\left(\bm{\mathcal{T}};\left(\delta_\pm, Q_\pm\right)\right)\;=\; \frac{{4}}{3\sqrt{\pi}}\cdot\left(T_2-T_1\right)\cdot \sqrt{\delta_\pm \cdot\left(Q_\pm\right)^5}\cdot\left(1+o(1)\right),
\end{equation*}
where the error term is uniform in the value of the integer $j\ge 0$. From the above definition of the quantities $\delta_\pm$ and from the fact that the sequence $\left(\vartheta_j\right)_{j\ge0}=\left(j^{-\eta/16}\right)_{j\ge0}$ vanishes at infinity, it is then a plain consequence of relations~\eqref{sqpm} and~\eqref{srpm} framing the pair $\left(Q;r\right)$ within suitable points of the grid $\mathfrak{F}_j$ that  inequalities~\eqref{ineqcoununifbegrid} imply that 
\begin{equation*}
\mathcal{N}^{\,\flat}\!\left(\bm{\mathcal{T}};\left(\delta, Q\right)\right)\;=\; \frac{{4}}{3\sqrt{\pi}}\cdot\left(T_2-T_1\right)\cdot \sqrt{{\delta \cdot Q^5}}\cdot\left(1+o(1)\right).
\end{equation*}
This completes the proof of the Asymptotic Uniform Counting Estimate~\eqref{asymunifestim}.
\end{proof}

\noindent The proofs of the remaining claims part of the Second Main Theorem  and of the Auxiliary Theorem rely on the Diagonal and Off--Diagonal Probability Estimates established in Propositions~\ref{diagestimterm} and~\ref{offdiagestimterm}, respectively. To make the link between them and the Uniform Counting function $\mathcal{N}^{\,\flat}\!\left(\bm{\mathcal{T}};\left(\delta, Q\right)\right)$ under consideration, recall from the relation~\eqref{defcountrand1} that, upon considering the   approximation function $\xi_Q^{(\delta)}$  defined in~\eqref{defpsideltaq}, one has that $\mathcal{N}^{\,\flat}\!\left(\bm{\mathcal{T}};\left(\delta, Q\right)\right)= \widehat{\mathcal{N}}_{\xi^{(\delta)}_{Q}}\!\left(\bm{\mathcal{T}}; Q\right)$. Here, $\widehat{\mathcal{N}}_{\xi^{(\delta)}_{Q}}\!\left(\bm{\mathcal{T}}; Q\right)$ is the counting function involved in the statement of the First Main Theorem. \\

\noindent Specialising the definitions of the Diagonal 
Probability Terms $\Delta_{\xi}\left(\bm{\mathcal{Q}}; \bm{\mathcal{T}}\right)$ 
given in~\eqref{DPT} 
to the case of the approximation function $\xi=\xi_Q^{(\delta)}$, 
equation~\eqref{deltaexpec}
makes it clear that 
\begin{align}\label{expdeldtqdel}
\E\left[\mathcal{N}^{\,\flat}\!\left(\bm{\mathcal{T}};\left(\delta, Q\right)\right)\right]\;=\; \Delta^*_{\delta}\left(Q; \bm{\mathcal{T}}\right). 
\end{align}
Here, for the sake of simplicity of notation, one has set   $\Delta^*_{\delta}\left(Q; \bm{\mathcal{T}}\right)=\Delta_{\xi_Q^{(\delta)}}\left(\bm{\mathcal{Q}}; \bm{\mathcal{T}}\right)$ when $\bm{\mathcal{Q}}=\left(1, Q\right)$. 
It then follows from a short calculation based on the conclusion of Proposition~\ref{diagestimterm}  that 
\begin{align}
\E\left[\mathcal{N}^{\,\flat}\!\left(\bm{\mathcal{T}};\left(\delta, Q\right)\right)\right]\;&=\;  \frac{{4}}{3\sqrt{\pi}}\cdot\left(T_2-T_1\right)\cdot\sqrt{\delta Q^5}+ O\left(V\!\left(\bm{\mathcal{T}}\right)\cdot\left(\sqrt{\delta Q^3}\cdot \log\left(Q\right)+\delta Q^2\right)\right) \label{asympexpbis} \\
&=\; \frac{{4}}{3\sqrt{\pi}}\cdot\left(T_2-T_1\right)\cdot\sqrt{\delta Q^5}\cdot\left(1+O\left(V\!\left(\bm{\mathcal{T}}\right)\cdot\left(\frac{\log Q}{Q}+\sqrt{\frac{\delta}{Q}}\right)\right)\right),\label{asympexp}
\end{align}
where the last equation holds under the assumption that $\delta Q^5=\delta(Q)\cdot Q^5\rightarrow \infty$.\\

\begin{proof}[Proof of the Probability Asymptotic Estimate~\eqref{prob1}  valid for Large Values of $\delta Q^5$] The variance of the random variable $\mathcal{N}^{\,\flat}\!\left(\bm{\mathcal{T}};\left(\delta, Q\right)\right)$ is the integral~\eqref{varint}, and the domination condition~\eqref{varintbis} then implies (this is the calculation carried out in the proof of the First Main Theorem in Subsection~\ref{compprFMT}) that 
\begin{align*}
\V\left[\mathcal{N}^{\,\flat}\!\left(\bm{\mathcal{T}};\left(\delta, Q\right)\right)\right]\;\ll\; \Delta^*_{\delta}\left(Q, \bm{\mathcal{T}}\right)\;+\; \Delta^*_{\delta}\left(Q, \bm{\mathcal{T}}\right)^{9(1+3\mu)/5}\;+\; \sum_{q=1}^{Q}\sqrt{q\cdot \xi_Q^{(\delta)}}\cdot \sigma(q)
\end{align*}
for all $\mu\in (0,1/27)$, where $\sigma$ is here again the sum of divisors function. From the above relations~\eqref{expdeldtqdel} and \eqref{asympexp}, and from the equation~\eqref{defMxibis0} stated in Proposition~\ref{offdiagestimterm} to estimate the above sum   involving the divisors function, one obtains that 
\begin{align}\label{estimvaras}
\V\left[\mathcal{N}^{\,\flat}\!\left(\bm{\mathcal{T}};\left(\delta, Q\right)\right)\right]\;\ll\;   \sqrt{\delta Q^5}+\left(\sqrt{\delta Q^5}\right)^{9(1+3\mu)/5}+\sqrt{\delta}\cdot Q^{3/2}\cdot \log Q.
\end{align}

\noindent Fix then $\varepsilon>0$. From Chebyshev's inequality, 
\begin{align*}
\bP\!\left(\left|\mathcal{N}^{\,\flat}\!\left(\bm{\mathcal{T}};\left(\delta, Q\right)\right)- \E\left[\mathcal{N}^{\,\flat}\!\left(\bm{\mathcal{T}};\left(\delta, Q\right)\right)\right]\right|\ge \varepsilon\cdot \E\left[\mathcal{N}^{\,\flat}\!\left(\bm{\mathcal{T}};\left(\delta, Q\right)\right)\right]\right)\le \frac{\V\!\left[\mathcal{N}^{\,\flat}\!\left(\bm{\mathcal{T}};\left(\delta, Q\right)\right)\right]}{ \left(\varepsilon\cdot \E\!\left[\mathcal{N}^{\,\flat}\!\left(\bm{\mathcal{T}};\left(\delta, Q\right)\right)\right]\right)^2}\cdotp
\end{align*}
Under the divergence  assumption that $\delta Q^5=\delta(Q)\cdot Q^5\rightarrow \infty$ as $Q$ tends to infinity, it follows from  relations~\eqref{asympexp} and~\eqref{estimvaras} that the right--hand side vanishes at infinity. This yields the sought Probability Asymptotic Estimate~\eqref{prob1} and thus completes the proof.
 \end{proof}

$\quad$

\noindent It remains to establish the   Uniform Counting Relation~\eqref{asymunifestimbis} and the Probability  Upper Bound~\eqref{prob2},  both concerned with the case where $\delta(Q)\cdot Q^5$ takes small values. \\

\begin{proof}[Proof of the  Uniform Counting relation~\eqref{asymunifestimbis} valid for Small Values of $\delta Q^5$] \sloppy  The sought relation is proved in the following slightly stronger form~: letting for a given $j\ge 1$ 
\begin{equation*}
R_j\;=\; \max_{2^j\le Q<2^{j+1} } \delta(Q)\cdot Q^5,
\end{equation*}
if 
\begin{equation}\label{assumptionconvuncount}
\sum_{j=1}^{\infty}\sqrt{R_j}\;<\; \infty,
\end{equation}
then almost surely for all sufficiently large $Q\ge 1$, 
\begin{equation}\label{eqn=0}
\mathcal{N}^{\,\flat}\!\left(\bm{\mathcal{T}};\left(\delta, Q\right)\right)\;=\; 0. 
\end{equation}
Since when $\delta(Q)\cdot Q^5\ll \left(\log Q\right)^{-2-\eta}$ for some $\eta>0$ the assumption~\eqref{assumptionconvuncount} is met, the Uniform Counting Equation~\eqref{asymunifestimbis} becomes a consequence of this claim. To see that the latter holds, note that when $2^j\le Q<2^{j+1}$ for some $j\ge 1$, 
\begin{equation*}
\delta(Q)\;=\; \frac{\delta(Q)\cdot Q^5}{Q^5}\;\le\; \frac{R_j}{2^{5j}}\cdotp
\end{equation*}
As a consequence, if $\mathfrak{A}_j$ denotes the event that there exists some integer $Q\in\left[2^j, 2^{j+1}-1\right]$ for which $\mathcal{N}^{\,\flat}\!\left(\bm{\mathcal{T}};\left(\delta, Q\right)\right)\ge 1$, then, clearly, 
\begin{align*}
\bP\!\left(\mathfrak{A}_j\right)\;\le\;  \bP\!\left(\mathcal{N}^{\,\flat}\!\left(\bm{\mathcal{T}};\left({\frac{2R_j}{2^{5j}}}, 2^{j+1}\right)\right)\ge 1\right).
\end{align*} 
By Markov's Inequality, this yields that 
\begin{align*}
\sum_{j\ge 1} \bP\left(\mathfrak{A}_j\right)\;{\ll}\;  \sum_{j\ge 1}\E\left[\mathcal{N}^{\,\flat}\!\left(\bm{\mathcal{T}};\left(\frac{2 R_j}{2^{5j}}, 2^{j+1}\right)\right)\right] \;  \underset{\eqref{expdeldtqdel} \& \eqref{asympexpbis}}{\ll} \;  \sum_{j\ge 1} \sqrt{R_j}.
\end{align*}
When this last sum converges, the Borel--Cantelli Lemma implies that the events $\left(\mathfrak{A}_j\right)_{j\ge 1}$ occur only finitely often, whence the equation~\eqref{eqn=0}. This concludes the proof.  
\end{proof}

$\quad $

\begin{proof}[Proof of the Probability  Upper Bound~\eqref{prob2} valid for Small Values of $\delta Q^5$]
This is an immediate consequence of Markov's Inequality since it implies that 
\begin{equation*}
\bP\left(\mathcal{N}^{\,\flat}\!\left(\bm{\mathcal{T}};\left(\delta, Q\right)\right)\ge 1\right)\;\le\; \E\left[\mathcal{N}^{\,\flat}\!\left(\bm{\mathcal{T}};\left(\delta, Q\right)\right)\right] \; \underset{\eqref{expdeldtqdel} \& \eqref{asympexpbis}}{\ll} \; \left(T_2-T_1\right)\cdot \sqrt{\delta Q^5}. 
\end{equation*}
The claim follows under the assumption that the quantity $\delta Q^5=\delta(Q)\cdot Q^5$ vanishes at infinity.
\end{proof}

\section{On the Probability that a Brownian Trajectory Hits a Rectangle}\label{probsquare}

\noindent This section establishes the main auxiliary statement needed to prove Propositions~\ref{diagestimterm}  (concerned with estimating the diagonal probability term), Propositon~\ref{offdiagestimterm} (concerned with estimating  the off--diagonal probability term) and Proposition~\ref{smeuf} (concerned with moment estimates for uniform counting). These three propositions  are then established  in Sections~\ref{asymdiag}, \ref{asymoffdiag} and~\ref{secfin}, respectively. \\

\noindent All these proofs rely on very subtle distinctions of regimes and lengthy calculations which are detailed whenever needed but not when the arguments are repeated.  Throughout this section, fix a rectangle $\mathcal{R}$  in the plane such that, in the  notation introduced in Subsection~\ref{sbscMVLPE}, one has that $x^-_{\mathcal{R}}=x^-\ge 0$. Set for the sake of simplicity of notation\textsuperscript{8}\let\thefootnote\relax\footnotetext{\textsuperscript{8}Recall that by definition, a rectangle is understood as a set with nonempty interior. As a consequence, $x^-<x^+$ and $y^-<y^+$.} 
\begin{equation}\label{defalphabeta}
\alpha\;=\; \alpha_{\mathcal{R}}\;=\; x^+-x^->0\qquad\qquad \textrm{ and }\qquad\qquad \beta\;=\; \beta_{\mathcal{R}}\;=\;  y^+-y^- >0.
\end{equation}

\noindent The \emph{complementary error function} 
\begin{equation}\label{erfc}
\erfc  : x\in\R\;\mapsto\; \frac{2}{\sqrt{\pi}}\cdot\int_{x}^{\infty}e^{-t^2}\cdot\textrm{d}t
\end{equation}
plays a crucial role in the various identities and estimations established hereafter. Note that $$\erfc(0)=1\qquad\qquad  \textrm{ and }\qquad\qquad  \erfc(\infty)\; :=\; \lim_{x\rightarrow\infty} \erfc(x)\;=\; 0.$$ Recall also the definition of the parametric family of Gaussian density functions $\left\{p_t\right\}_{t> 0}$ introduced in~\eqref{def_p}. It is extended to the case where $t=0$ upon defining $p_0$ as the Dirac mass at the origin. This is saying that for any map $f~: \R\mapsto\R$, $$\int_\R p_0\cdot f\;=\; f(0).$$

\noindent  A fraction such as $y/x$ is understood to take the value $\infty$ when $y> 0$ and $x=0$ (this is consistent with the fact that only nonnegative values of the denominators of such fractions are considered in this section). 

\subsection{The Main Estimates}

\noindent The main results established in this section read as follows~: 

\begin{prop}[Probability of Hitting a Rectangle]\label{proppasssq}  Recall that  $\mathcal{R}$ stands for a rectangle in the plane such that  $x^-_{\mathcal{R}}=x^-\ge 0$. Then, assuming first that $y^+\ge 0$, the hitting probability $\ppi\left(\mathcal{R}\right)$ is given by the closed--form identity 
\begin{align*}
\ppi\left(\mathcal{R}\right)
\;=\; \erfc\left(\frac{\max\{y^-, 0\}}{\sqrt{2x^+}}\right)\\
+\;  \int_{0}^{\infty} p_{\alpha}(u)\cdot &\left(\textrm{\emph{sgn}}\left(y^--u\right)\cdot \erfc\left(\frac{\left|y^--u\right|}{\sqrt{2x^-}}\right) -\erfc\left(\frac{\left|y^++u\right|}{\sqrt{2x^-}}\right)\right)\cdot \textrm{d}u, \label{compactestimprob2}
\end{align*}
where $\textrm{\emph{sgn}}$ stands for the sign function. This identity  remains valid when $y^+\le 0$ upon switching the role of the parameters $y^{\pm}$ according to the substitution $\left(y^-, y^+\right)\mapsto\left(-y^+, -y^-\right)$.
\end{prop}

\noindent Although compact and self--contained, the above formula does not make it clear the way the probability $\ppi\left(\mathcal{R}\right)$ behaves as a function of the various parameters it involves. This is made explicit in the following statement where, for the sake of simplicity, one adopts the short--hand notation 
\begin{equation}\label{defak}
A\!\left(a,x\right)\;=\;\sqrt{\frac{a}{x}}\quad \textrm{and}\quad K(u,v)\;=\; \frac{u^2}{2v}\quad \textrm{when}\quad a,u\ge 0\quad \textrm{and}\quad x,v>0.
\end{equation}
It is also convenient to set
\begin{itemize}
\item $E(x,y)= \exp\left(-K(y,x)\right)$ when $y\ge 0$ and  $x>0$;
\item  $F\left(a, x,s,y\right)=\min\left\{\pi/2, \;A\left(a, x\right),\; \sqrt{\pi/2}\cdot\sqrt{s}/y\right\}$ when  $a,s\ge 0$ and  $x,y>0$;
\item $G\left(\beta, x, y\right)=\min\left\{\beta/\sqrt{2x}, \;2/\left(\sqrt{\pi}\cdot\left(1+\sqrt{K(y,x)}\right)\right)\right\}$ when $\beta, y\ge0$ and $x>0$;
\item $H(a, x)=\arctan A(a,x)-A(a,x)$ when $a\ge 0$ and $x>0$.
\end{itemize}

\begin{prop}[Estimation of the  Probability of Hitting a Rectangle]\label{proppasssqbis} Set for the sake of simplicity of notation $\bm{x}=\left(x^-, x^+\right)$, $\bm{y}=\left(y^-, y^+\right)$ and $\bm{\varphi}=\left(\alpha, \beta\right)$. The probability $\ppi\left(\mathcal{R}\right)$ can then be estimated according to this distinction of cases~: 

 \paragraph{$\bullet $ Case~1.} Under the assumption that $x^->0$ and $y^->0$, 
\begin{align}\label{genboundcase1}
\ppi\left(\mathcal{R}\right)\;\le\; \frac{E\left(x^+, y^-\right)}{\pi}\cdot\left(2\cdot F\left(\alpha, x^-, x^+, y^-\right)+ \frac{\pi}{2}\cdot G\left(\beta, x^+, y^-\right)\right) + E\left(x^-, y^-\right)\cdot G\left(\beta, x^-, y^-\right).
\end{align}
This bound can be put in the more refined form 
\begin{align}\label{ppi1a}
\ppi\left(\mathcal{R}\right)\;=\; \frac{E\left(x^+, y^-\right)}{\pi}\cdot U_{\sigma, \tau}\left(\bm{x}, y^-, \bm{\varphi}\right)\;&+\; \frac{E\left(x^+, y^+\right)}{\pi}\cdot V_{\theta}\left(\bm{x}, y^+, \alpha\right)\nonumber\\
&+\; \frac{E\left(x^-, y^-\right)}{2}\cdot  W_{\zeta}\left(x^-, y^-, \beta\right)
\end{align}
for some parameters $$\sigma=\sigma\left(\mathcal{R}\right),\quad \tau=\tau\left(\mathcal{R}\right),\quad  \theta=\theta\left(\mathcal{R}\right)\in \left[0,1\right]\qquad \textrm{and}\qquad \zeta=\zeta\left(\mathcal{R}\right)\in \left[0,2\right]$$ in the following regimes, which determine the values of the three functions $U_{\sigma, \tau}\left(\bm{x}, y^-, \bm{\varphi}\right)$, $V_{\theta}\left(\bm{x}, y^+, \alpha\right)$ and $W_{\zeta}\left(x^-, y^-, \beta\right)$~: 
\begin{itemize}
\item[$\textbf{1.(a)}$] when $K\left(y^+, x^+\right)\cdot A^2\left(\alpha, x^-\right)\le 1$, one can set
\begin{align*}
&\Diamond\;  U_{\sigma, \tau}\left(\bm{x}, y^-, \bm{\varphi}\right)=\arctan A\left(\alpha, x^-\right)+\sigma\cdot K\left(y^-, x^+\right)\cdot H\left(\alpha, x^-\right)+\frac{\pi\tau}{2}\cdot G\left(\beta, x^+, y^-\right),\\
&\Diamond\;  V_{\theta}\left(\bm{x}, y^+, \alpha\right)\;=\; \arctan A\left(\alpha, x^-\right)+\theta\cdot K\left(y^+, x^+\right)\cdot H\left(\alpha, x^-\right)\qquad \textrm{and}\\
& \Diamond\;  W_{\zeta}\left(x^-, y^-, \beta\right)\;=\;\zeta\cdot G\left(\beta, x^-, y^-\right).
\end{align*}
Under the additional assumption that $A\left(\alpha, x^-\right)< 1$, the values 
\begin{align*}
&\star U_{\sigma, \tau}\left(\bm{x}, y^-, \bm{\varphi}\right)\;=\\
&\quad A\left(\alpha, x^-\right)\cdot\left(1-\frac{\sigma}{3}\cdot A^2\left(\alpha, x^-\right)\cdot\left(1+K\left(y^-, x^+\right)\right)\right)+\tau\cdot \frac{\pi}{2}\cdot G\left(\beta, x^+, y^-\right)\quad \textrm{and}\\
&\star V_{\theta}\left(\bm{x}, y^+, \alpha\right)\;=\;   A\left(\alpha, x^-\right)\cdot \left(1-\frac{\theta}{3}\cdot A^2\left(\alpha, x^-\right)\cdot\left(1+K\left(y^+, x^+\right)\right)\right)
\end{align*}
are also admissible. 

\item[$\textbf{1.(b)}$]  when $K\left(y^-, x^+\right)\cdot A^2\left(\alpha, x^-\right)\le 1\le K\left(y^+, x^+\right)\cdot A^2\left(\alpha, x^-\right)$, one can set
\begin{align*}
&\Diamond\;  U_{\sigma, \tau}\left(\bm{x}, y^-, \bm{\varphi}\right)=\arctan A\left(\alpha, x^-\right)+\sigma\cdot K\left(y^-, x^+\right)\cdot H\left(\alpha, x^-\right)+\frac{\pi\tau}{2}\cdot G\left(\beta, x^+, y^-\right),\\
&\Diamond\;  V_{\theta}\left(\bm{x}, y^+, \alpha\right)\;=\; \theta\cdot  F\left(\alpha, x^-, x^+, y^+\right)\qquad \textrm{and}\\
&\Diamond\;  W_{\zeta}\left(x^-, y^-, \beta\right)\;=\;\zeta\cdot G\left(\beta, x^-, y^-\right).
\end{align*}
Under the additional assumption that $A\left(\alpha, x^-\right)< 1$, the value
\begin{align*}
&\star U_{\sigma, \tau}\left(\bm{x}, y^-, \bm{\varphi}\right)\;=\\ 
&\qquad  A\left(\alpha, x^-\right)\cdot\left(1-\frac{\sigma}{3}\cdot A^2\left(\alpha, x^-\right)\cdot\left(1+K\left(y^-, x^+\right)\right)\right)+\tau\cdot \frac{\pi}{2}\cdot G\left(\beta, x^+, y^-\right) 
\end{align*}
is also admissible
\end{itemize}

\paragraph{$\bullet $ Case~2.} The cases where either or both of the coordinates $x^-$ or $y^-$ vanish can be dealt with as follows~: 
\begin{itemize}
\item[$\textbf{2.(a)}$] under the assumption that $x^->0$ and $y^-=0$, the estimates of Case~1 hold true upon setting $y^+=\beta$ and upon taking the limit $y^-\rightarrow 0^+$;
\item[$\textbf{2.(b)}$] under the assumption that $x^-=0$ and $y^->0$, the estimates of Case~1 hold true upon setting $x^+=\alpha$ and upon taking the limit $x^-\rightarrow 0^+$;
\item[$\textbf{2.(c)}$] under the assumption that $x^-=y^-=0$, one has that $\ppi\left(\mathcal{R}\right)=1$.
\end{itemize}

\paragraph{$\bullet $ Case~3.} Under the assumption that $x^-\ge 0$ and $y^+\le 0$, one has that $\ppi\left(\mathcal{R}\right)=\ppi\left(-\mathcal{R}\right)$, where the probability $\ppi\left(-\mathcal{R}\right)$ can be estimated from Cases~1 and~2. (Here, $-\mathcal{R}$ denotes the reflection of the rectangle $\mathcal{R}$ with respect to the abscissa axis $\left\{y=0\right\}$ in the $Oty$--plane.)

\paragraph{$\bullet $ Case~4.} The case where $y^-\le 0\le y^+$ can be dealt with as follows~:
\begin{itemize}
\item[$\textbf{4.(a)}$] under the assumption that $x^->0$, decompose the rectangle $\mathcal{R}$ into the two subrectangles $\mathcal{R}^+=\mathcal{R}\cap \left\{y\ge 0\right\}$ and $\mathcal{R}^-=\mathcal{R}\cap \left\{y\le 0\right\}$ in the $Oty$--plane. Then, $$\ppi\left(\mathcal{R}\right)\;=\; \ppi\left(\mathcal{R}^+\right)+\ppi\left(\mathcal{R}^-\right)-\left(1-\frac{2}{\pi}\cdot\arcsin\sqrt{\frac{x^-}{x^+}}\right), $$ where the probabilities $\ppi\left(\mathcal{R}^+\right)$ and $\ppi\left(\mathcal{R}^-\right)$ can be estimated from Cases~1, 2 and~3.
\item[$\textbf{4.(b)}$] under the assumption that $x^-=0$, one has that $\ppi\left(\mathcal{R}\right)=1$
\end{itemize} 
\end{prop}
 
 \noindent This proposition  makes it clear that the key quantities controlling the probability $\ppi\left(\mathcal{R}\right)$ at fine scales are, on the one hand, the product $K\left(y^+, x^+\right)\cdot A^2\left(\alpha, x^-\right)$ and, on the other, the real $A\left(\alpha, x^-\right)$ (restricting to Case~1 without loss of generality). This leads one to this simplified statement providing the order of magnitude of $\ppi\left(\mathcal{R}\right)$ when $\alpha, \beta\rightarrow 0^+$ and when the coordinates $x^-$ and $y^-$ are held fixed. This is the form of estimates needed to check the validity of the assumptions in the Mean--Variance Lemma~\ref{ebc}.
 
 \begin{coro}[Asymptotics of the hitting probability as $\alpha, \beta\rightarrow 0^+$]\label{coroutile} Assume that $x^-, y^->0$. The probability $\ppi\left(\mathcal{R}\right)$ can then be bounded as 
 \begin{align}\label{mainupbprob}  
 \ppi\left(\mathcal{R}\right)\ll \exp\left(-\frac{\left(y^-\right)^2}{2x^+}\right)\cdot \left(\min\left\{1, \;\sqrt{\frac{\alpha}{x^-}}, \frac{\sqrt{x^+}}{y^-}\right\}+\frac{\beta}{\sqrt{x^+}}\right)+ \exp\left(-\frac{\left(y^-\right)^2}{2x^-}\right)\cdot \frac{\beta}{\sqrt{x^-}}\cdotp
 \end{align}
 Under the additional assumption that 
 \begin{equation}\label{refcondi}
 0\; <\; \alpha\; <\; x^- \qquad \textrm{and that }\qquad  0\; <\; y^-\; <\;  -\beta+\sqrt{\frac{2x^-\left(\alpha+x^-\right)}{\alpha}},
 \end{equation}
 one has the more refined estimate 
 \begin{align}
 &\left|\ppi\left(\mathcal{R}\right)-\frac{2}{\pi}\cdot \sqrt{\frac{\alpha}{x^-}}\cdot\exp\left(-\frac{\left(y^-\right)^2}{2x^+}\right)\right|\nonumber\\
 &\ll\;   \exp\left(-\frac{\left(y^-\right)^2}{2x^+}\right)\cdot \left(\sqrt{\frac{\alpha}{x^-}}\cdot \min\left\{1, \beta\cdot \frac{y^+}{x^+}\right\}+\left(\frac{\alpha}{x^-}\right)^{3/2}\cdot\max\left\{1, \frac{\left(y^-\right)^2}{2x^+}\right\}+\frac{\beta}{\sqrt{2x^+}}\right)\nonumber\\
 &\quad + \;\exp\left(-\frac{\left(y^-\right)^2}{2x^-}\right)\cdot \frac{\beta}{\sqrt{2x^-}}\cdotp
 \label{refestimcoro} 
 \end{align}
The remaining cases where $x^\pm=0$, $y^\pm=0$, $y^+\le0$ or $y^-\le 0 \le y^+$ are covered in the same way as in Cases~2, 3 and~4 of Proposition~\ref{proppasssqbis}.
 \end{coro}

\begin{proof}[Deduction of Corollary~\ref{coroutile} from Proposition~\ref{proppasssqbis}] The inequality~\eqref{genboundcase1} implies that 
\begin{align*}
\ppi\left(\mathcal{R}\right)\;\ll\;  &\exp\left(-\frac{\left(y^-\right)^2}{2x^+}\right)\cdot \left(\min\left\{1, \;\sqrt{\frac{\alpha}{x^-}}, \frac{\sqrt{x^+}}{y^-}\right\}+\min\left\{\frac{\beta}{\sqrt{2x^+}}, \frac{1}{1+\frac{y^-}{\sqrt{2x^+}}}\right\}\right)\\
&+ \exp\left(-\frac{\left(y^-\right)^2}{2x^-}\right)\cdot \min\left\{\frac{\beta}{\sqrt{2x^-}}, \frac{1}{1+\frac{y^-}{\sqrt{2x^-}}}\right\},
\end{align*}
whence the first claim. To deal with the more refined estimate~\eqref{refestimcoro}, consider the situation where the inequalities $K\left(y^+, x^+\right)\cdot A^2\left(\alpha, x^-\right)\le 1$ and $A\left(\alpha, x^-\right)< 1$ hold simultaneously. A short calculation shows that they are equivalent to claiming that 
\begin{align*}
\alpha\; <\; x^-\qquad \textrm{and}\qquad \alpha\cdot\left(\beta^2+2\beta y^-+\left(y^-\right)^2-2x^-\right)\; <\; 2\left(x^-\right)^2.
\end{align*}
Solving this system in $y^->0$ translates into the condition~\eqref{refcondi}. The corresponding refined estimate in Case~1.(a) then implies that 
\begin{align*}
&\left| \ppi\left(\mathcal{R}\right)-\frac{2}{\pi}\cdot \exp\left(-\frac{\left(y^-\right)^2}{2x^+}\right)\cdot \sqrt{\frac{\alpha}{x^-}} \right|\;\ll\; \sqrt{\frac{\alpha}{x^-}}\cdot\left( \exp\left(-\frac{\left(y^-\right)^2}{2x^+}\right)- \exp\left(-\frac{\left(y^+\right)^2}{2x^+}\right)\right)\\
&+\;\left(\frac{\alpha}{x^-}\right)^{3/2}\cdot \exp\left(-\frac{\left(y^-\right)^2}{2x^+}\right)\cdot\left(1+\frac{\left(y^-\right)^2}{2x^+}\right)\;+\; \exp\left(-\frac{\left(y^-\right)^2}{2x^+}\right)\cdot \frac{\beta}{\sqrt{2x^+}}\\
& +\;\left(\frac{\alpha}{x^-}\right)^{3/2}\cdot \exp\left(-\frac{\left(y^+\right)^2}{2x^+}\right)\cdot\left(1+\frac{\left(y^+\right)^2}{2x^+}\right)\;+\; \exp\left(-\frac{\left(y^-\right)^2}{2x^-}\right)\cdot \frac{\beta}{\sqrt{2x^-}}\cdotp
\end{align*}
From the Mean Value Theorem, the first term is bounded as 
\begin{align*}
\sqrt{\frac{\alpha}{x^-}}\cdot\left( \exp\left(-\frac{\left(y^-\right)^2}{2x^+}\right)- \exp\left(-\frac{\left(y^+\right)^2}{2x^+}\right)\right)\ll \sqrt{\frac{\alpha}{x^-}}\cdot \exp\left(-\frac{\left(y^-\right)^2}{2x^+}\right)\cdot \min\left\{1, \beta\cdot\frac{y^+}{x^+}\right\}.
\end{align*}
This completes the deduction of the corollary from Proposition~\ref{proppasssqbis}.
\end{proof}

\noindent The remainder of this section is devoted to the proofs of Propositions~\ref{proppasssq} and ~\ref{proppasssqbis}, which  break into a succession of lemmata. 
The results established in all the lemmata are put together at the end of the section to derive the two propositions 
from them.

\subsection{Decomposition of the Probability of Hitting  a Rectangle}

\begin{lem}[Decomposition of the Probability]\label{decompoprob1} The hitting probability $\ppi\left(\mathcal{R}\right)$ can be expressed as 
\begin{equation*}
\ppi\left(\mathcal{R}\right)\;=\; \mu\left(\mathcal{R}\right)+\nu\left(\mathcal{R}\right)+\lambda\left(\mathcal{R}\right),
\end{equation*} 
where 
\begin{equation*}
 \mu\left(\mathcal{R}\right)\;=\; \int_{y^-}^{y^+}p_{x^-}(u)\cdot\textrm{d}u,\qquad \qquad 
 \nu\left(\mathcal{R}\right)\;=\; 2\cdot\int_{-\infty}^{y^-}p_{x^-}(u)\cdot\int_{y^--u}^{\infty}p_{\alpha}(v)\cdot\textrm{d}v\cdot\textrm{d}u
\end{equation*}
and 
 \begin{equation*}
 \lambda\left(\mathcal{R}\right)\;=\; 2\cdot\int^{+\infty}_{y^+}p_{x^-}(u)\cdot\int_{u-y^+}^{\infty}p_{\alpha}(v)\cdot\textrm{d}v\cdot\textrm{d}u.
\end{equation*}
\end{lem}  

\begin{proof}
Let $E(\mathcal{R})$ be the event that a Brownian motion passes through the rectangle $\mathcal{R}$~: $$E(\mathcal{R})\;=\; \left\{\omega\in\Omega\;\middle|\; \exists t\in\left[x^-, x^+\right]\;:\; B(t)\in \left[y^-, y^+\right] \right\}.$$ The event $E(\mathcal{R})$ can be decomposed as a pairwise disjoint union of three sub-events~:  $$E(\mathcal{R})\;=\; E_1(\mathcal{R})\cup E_2(\mathcal{R})\cup E_3(\mathcal{R}).$$ Here, 
\begin{itemize}

\item $E_1(\mathcal{R})$ is the sub-event of $E(\mathcal{R})$ defined by the additional restriction that $B(x^-)\in\left[y^-, y^+\right]$. Its probability is plainly given by the quantity $\mu(\mathcal{R})$;

\item $E_2(\mathcal{R})$ is the sub-event of $E(\mathcal{R})$ defined by the additional restriction that $B(x^-)<y^-$. To compute its probability, note first that when $B(x^-)<y^-$,  the trajectory of Brow\-nian motion hits the rectangle $\mathcal{R}$ if, and only if, its maximum over the interval $\left[x^-, x^+\right]$ is at least $y^-$. Fix then a real number $u<y^-$ and consider the event $A(u)=\left\{\omega\in\Omega\;\middle|\; B(x^-)=u\right\}$. From the Law of Increments (P3) and the property of Independence of Increments (P4), it is easily seen that the conditional pro\-ba\-bility $\bP\left(E_2(\mathcal{R})\;\middle|\; A(u)\right)$ equals the probability that the maximum of a Brownian motion should be at least $y^--u$ over the interval $\left[0, \alpha\right]$ (where $\alpha$ is the quantity defined  in~\eqref{defalphabeta}). From the Law of Maximum, one thus obtains that $$\bP\left(E_2(\mathcal{R})\;\middle| \; A(u)\right)\;=\; 2\cdot \int_{y^--u}^{\infty}p_{\alpha}(v)\cdot\textrm{d}v.$$ To deconditionate the probability, integrate the above quantity with respect to the density probability function of the event $A(u)$, namely with respect to the map $u\mapsto p_{x^-}(u)$ (seen as a Dirac distribution when $x^-=0$) over the admissible values $u<y^-$. This gives the stated expression for the probability $\nu(\mathcal{R})$.

\item $E_3(\mathcal{R})$ is the sub-event of $E(\mathcal{R})$ defined by the additional restriction that $B(x^-)>y^+$. Conditionally on this event, the trajectory of Brow\-nian Motion  hits the rectangle $\mathcal{R}$ if, and only if, its minimum over the interval $\left[x^-, x^+\right]$  is at most $y^+$. 
The argument to compute the probability of the sub--event $E_3(\mathcal{R})$  is then the same as the one developed in the previous case upon relying on the Law of Maximum (which, as stated in~\eqref{Mt}, also deals with the probability distribution of the minimum of a Brownian motion).
\end{itemize}
This completes the proof of the lemma.
\end{proof}

\noindent When $x^->0$, the estimation of the probabilities $ \nu\left(\mathcal{R}\right)$ and  $\lambda\left(\mathcal{R}\right)$ introduced in Lemma~\ref{decompoprob1} requires a refined analysis developed over the next few statements. To this end, upon making suitable changes of variables, integration is performed in several places over the domain 
\begin{equation}\label{dedomd1}
\mathfrak{D}\left(x^-, y^-, \alpha\right)\;=\; \mathfrak{D}_\ell\!\left(x^-, y^-, \alpha\right)\;\cup\; \mathfrak{D}_r\!\left(x^-, y^-, \alpha\right),
\end{equation} 
where 
\begin{equation*}\label{dedomdl}
\mathfrak{D}_\ell\!\left(x^-, y^-, \alpha\right)\;=\;\left\{\left(u,v\right)\in\R^2\;:\; u\;\le\; \frac{y^-}{\sqrt{2\cdot x^-}}, \quad \frac{y^--u\cdot\sqrt{2\cdot x^-}}{\sqrt{2\cdot \alpha}}\;\le\; v\right\}
\end{equation*} 
and
\begin{equation*}\label{dedomdr}
\mathfrak{D}_r\!\left(x^-, y^-, \alpha\right)\;=\; \left\{\left(u,v\right)\in\R^2\;:\;  \frac{y^+}{\sqrt{2\cdot x^-}}\;\le\; u, \quad \frac{u\cdot\sqrt{2\cdot x^-}-y^+}{\sqrt{2\cdot \alpha}}\;\le\; v\right\}.
\end{equation*} 
\noindent The domain $\mathfrak{D}\left(x^-, y^-, \alpha\right)$ with its left and right components $\mathfrak{D}_\ell\!\left(x^-, y^-, \alpha\right)$ and $\mathfrak{D}_r\!\left(x^-, y^-, \alpha\right)$ is represented in Figure~\ref{figure_of_D1}. 

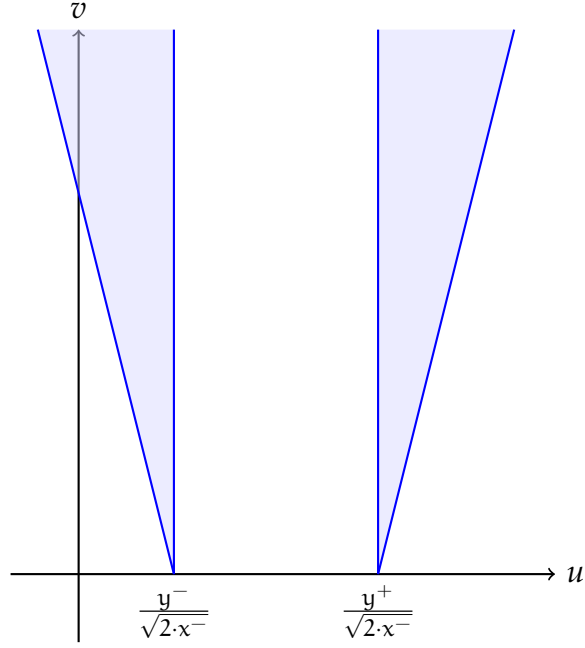
\begin{figure}[h!]
    \centering
\begin{tikzpicture}[scale=1.8]  
    \draw[thick,->] (-0.5,0) -- (3.5,0) node[right] {\textit{u}};
    \draw[thick,->] (0,-0.5) -- (0,4) node[above] {\textit{v}};

    \fill[blue!15,opacity=0.5] (0.7,0) -- (0.7,4) -- (-0.3,4) -- cycle;

    \fill[blue!15,opacity=0.5] (2.2,0) -- (2.2,4) -- (3.2,4) -- cycle;

    \draw[thick,blue] (0.7,0) -- (0.7,4);
    \draw[thick,blue] (2.2,0) -- (2.2,4);

    \draw[thick,blue] (0.7,0) -- (-0.3,4);  
    \draw[thick,blue] (2.2,0) -- (3.2,4);   

    \node[below] at (0.7,0) {$\frac{y^-}{\sqrt{2\cdot x^-}}$};
    \node[below] at (2.2,0) {$\frac{y^+}{\sqrt{2\cdot x^-}}$};
    \end{tikzpicture}
    
    \caption{Domain $\mathfrak{D}\left(x^-, y^-, \alpha\right)$ with its left and right components $\mathfrak{D}_\ell\!\left(x^-, y^-, \alpha\right)$ and $\mathfrak{D}_r\!\left(x^-, y^-, \alpha\right)$. Either or both of the quantities $\frac{y^-}{\sqrt{2\cdot x^-}}$ and $\frac{y^+}{\sqrt{2\cdot x^-}}$ may be negative.}
    \label{figure_of_D1}

\end{figure}

\begin{lem}[Polar-type Representations for the Probabilities $\nu(\mathcal{R})$ and $\lambda(\mathcal{R})$ when $x^->0$]\label{lemnulambda11}  The proba\-bilities $\nu(\mathcal{R})$ and $\lambda(\mathcal{R})$   introduced in Lemma~\ref{decompoprob1} admit the following alternative representations depending on  the domain $\mathfrak{D}\left(x^-, y^-, \alpha, \beta\right)$ defined in~\eqref{dedomd1} when $x^->0$~: 
\begin{equation*}
\nu(\mathcal{R})\;=\; \frac{2}{\pi}\cdot \underset{\mathfrak{D}_\ell\left(x^-, y^-, \alpha\right)}{\int\int} e^{-u^2-v^2}\cdot \textrm{d}u\cdot \textrm{d}v\qquad \textrm{and}\qquad \lambda(\mathcal{R})\;=\; \frac{2}{\pi}\cdot \underset{\mathfrak{D}_r\left(x^-, y^-, \alpha\right)}{\int\int} e^{-u^2-v^2}\cdot \textrm{d}u\cdot \textrm{d}v.
\end{equation*}
In particular, 
\begin{equation*}
\nu(\mathcal{R})+\lambda(\mathcal{R})\;=\; \frac{2}{\pi}\cdot \underset{\mathfrak{D}\left(x^-, y^-, \alpha\right)}{\int\int} e^{-u^2-v^2}\cdot \textrm{d}u\cdot \textrm{d}v.
\end{equation*}
\end{lem}

\noindent These easily--obtained concise expressions considerably simplify the calculations involved in the proof of Proposition~\ref{proppasssqbis}.

\begin{proof}
From Lemma~\ref{decompoprob1}, 
 \begin{align*}
 \nu\left(\mathcal{R}\right)\;&=\; 2\cdot\int_{-\infty}^{y^-}p_{x^-}(u)\cdot\int_{y^--u}^{\infty}p_{\alpha}(v)\cdot\textrm{d}v\cdot\textrm{d}u\\
 &\underset{\eqref{def_p}}{=}\; \frac{1}{\pi\cdot\sqrt{x^-\cdot\alpha}} \int_{-\infty}^{y^-}\int_{y^--u}^{\infty}\exp\left(-\frac{u^2}{2\cdot x^-}-\frac{v^2}{2\cdot\alpha}\right)\cdot\textrm{d}v\cdot\textrm{d}u.
\end{align*}
The expression for $ \nu\left(\mathcal{R}\right)$ then follows from the change of variables $u\mapsto u/\sqrt{2\cdot x^-}$ and $v\mapsto v/\sqrt{2\alpha}$. The same change of variables in the integrals defining $ \lambda(\mathcal{R})$ also yields the expression stated  for this quantity.
\end{proof} 

\noindent The polar--type representations for the probabilities $\nu(\mathcal{R})$ and $\lambda(\mathcal{R})$ given in the above lemma yield an explicit form of these quantities when $x^->0$ and $y^->0$ . This is achieved in the next two statements with the help of \emph{Owen's T--function}~\cite{O} 
\begin{equation}\label{owen}
T~: (a,k)\in \left(\R_{\ge 0}\right)^2 \;\mapsto\; \int_0^a\frac{\exp\left(-kt^2\right)}{1+t^2}\cdot\textrm{d}t.
\end{equation}  
Their proofs both rely on the so--called Craig's Formula~\cite[p.271]{St} claiming that for any $k>0$, 
\begin{equation}\label{craig}
\erfc\left(\sqrt{k}\right)\;=\; \frac{2}{\pi}\cdotp\int_0^{\pi/2}\exp\left(-\frac{k}{\cos^2\theta}\right)\cdot \textrm{d}\theta.
\end{equation}

\begin{lem}[Decomposition of the Probability $\nu(\mathcal{R})$  when $x^->0$ and $y^->0$]\label{decnuR1}
Assume that $x^->0$ and $y^->0$. Then, the probability  $\nu(\mathcal{R})$  defined in Lemma~\ref{decompoprob1} can be expressed as 
\begin{align*}
\nu(\mathcal{R})\;=\; \frac{1}{\pi}\cdot\exp\left(-K\left(y^-, x^+\right)\right)\cdot T\left(A\left(\alpha,x^-\right), K\left(y^-, x^+\right)\right)\\
+\frac{1}{2}\cdot\left(\erfc\left(\sqrt{K\left(y^-, x^+\right)}\right)-\erfc\left(\sqrt{K\left(y^-, x^-\right)}\right)\right).
\end{align*}
Here, the functions $A$ and $K$ are those introduced in~\eqref{defak}.
\end{lem}

\begin{proof}
The quantity $\nu(\mathcal{R})$ is given in the Polar--type Representation Lemma~\ref{lemnulambda11}  as an integral over the domain $\mathfrak{D}_\ell\!\left(x^-, y^-, \alpha\right)$ introduced  in~\eqref{dedomd1}. This domain (see  Figure~\ref{figure_of_D1}) is bounded by the two lines $$\left(\mathcal{L}_1\right)~: y\;=\; \frac{y^--x\cdot\sqrt{2x^-}}{\sqrt{2\alpha}}\qquad \textrm{and}\qquad \left(\mathcal{L}_2\right)~: x\;=\; \sqrt{K\left(y^-, x^-\right)}\cdot$$ In the polar coordinates $\left(x,y\right)=\left(r\cdot \cos\theta, r\cdot \sin\theta\right)$, the equations of these lines are $$\left(\mathcal{L}_1\right): r\;=\;r_1\left(\theta\right)\;=\; \frac{y^-}{\sqrt{2\alpha}\cdot \sin\theta+\sqrt{2x^-}\cos\theta}\;=\;   \frac{\sqrt{K\left(y^-, x^+\right)}}{\cos\left(\theta-\tau(\alpha, x^-)\right)}$$ and $$ \left(\mathcal{L}_2\right): r\;=\; r_2\left(\theta\right)\;=\; \frac{\sqrt{K\left(y^-, x^-\right)}}{\cos\theta}$$ respectively, where $$\tau(\alpha, x^-)\;=\; \arctan \; A(\alpha, x^-).$$ Under the assumption that $y^->0$,  the integral expressing $\nu(\mathcal{R})$ in  the Polar--type  Representation Lemma~\ref{lemnulambda11} then becomes 
\begin{align}
\nu(\mathcal{R})\;&=\; \frac{2}{\pi}\cdot\int_{0}^{\frac{\pi}{2}}\textrm{d}\theta\int_{r_1(\theta)}^{r_2(\theta)}r\cdot e^{-r^2}\cdot \textrm{d}r\;+\; \frac{2}{\pi}\cdot \int_{\frac{\pi}{2}}^{\frac{\pi}{2}+\tau(\alpha, x^-)}\textrm{d}\theta\int_{r_1(\theta)}^\infty r\cdot e^{-r^2}\cdot\textrm{d}r.\nonumber \\
&=\; \underbrace{\frac{1}{\pi}\cdot\int_{0}^{\frac{\pi}{2}+\tau(\alpha, x^-)} e^{-r_1(\theta)^2}\cdot \textrm{d}\theta}_{=\;\nu_{1}\left(x^-, y^-, \alpha\right)}\;-\; \underbrace{\frac{1}{\pi}\cdot \int_{0}^{\frac{\pi}{2}}e^{-r_2(\theta)^2}\cdot\textrm{d}\theta}_{=\;\nu_{2}\left(x^-, y^-, \alpha\right)}.\label{decomvu121}
\end{align} 
From Craig's Formula~\eqref{craig}, 
\begin{equation}\label{idvu11}
\nu_{2}\left(x^-, y^-, \alpha\right)\;=\; \frac{1}{2}\cdot \erfc\left(\sqrt{K\left(y^-, x^-\right)}\right).
\end{equation}
Also, the polar equation of the line  $\left(\mathcal{L}_1\right)$ yields 
\begin{align*}
\nu_{1}\left(x^-, y^-, \alpha\right)\;&=\; \frac{1}{\pi}\cdot\int_{0}^{\frac{\pi}{2}+\tau(\alpha, x^-)} \exp\left(- K\left(y^-, x^+\right)\cdot \frac{1}{\cos^2\left(\theta-\tau(\alpha, x^-)\right)}\right)\cdot \textrm{d}\theta\\
&=\; \frac{1}{\pi}\cdot\int_{-\tau(\alpha, x^-)}^{\frac{\pi}{2}} \exp\left(- K\left(y^-, x^+\right)\cdot \frac{1}{\cos^2\varphi}\right)\cdot \textrm{d}\varphi.
\end{align*}
From the parity of the integrand,  one obtains
\begin{align*}
\nu_{1}\left(x^-, y^-, \alpha\right)\;&=\;  \frac{1}{\pi}\cdot\int_{0}^{\tau(\alpha, x^-)} \exp\left(- K\left(y^-, x^+\right)\cdot \frac{1}{\cos^2\varphi}\right)\cdot \textrm{d}\varphi \\
&\qquad + \frac{1}{\pi}\cdot\int_{0}^{\frac{\pi}{2}} \exp\left(- K\left(y^-, x^+\right)\cdot \frac{1}{\cos^2\varphi}\right)\cdot \textrm{d}\varphi .
\end{align*}
The change of variables $t=\tan \varphi$ in the first integral and the use of Craig's Formula~\eqref{craig}  in the second one then imply that
\begin{align}
\nu_{1}\left(x^-, y^-, \alpha\right)\;&=\; \frac{1}{\pi}\cdot \exp\left(-K\left(y^-, x^+\right)\right)\cdot  \bigintsss_{0}^{\sqrt{\frac{\alpha}{x^-}}}\frac{\exp\left(- K\left(y^-, x^+\right)\cdot t^2\right)}{1+t^2}\cdot\textrm{d}t\nonumber\\
&\qquad + \frac{1}{2}\cdot \erfc\left( \sqrt{K\left(y^-, x^+\right)}\right)\cdotp\label{idvu21}
\end{align}
The statement is then a consequence of the decomposition~\eqref{decomvu121} and the identities~\eqref{idvu11} and~\eqref{idvu21}.
\end{proof}

\begin{lem}[Decomposition of the Probability $\lambda(\mathcal{R})$   when $x^->0$ and $y^->0$]\label{declaR1}
Assume that $x^->0$ and $y^->0$. Then, the probability  $\lambda(\mathcal{R})$  defined in Lemma~\ref{decompoprob1} can be expressed as 
\begin{align*}
\lambda(\mathcal{R})\;=\; \frac{1}{\pi}\cdot\exp\left(-K\left(y^+, x^+\right)\right)\cdot T\left(A\left(\alpha, x^-\right), K\left(y^+, x^+\right)\right)\\
 -\frac{1}{2}\cdot\left(\erfc\left(\sqrt{K\left(y^+, x^+\right)}\right)-\erfc\left(\sqrt{K\left(y^+, x^-\right)}\right)\right).
\end{align*}
\end{lem}

\begin{proof}
The proof mimics that of the preceding statement and is thus only sketched. Specifically, from the Polar--type Representation Lemma~\ref{lemnulambda11}, the quantity $\lambda(\mathcal{R})$ is given by an integral  over the domain $\mathfrak{D}_r\!\left(x^-, y^-, \alpha\right)$ defined in~\eqref{dedomd1}. This domain is bounded by the two lines $$\left(\mathcal{L}_3\right)~: x\;=\; \sqrt{K\left(y^+, x^-\right)} \qquad \textrm{and}\qquad \left(\mathcal{L}_4\right)~: y\;=\; \frac{x\cdot\sqrt{2x^-}-y^--\beta}{\sqrt{2\alpha}}$$ (to make the link with Figure~\ref{figure_of_D1}, recall here that $y^+=y^-+\beta$ from the definition of the parameter $\beta$ in~\eqref{defalphabeta}). In polar coordinates, the equations of these lines become $$\left(\mathcal{L}_3\right): r=\widehat{r}_1\left(\theta\right)= \frac{\sqrt{K\left(y^+, x^-\right)}}{ \cos\theta}\quad \textrm{and}\quad \left(\mathcal{L}_4\right): r=\widehat{r}_2\left(\theta\right)\;=\; \frac{\sqrt{K\left(y^+, x^+\right)}}{\cos\left(\theta+\tau(\alpha, x^-)\right)},$$ where, as in the previous proof, one sets $\tau(\alpha, x^-)= \arctan \; A\!\left(\alpha, x^-\right)$. The integral expressing $\lambda(\mathcal{R})$ in  the Polar--type Representation Lemma~\ref{lemnulambda11} then yields that 
\begin{align*}
\lambda(\mathcal{R})\;&=\; \frac{2}{\pi}\cdot\int_{0}^{\frac{\pi}{2}-\tau(\alpha, x^-)}\textrm{d}\theta\int_{\widehat{r}_1(\theta)}^{\widehat{r}_2(\theta)}r\cdot e^{-r^2}\cdot \textrm{d}r\;+\; \frac{2}{\pi}\cdot \int_{\frac{\pi}{2}-\tau(\alpha, x^-)}^{\frac{\pi}{2}}\textrm{d}\theta\int_{\widehat{r}_1(\theta)}^\infty r\cdot e^{-r^2}\cdot\textrm{d}r.\nonumber \\
&=\; \frac{1}{\pi}\cdot\int_{0}^{\frac{\pi}{2}} e^{-\widehat{r}_1(\theta)^2}\cdot \textrm{d}\theta \;-\;  \frac{1}{\pi}\cdot \int_{0}^{\frac{\pi}{2}-\tau(\alpha, x^-)}e^{-\widehat{r}_2(\theta)^2}\cdot\textrm{d}\theta.
\end{align*}
The calculations are completed in the same way as in the proof of the  Lemma~\ref{decnuR1}.
\end{proof}

\subsection{An Auxiliary Estimate}

\noindent The proof of Proposition~\ref{proppasssqbis} requires the determination of sharp estimates on Owen's $T$--function defined in~\eqref{owen} in all ranges of its variables. This plays a crucial role in the exploitation of the identities established in  Lemmata~\ref{decnuR1} and~\ref{declaR1}. 

\begin{lem}[Bounds on Owen's $T$--function]\label{boundsowen} Let $a,k>0$ be reals. Then, 
\begin{equation}\label{genboundowen}
T(a,k)\;\le\; \min\left\{\arctan(a),\; \frac{\sqrt{\pi}}{2}\cdot\frac{1}{\sqrt{k}}\right\}\;\le\;  \min\left\{\frac{\pi}{2}, \; a,\; \frac{\sqrt{\pi}}{2}\cdot\frac{1}{\sqrt{k}}\right\}.
\end{equation}
Furthermore,   
\begin{equation}\label{regimeka2le1}
0\;\le\; \arctan(a)-T(a,k)\;\le\; k\cdot\left(a-\arctan(a)\right)\qquad \textrm{when}\qquad ka^2\le 1
\end{equation}
and 
\begin{equation}\label{regimeka2le1ale1}
0\;\le\; a-T(a,k)\;\le\; (k+1)\cdot \frac{a^3}{3} \qquad \textrm{when}\qquad ka^2\le 1\qquad \textrm{and}\qquad  a<1.
\end{equation}
\end{lem}

\begin{proof}
It is clear from the definition~\eqref{owen}  of Owen's T--function that $T(a,k)\le \int_0^1 \textrm{d}t/(1+t^2)=\arctan (a)$. What is more, from an elementary change of variables, 
\begin{equation*}
T(a,k)\;\le\;  \int_0^a e^{-kt^2}\cdot\textrm{d}t\;=\; \frac{1}{\sqrt{k}}\cdot \int_0^{a\sqrt{k}}e^{-v^2}\cdot \textrm{d}v\;\le\; \frac{\sqrt{\pi}}{2}\cdot \frac{1}{\sqrt{k}}\cdotp
\end{equation*}
This yields the bound~\eqref{genboundowen}.\\

\noindent In the regime $ka^2\le 1$, it follows from the convexity inequality $1-u\le e^{-u}\le 1$ valid for all $u\ge 0$ that 
\begin{align*}
0\;\le\; a\cdot \int_0^1\frac{1}{1+a^2u^2}\cdot\textrm{d}u-T(a,k)\;\le\; ka^3\cdot\int_0^1\frac{u^2}{1+a^2u^2}\cdot \textrm{d}u .
\end{align*}
An explicit calculation of the integrals gives the inequality~\eqref{regimeka2le1}. When $a<1$, the alternating power series expansion of the $\arctan$ function shows that $a-a^3/3\le \arctan(a)\le a$, whence the inequality~\eqref{regimeka2le1ale1}.  
\end{proof}

\subsection{Derivation of Sharp Bounds on the Subprobabilities}

\noindent The sum $\nu(\mathcal{R})+\lambda(\mathcal{R})$ is first estimated with the help of Lemmata~\ref{decnuR1} and~\ref{declaR1}.

\begin{coro}[Bounds on the sum  $\nu(\mathcal{R})+\lambda(\mathcal{R})$ when $x^->0$ and $y^->0$]\label{asymptdecomposumnulambda1} Assume that $x^->0$ and $y^->0$. Then, with the notation introduced just before the statement of Proposition~\ref{proppasssqbis},
\begin{align*}
\nu(\mathcal{R})+\lambda(\mathcal{R})\;\le\; &\frac{E\left(x^+, y^-\right)}{\pi}\cdot\left(2\cdot F\left(\alpha, x^-, x^+, y^-\right)+ \frac{\pi}{2}\cdot G\left(\beta, x^+, y^-\right)\right)\\
&\qquad  + \frac{E\left(x^-, y^-\right)}{2}\cdot G\left(\beta, x^-, y^-\right).
\end{align*}
Furthermore, this estimate can be refined in the regimes 1.(a) and 1.(b) defined in Proposition~\ref{proppasssqbis}. Indeed, the sum $\nu(\mathcal{R})+\lambda(\mathcal{R})$ then equals the right--hand side of the equation~\eqref{ppi1a} with the same specification of the values of the function involved in it but with  only one change~: one requires here that $\lambda=\tau\in [0,1]$

\end{coro}

\noindent The proof of Corollary~\ref{asymptdecomposumnulambda1} relies on the well--known inequality
\begin{equation}\label{ineqerrfct1}
0\;\le\; \erfc(x)\;\le\; \frac{2}{\sqrt{\pi}}\cdot \frac{e^{-x^2}}{1+x}
\end{equation}
valid for all $x\ge 0$ --- see, e.g., \cite[p.17]{IMcK} for a proof. 

\begin{proof}[Proof of Corollary~\ref{asymptdecomposumnulambda1}] Putting together the conclusions of the above two Lemmata~\ref{decnuR1} and~\ref{declaR1}, one has that
\begin{align}
\nu(\mathcal{R})+\lambda(\mathcal{R})\;=\; &\frac{1}{\pi}\cdot\exp\left(-K\left(y^-, x^+\right)\right)\cdot T\left(A\left(\alpha, x^-\right), K\left(y^-, x^+\right)\right)\nonumber\\
&+   \frac{1}{\pi}\cdot\exp\left(-K\left(y^+, x^+\right)\right)\cdot T\left(A\left(\alpha, x^-\right), K\left(y^+, x^+\right)\right) \nonumber +\frac{1}{2}\cdot \mathcal{E}\left(x^-, y^-, \alpha, \beta\right), \label{endal-1}
\end{align}
where, given  $a,b,\varepsilon, \eta>0$ one has set  
\begin{align*}
&\mathcal{E}\left(a,b,\varepsilon, \eta\right)\;=\; \\
 &\qquad \left(\erfc\left(\frac{b}{\sqrt{2\left(a+\varepsilon\right)}}\right)-\erfc\left(\frac{b+\eta}{\sqrt{2\left(a+\varepsilon\right)}}\right)\right) +\left(\erfc\left(\frac{b}{\sqrt{2a}}\right)-\erfc\left(\frac{b+\eta}{\sqrt{2a}}\right)\right).
 \end{align*}
It follows from the Mean Value Inequality and from the above relation~\eqref{ineqerrfct1} that
\begin{align*}
\erfc\left(\frac{b}{\sqrt{2a}}\right)-\erfc\left(\frac{b+\eta}{\sqrt{2a}}\right)\;\le\; \frac{\exp\left(-\frac{b^2}{2a}\right)}{2}\cdot \min\left\{\frac{\eta}{2a}, \frac{2}{\sqrt{\pi}}\cdot\frac{1}{1+\frac{b}{\sqrt{2a}}}\right\},
\end{align*}
and similarly when $a$ is replaced with $a+\varepsilon$. 
Thus,  
\begin{align*}
0\;\le\;\mathcal{E}\left(a,b,\varepsilon,\; \eta\right)\;\le\;  &  \frac{\exp\left(-b^2/(2a)\right)}{2}\cdot \min\left\{\frac{\eta}{\sqrt{2a}}; \, \frac{2}{\sqrt{\pi}}\cdot \frac{1}{1+\frac{b}{\sqrt{2a}}}\right\} \\
&+ \frac{\exp\left(-\frac{b^2}{2(a+\varepsilon)}\right)}{2} \cdot \min\left\{\frac{\eta}{\sqrt{2(a+\varepsilon)}}; \, \frac{2}{\sqrt{\pi}}\cdot \frac{1}{1+\frac{b}{\sqrt{2(a+\varepsilon)}}}\right\}.
\end{align*}
As a consequence, $$ \mathcal{E}\left(x^-,y^-,\alpha,\beta\right)\;\le\;  E\left(x^-, y^-\right)\cdot G\left(\beta, x^-, y^-\right)+E\left(x^+, y^-\right)\cdot G\left(\beta, x^+, y^-\right).$$
This leads one to the statement upon putting these estimates together with the bounds for Owen's T--function obtained in Lemma~\ref{boundsowen}.  
\end{proof}

\noindent A final lemma in this section provides an elementary bound for the probability $\mu\left(\mathcal{R}\right)$ introduced as part of the Decomposition Lemma~\ref{decompoprob1}.

\begin{lem}[Elementary upper bound for  the probability $\mu\left(\mathcal{R}\right)$ when $x^->0$]\label{probmuR1}  Assume that  $x^->0$. Then, the  probabi\-lity $\mu(\mathcal{R})$  meets the relation $$\mu\left(\mathcal{R}\right)\;\le\; \frac{E\left(x^-, y^-\right)}{2}\cdot G\left(\beta, x^-, y^-\right).$$
\end{lem}

\begin{proof} Since by definition,  $$\mu(\mathcal{R})\;=\; \frac{1}{2}\cdot\left(\erfc\left(\frac{y^-}{\sqrt{2x^-}}\right)-\erfc\left(\frac{y^+}{\sqrt{2x^-}}\right)\right),$$ the statement is, as in the previous proof,  an immediate consequence of the Mean Value Inequality and of  the  relation~\eqref{ineqerrfct1}. 
\end{proof}

%
%
%

\subsection{Completion of the Proof {of} the Main Estimates}

\begin{proof}[Proof of Proposition~\ref{proppasssq}]
From the symmetry property (P1), it is enough to prove the statement when assuming that $y^+\ge 0$. Fix the coordinates $y^\pm$ and also the value of the parameter $\alpha>0$ defined in~\eqref{defalphabeta}. Consider the family of horizontally translated rectangles $$\mathfrak{R}\left[\alpha, y^\pm\right]\;=\;\left\{\mathcal{R}\subset \R_{\ge 0}\times\R\;:\; \left(x^-_{\mathcal{R}},\; x^+_{\mathcal{R}},\; y^-_{\mathcal{R}},\; y^+_{\mathcal{R}}\right) \;=\;  \left(x^-, \; x^-+\alpha, \; y^-,\;  y^+\right) \right\}_{x^-\ge 0}$$ parametrised by the value of $x^-\ge 0$.  An element in the family $\mathfrak{R}\left[\alpha, y^\pm\right]$ is thus denoted by $\mathcal{R}\left(x^-\right)$ and the corresponding hitting probability by $$\ppi\left(x^-\right)\; :=\; \ppi\left(\mathcal{R}\left(x^-\right)\right).$$ The key idea in the proof is to combine the Probability Decomposition Lemma~\ref{decompoprob1} with the \emph{heat equation} satisfied by the Gaussian density function $p_t$ when $t>0$, namely that 
\begin{equation*}
\frac{\partial p_t}{\partial t}(u)\;=\; \frac{1}{2}\cdot\frac{\partial^2 p_t}{\partial u^2}(u).
\end{equation*}
When $x^->0$, one thus obtains that 
\begin{align*}
\frac{\partial \ppi\left(\mathcal{R}\left(x^-\right)\right)}{\partial x^-}\;&=\; \frac{1}{2}\cdot \int_{y^-}^{y^+}\frac{\partial^2 p_{x^-}}{\partial u^2}(u)\cdot \textrm{d}u \;+\; \int_{-\infty}^{y^-} \frac{\partial^2 p_{x^-}}{\partial u^2}(u)\cdot\left(\int_{y^--u}^{\infty}p_\alpha(v)\cdot\textrm{d}v\right)\\
&\qquad \qquad \qquad \qquad  \qquad \qquad +\int_{y^+}^{\infty} \frac{\partial^2 p_{x^-}}{\partial u^2}(u)\cdot\left(\int_{u-y^+}^{\infty}p_\alpha(v)\cdot\textrm{d}v\right).
\end{align*}
Since $$\frac{\partial p_t}{\partial u}(u)\;=\; \left(-\frac{u}{t}\right)\cdot p_t(u),$$ it follows from an integration by parts that 
\begin{align*}
\frac{\partial \ppi\left(\mathcal{R}\left(x^-\right)\right)}{\partial x^-}\;&=\; \int_{-\infty}^{y^-}\frac{u}{x^-}\cdot p_{x^-}(u)\cdot p_{\alpha}\left(y^--u\right)\cdot\textrm{d}u\;-\; \int_{y^+}^{\infty}\frac{u}{x^-}\cdot p_{x^-}(u)\cdot p_{\alpha}\left(u-y^+\right)\cdot\textrm{d}u,
\end{align*}
and thus that 
\begin{align}
 \ppi\left(\mathcal{R}\left(x^-\right)\right)- \ppi\left(\mathcal{R}(0)\right)\;=\; &\int_{-\infty}^{y^-}u\cdot \left(\int_0^{x^-}\frac{p_t(u)}{t}\cdot\textrm{d}t\right)\cdot p_{\alpha}\left(y^--u\right)\cdot\textrm{d}u \nonumber\\
& - \int_{y^+}^{\infty}u\cdot \left(\int_0^{x^-}\frac{p_t(u)}{t}\cdot\textrm{d}t\right)\cdot p_{\alpha}\left(u-y^+\right)\cdot\textrm{d}u .\label{eqcomppro}
 \end{align}
 In this relation, by an elementary change of variables when $u\neq 0$,
\begin{align}
\int_0^{x^-}\frac{p_t(u)}{t}\cdot\textrm{d}t\;&=\; \frac{1}{\sqrt{2\pi}}\cdot \int_0^{x^-}\frac{e^{-u^2/(2t)}}{t}\cdot\textrm{d}t\nonumber \;=\; \frac{2}{|u|\cdot \sqrt{\pi}}\int_{\frac{|u|}{\sqrt{2x^-}}}^{\infty}e^{-v^2}\cdot\textrm{d}v\\
&=\; \frac{1}{|u|}\cdot \erfc\left(\frac{|u|}{\sqrt{2x^-}}\right).\label{intgaust}
\end{align} 
Also, the probability that  Brownian motion should hit the rectangle $\mathcal{R}(0)$ equals
$$ \ppi\left(\mathcal{R}(0)\right)\;=\; \bP\left(M\left(x^+\right)\ge \max\left\{0, y^-\right\}\right).$$ From the Law of Maximum~\eqref{Mt}, this is saying  that
\begin{equation}
\label{lawmaxprobcom}\ppi\left(\mathcal{R}(0)\right)\;=\; 2\cdot \int_{ \max\left\{0, y^-\right\}}^{\infty} p_{x^+}(u)\cdot \textrm{d}u\;\underset{\eqref{def_p}}{=}\; \erfc\left( \frac{\max\left\{0, y^-\right\}}{\sqrt{2x^+}}\right).
\end{equation}
The conclusion then follows upon putting together the equations~\eqref{eqcomppro}, \eqref{intgaust},  and~\eqref{lawmaxprobcom}. 
\end{proof}

\begin{proof}[Proof of Proposition~\ref{proppasssqbis}] In view of the  Decomposition Lemma~\ref{decompoprob1}, \textbf{Case~1} of the statement is obtained by adding up the estimates established in Corollary~\ref{asymptdecomposumnulambda1} and Lemma~\ref{probmuR1}.\\

\noindent \textbf{Cases~2.(c)} and  \textbf{4.(b)} are immediate consequences of the fact that  Brownian motion starts almost surely at the origin.\\

\noindent The reduction stated in \textbf{Case~3} is justified by the property of symmetry (P1). \\

\noindent \textbf{Cases~2.(a, b)} are obtained from an application of the Dominated Convergence Theorem.\\

\noindent Finally, to deal with \textbf{Case~4.(a)} , it is enough to notice that the probability $ \ppi\left(\mathcal{R}\right)$ is the sum of the probabilities that  Brownian motion should cross the subrectangles $\mathcal{R}^+$ and $\mathcal{R}^-$ minus the probability of intersecting both. From~\cite[Theorem~3.26, p.77]{K}, the latter probability equals $$1-\frac{2}{\pi}\arcsin\sqrt{\frac{x^-}{x^+}},$$ whence the statement in this case.
\end{proof}

\section{Asymptotics of the Diagonal Probability Term}\label{asymdiag} 

Propositions~\ref{diagestimterm} and~\ref{diagestimtermbis} are established in this section. As a matter of fact, a common  inhomogeneous generalisation of them is proved. It corresponds to the case where the rectangles under consideration are not centered at rationals but at rational points shifted by some inhomogeneous quantity. This generality is needed in the proof of the estimation of the off--diagonal probability term (i.e.~of Proposition~\ref{offdiagestimterm}) in Section~\ref{asymoffdiag}. It requires some additional notation. 

\paragraph{Notation for the inhomogeneous setting.}  Throughout this section, the generic notation $\bm{s}/q=\left(s_1/q, s_2/q\right)$ always stands for a fraction with denominator  $q\ge 1$ and numerators $s_1\ge 0$ and $s_2\in\Z$. Thus,   
\begin{equation}\label{convention}
\frac{\bm{s}}{q}=\left(\frac{s_1}{q}, \frac{s_2}{q}\right)\in \Q_{\ge 0}\times \Q,
\end{equation}
and this is assumed implicitly but not be specified anymore to alleviate the notation.  Given a real parameter $\gamma_1$, define the inhomogeneous version of the strip $\Q_{\xi}( \bm{\mathcal{T}})$ introduced in~\eqref{defstrip} as the set 
\begin{align}\label{deinhomogstrip}
 \Q_{\xi}^ {(\gamma_1)} ( \bm{\mathcal{T}})\;=\;  \left\{\frac{\bm{s}}{q} \;:\; \left[\frac{s_1}{q}+\gamma_1-\frac{\xi(q)}{2q}, \frac{s_1}{q}+\gamma_1+\frac{\xi(q)}{2q}\right]\cap\left[T_1, T_2\right]\;\neq\; \emptyset\right\}.
\end{align}
Introducing an additional real parameter $\gamma_2$ and setting $\bm{\gamma}=\left(\gamma_1, \gamma_2\right)$,  decompose this set into two subsets 
\begin{align}\label{decomCQT}
\Q_{\xi}^ {(\gamma_1)} (\bm{\mathcal{T}})\;=\; \widehat{\Q}_{\xi}^ {(\bm{\gamma})}( \bm{\mathcal{T}}) \; \cup\;  \partial \Q_{\xi}^ {(\bm{\gamma})} (\bm{\mathcal{T}}).  
\end{align}
Here, the so--called \emph{main set of rationals} $\widehat{\Q}_{\xi}^ {(\gamma)}( \bm{\mathcal{T}})$ is 
 \begin{align}\label{decomCQT1}
\widehat{\Q}_{\xi}^ {(\gamma)}( \bm{\mathcal{T}}) \;=\; &\left\{\frac{\bm{s}}{q}  \in   \Q_{\xi}^ {(\gamma_1)} ( \bm{\mathcal{T}}) \;:\;  \left|\frac{s_2}{q}+\gamma_2\right|\;\ge\;  \frac{\xi(q)}{2q}, \right. \nonumber\\
 &\qquad \quad   \left. \max\left\{\frac{\xi(q)}{q}, \;T_1\right\}\le \frac{s_1}{q}+\gamma_1-\frac{\xi(q)}{2q}<\frac{s_1}{q}+\gamma_1+\frac{\xi(q)}{2q}\le T_2\right\}, 
 \end{align}
and the so--called \emph{set of boundary rationals} $\partial \Q_{\xi}^ {(\bm{\gamma})} (\bm{\mathcal{T}})$ is the complementary set 
\begin{align}\label{decomCQT0}
 \partial \Q_{\xi}^ {(\bm{\gamma})} (\bm{\mathcal{T}}) \;=\; \Q_{\xi}^ {(\gamma_1)} (\bm{\mathcal{T}})\;\backslash\;  \widehat{\Q}_{\xi}^ {(\gamma)}( \bm{\mathcal{T}}) . 
 \end{align} 

\noindent Also, given  a rational point $\bm{s}/q=\left(s_1/q, s_2/q\right)$, denote by $\mathcal{R}_\xi^{(\bm{\gamma})}\!\left(\bm{\mathcal{T}}; \bm{s}/q\right)$ the rectangle obtained by shifting the rectangle $\mathcal{R}_\xi\!\left(\bm{\mathcal{T}}; \bm{s}/q\right)$ defined in~\eqref{defract}  by $\bm{\gamma}$ in the following sense~:  
\begin{align}\label{inhomogrect}
\mathcal{R}_\xi^{(\bm{\gamma})}\!\left(\bm{\mathcal{T}}; \frac{\bm{s}}{q}\right) \;= \; &\left(\left[\frac{s_1}{q}+\gamma_1-\frac{\xi(q)}{2q}, \; \frac{s_1}{q}+\gamma_1+\frac{\xi(q)}{2q}\right]\cap \left[T_1,T_2\right]\right)\nonumber\\
&\qquad\qquad \times\,\left[\frac{s_2}{q}+\gamma_2-\frac{\xi(q)}{2q}, \; \frac{s_2}{q}+\gamma_2+\frac{\xi(q)}{2q}\right].
\end{align}
By analogy with the quantity $\ppi_\xi\left(\bm{\mathcal{T}}; \bm{s}/q\right)$ introduced in~\eqref{defpiprob}, the probability that a trajectory of  Brownian motion should hit the inhomogeneous rectangle $\mathcal{R}_\xi^{(\bm{\gamma})}\!\left(\bm{\mathcal{T}} ; \bm{s}/q\right)$ is  defined as
\begin{equation*}\label{inhomogprob}
\ppi_\xi^{(\bm{\gamma})}\!\left(\bm{\mathcal{T}}; \frac{\bm{s}}{q}\right)\;=\; \ppi\left(\mathcal{R}_\xi^{(\bm{\gamma})}\!\left(\bm{\mathcal{T}}; \frac{\bm{s}}{q}\right)\right).
\end{equation*} 
\hfill \qedsymbol\\

\noindent With this set of notations, given an integer $q\ge 1$, the main quantities of interest are the \emph{global inhomogeneous sum}
\begin{align}\label{defaxilfixq} 
\mathcal{S}_{\xi}^{(\bm{\gamma})}\left(\bm{\mathcal{T}}; q\right)\;=\; \sum_{\frac{\bm{s}}{q}\in   \Q_{\xi}^{(\gamma_1)} ( \bm{\mathcal{T}})}\ppi_\xi^{(\bm{\gamma})}\!\left(\bm{\mathcal{T}}; \frac{\bm{s}}{q}\right)
\end{align}
and the  \emph{auxiliary homogeneous one}
\begin{align}\label{defaxilfixqbis} 
\mathcal{O}_{ \xi}\left(\bm{\mathcal{T}}; q\right)\;=\; \sum_{\frac{\bm{s}}{q}\in  \Q_\xi(\bm{\mathcal{T}})}\ppi_\xi\!\left(\frac{\bm{s}}{q}\right)\cdot \frac{s_1}{q}, 
\end{align}
where the summation is over the set of integers $\bm{s}=\left(s_1, s_2\right)\in\N_{\ge 0}\times\Z$ meeting the stated conditions. Most of the work consists in establishing a precise enough asymptotic expansion for the global inhomogeneous sum. The one for the auxiliary sum  follows from it. \\

\noindent Focussing consequently first on the quantity $\mathcal{S}_{\xi}^{(\bm{\gamma})}\left(\bm{\mathcal{T}}; q\right)$, in accordance with the decomposition~\eqref{decomCQT}, the sum is split (with obvious notation)   into two subsums~: 
\begin{align}\label{decomSIB}
\mathcal{S}_{\xi}^{(\bm{\gamma})}\left(\bm{\mathcal{T}}; q\right)\;=\; \widehat{\mathcal{S}}_{\xi}^{(\bm{\gamma})}\left(\bm{\mathcal{T}}; q\right)\;+\;  \partial \mathcal{S}_{\xi}^{(\bm{\gamma})}\left(\bm{\mathcal{T}}; q\right).
\end{align}
The \emph{main subsum} $ \widehat{\mathcal{S}}_{\xi}^{(\bm{\gamma})}\left(\bm{\mathcal{T}}; q\right)$ corresponds to a regime where the sharp estimate~\eqref{refestimcoro} established in Corollary~\ref{coroutile} applies whereas the cruder upper bound~\eqref{mainupbprob}  in the same statement suffices to estimate the  \emph{complementary subsum} $\partial \mathcal{S}_{\xi}^{(\bm{\gamma})}\left(\bm{\mathcal{T}}; q\right)$.  The main result of this section determines the asymptotic behavior of the global inhomogeneous sum $\mathcal{S}_{\xi}^{(\bm{\gamma})}\left(\bm{\mathcal{T}}; q\right)$ regardless of the relative size of the parameters $T_1$ and $T_2$ on the one hand and of the integer $q$ on the other. This is crucial to estimate the off--diagonal probability term in the following section as the parameter $\bm{\mathcal{T}}=\left(T_1, T_2\right)$ is then  taken as a suitable function of $q$.

\begin{prop}[Asymptotic behavior of the global inhomogeneous sum $\mathcal{S}_{\xi}^{(\bm{\gamma})}\left(\bm{\mathcal{T}}; q\right)$]\label{asymptinhomg}
Recall that $\xi~:~\left[1, \infty\right)\rightarrow (0, 1]$ is an approximation function. Fix the limit times  $\bm{\mathcal{T}}=(T_1, T_2)$, an integer $q\ge 1$ and   the shift vector 
\begin{equation}\label{eqshift}
\bm{\gamma}=\left(\gamma_1, \gamma_2\right) \in \left[0, \;\frac{1}{q}\right)\times \left[-\frac{1}{2q}, \;\frac{1}{2q}\right).
\end{equation} 
Then,  
the global inhomogeneous sum $\mathcal{S}_{\xi}^{(\bm{\gamma})}\left(\bm{\mathcal{T}}; q\right)$ defined in~\eqref{defaxilfixq} satisfies the asymptotic expansion
\begin{align*}
&\mathcal{S}_{\xi}^{(\bm{\gamma})}\left(\bm{\mathcal{T}}; q\right)\;=\; \frac{2\sqrt{2}}{\sqrt{\pi}}\cdot\left(T_2-T_1\right)\cdot\sqrt{q^3\cdot\xi(q)}\;+\; O\left(\mathcal{E}_{\xi}^{\left(\bm{\gamma}\right)}\left(\bm{\mathcal{T}}; q\right)\right),
\end{align*}
where the error term $\mathcal{E}_{\xi}^{\left(\bm{\gamma}\right)}\left(\bm{\mathcal{T}}, q\right)$ is defined as 
\begin{align*}
\mathcal{E}_{\xi}^{\left(\bm{\gamma}\right)}\left(\bm{\mathcal{T}}, q\right)\;=\;\sqrt{q \xi(q)}\cdot&\left(\left(1+\sqrt{q \xi(q)}\right)\cdot\left(1+T_2-T_1+\sqrt{T_2}-\sqrt{T_1}\right)\right.\\
&\quad  \left.+\left|\log\left(\frac{T_2}{\max\left\{T_1, 1/q\right\}}\right)\right|\right)
+\Lambda_{\xi}^{\left(\bm{\gamma}\right)}\left(\bm{\mathcal{T}}; q\right).
\end{align*} 
Here, $\Lambda_{\xi}^{\left(\bm{\gamma}\right)}\left(\bm{\mathcal{T}}; q\right)$ is the sum of four terms which are nonzero only conditionally. Specifically, $$\Lambda_{\xi}^{\left(\bm{\gamma}\right)}\left(\bm{\mathcal{T}}, q\right)\;=\; \sum_{i=1}^4 \Lambda_{\xi}^{\left(\bm{\gamma}, i\right)}\left(\bm{\mathcal{T}}; q\right),$$
where
\begin{itemize}
\item $\Lambda_{\xi}^{\left(\bm{\gamma}; 1\right)}\left(\bm{\mathcal{T}}, q\right)$ vanishes unless 
\begin{align*}
T_1<\frac{\xi(q)}{q} \qquad \textrm{and}\qquad \frac{s_1}{q}+\gamma_1<\frac{3\xi(q)}{2q}\quad  \textrm{ for some } s_1\in\left\{ 0,1\right\},
\end{align*} 
in which case it equals 
\begin{align*}
\Lambda_{\xi}^{\left(\bm{\gamma}, 1\right)}\left(\bm{\mathcal{T}}, q\right)\;=\; 1\;+\;q\sqrt{T_1} \;+\;
\begin{cases}
\xi(q)/\left(q\sqrt{T_1}\right) & \textrm{if} \quad\left|\gamma_2\right|\le\xi(q)/\left(2q\right);\\
0&\textrm{otherwise};
\end{cases}
\end{align*}

\item $\Lambda_{\xi}^{\left(\bm{\gamma}, 2\right)}\left(\bm{\mathcal{T}}; q\right)$ vanishes unless 
\begin{align*}
T_2-T_1\;\ge\; \frac{\xi(q)}{q},
\end{align*} 
in which case it equals 
\begin{align*}
&\Lambda_{\xi}^{\left(\bm{\gamma}, 2\right)}\left(\bm{\mathcal{T}}, q\right)\;=\;  \\
&\min\!\left\{1, \sqrt{\frac{\xi(q)}{qT_2-\xi(q)}}\right\}\cdot \left(\sqrt{q\left(qT_2-\xi(q)\right)}+\max\!\left\{\exp\!\left(-\frac{\left\{\xi(q)/2\pm q\gamma_2\right\}^2}{2q\left(qT_2-\xi(q)\right)}\right)\right\}\right); 
\end{align*}

\item $\Lambda_{\xi}^{\left(\bm{\gamma}, 3\right)}\left(\bm{\mathcal{T}}; q\right)$ vanishes unless 
\begin{align*}
\max\left\{q\left(T_1-\gamma_1\right)-\xi(q), 0\right\}\;\le\; \left\lfloor \frac{3}{2}\cdot \xi(q)-q\gamma_1\right\rfloor,
\end{align*} 
in which case it equals 
\begin{align*}
&\Lambda_{\xi}^{\left(\bm{\gamma}, 3\right)}\left(\bm{\mathcal{T}}, q\right)\;=\;  1;
\end{align*}

\item $\Lambda_{\xi}^{\left(\bm{\gamma}, 4\right)}\left(\bm{\mathcal{T}}; q\right)$ vanishes unless 
\begin{align*}
\max\left\{\frac{\xi(q)}{q}, T_1\right\}\;\le\; \gamma_1-\frac{\xi(q)}{2q}\;<\; \gamma_1+\frac{\xi(q)}{2q}\;\le\; T_2,
\end{align*} 
in which case it equals 
\begin{align*}
&\Lambda_{\xi}^{\left(\bm{\gamma}, 4\right)}\left(\bm{\mathcal{T}}, q\right))\;=\; \sqrt{\frac{\xi(q)}{q\gamma_1-\xi(q)/2}}\cdot\left(1+\sqrt{q\cdot\left(q\gamma_1+\frac{\xi(q)}{2}\right)}\right)\nonumber\\
&\qquad +\left(\frac{\xi(q)}{q\gamma_1-\xi(q)/2}\right)^{3/2}\cdot\left(1+\frac{\left(q\gamma_1+\xi(q)/2\right)^{3/2}}{\sqrt{q}\cdot\left(q\gamma_1-\xi(q)/2\right)}+\sqrt{q\cdot\left(q\gamma_1+\frac{\xi(q)}{2}\right)}\right).
\end{align*}
\end{itemize}
In particular, $\Lambda_{\xi}^{\left(\bm{\gamma}, 1\right)}\left(\bm{\mathcal{T}}; q\right)=\Lambda_{\xi}^{\left(\bm{\gamma}, 3\right)}\left(\bm{\mathcal{T}}; q\right)=\Lambda_{\xi}^{\left(\bm{\gamma}, 4\right)}\left(\bm{\mathcal{T}}; q\right)=0$ and 
\begin{align}\label{redasympt}
\Lambda_{\xi}^{\left(\bm{\gamma}\right)}\left(\bm{\mathcal{T}}; q\right)=\Lambda_{\xi}^{\left(\bm{\gamma}, 2\right)}\left(\bm{\mathcal{T}}; q\right)\qquad \textrm{whenever}\qquad q\ge \frac{4}{T_1}\cdotp
\end{align}
\end{prop}

\noindent It should be noted that the assumption~\eqref{eqshift} is made without loss of generality in view of the $(1/q)$--periodicity in the plane of the rationals with denominators $q$. A consequence of this statement is a similar asymptotic expansion for the auxiliary inhomogeneous sum $\mathcal{O}_{ \xi}^{(\bm{\gamma})}\left(\bm{\mathcal{T}}; q\right)$ defined in~\eqref{defaxilfixqbis}.

\begin{coro}[Asymptotic behavior of the auxiliary homogeneous  sum $\mathcal{O}_{ \xi}\!\left(\bm{\mathcal{T}}; q\right)$]\label{asymptinhomgbis} Keep the notation and assumptions of Proposition~\ref{asymptinhomg}. Then, provided that 
\begin{equation*}
q\;\ge\; \frac{1}{T_1},
\end{equation*}
one has that  
\begin{equation*}
\mathcal{O}_{ \xi}\!\left(\bm{\mathcal{T}}; q\right)\;=\; \frac{\sqrt{2}}{\sqrt{\pi}}\cdot\left(T_2^2-T_1^2\right)\cdot\sqrt{q^3\cdot\xi(q)}\;+\;O\left(V\!\left(\bm{\mathcal{T}}\right)\cdot\sqrt{q\cdot\xi(q)^3}\;+\; T_2\cdot \mathcal{E}_{\xi}^{\left(\bm{0}\right)}\!\left(\bm{\mathcal{T}}, q\right)\right)
\end{equation*}
with an absolute implicit constant. Here, $V\!\left(\bm{\mathcal{T}}\right)$ is the constant defined in~\eqref{defVT} and $\mathcal{E}_{\xi}^{\left(\bm{0}\right)}\!\left(\bm{\mathcal{T}}, q\right)$ is the error term $\mathcal{E}_{\xi}^{\left(\bm{\gamma}\right)}\!\left(\bm{\mathcal{T}}, q\right)$ introduced in the statement of Proposition~\ref{asymptinhomg} specialised to the homogeneous case $\bm{\gamma}=\bm{0}$.
\end{coro}

\noindent Proposition~\ref{asymptinhomg} suffices to determine the asymptotic behavior of the diagonal probabi\-lity term~\eqref{DPT}.

\begin{proof}[Deduction of Proposition~\ref{diagestimterm} from Proposition~\ref{asymptinhomg}] It should be clear that, given integers $Q_2>Q_1\ge 1$,   the diagonal probability term $\Delta_{\xi}\left(\bm{\mathcal{Q}}; \bm{\mathcal{T}}\right)$ defined in~\eqref{DPT} is obtained upon specialising   the quantity $\mathcal{S}_{\xi}^{(\bm{\gamma})}\left(\bm{\mathcal{T}}; q\right)$ to the value  $\bm{\gamma}=\bm{0}$ in the sense that
$\Delta_{\xi}\left(\bm{\mathcal{Q}}; \bm{\mathcal{T}}\right)= \sum_{q=Q_1}^{Q_2} \mathcal{S}_{\xi}^{(\bm{0})}\left(\bm{\mathcal{T}}; q\right).$ Proposition~\ref{diagestimterm} then becomes an immediate consequence of the asymptotic expansion for the quantity $\mathcal{S}_{\xi}^{(\bm{0})}\left(\bm{\mathcal{T}}; q\right)$ stated in the above Proposition~\ref{asymptinhomg} upon noticing that under the assumption~\eqref{asymptlageT_1T_2}, the reduction~\eqref{redasympt} holds.
\end{proof}

\begin{proof}[Deduction of Proposition~\ref{diagestimtermbis} from Proposition~\ref{asymptinhomg}] From the definition of the quantity $\widehat{f}_q$ in~\eqref{dehtfq}, it is immediate that $\widehat{f}_q=\mathcal{S}_{\xi}^{(\bm{0})}\left(\bm{\mathcal{T}}; q\right)$. The conclusion of Proposition~\ref{diagestimtermbis}  is then a particular case of that of  Proposition~\ref{asymptinhomg}.
\end{proof}

\color{black}

\noindent Corollary~\ref{asymptinhomgbis} is  easily derived in Section~\ref{subsubasymexpains} below from the proof of Proposition~\ref{asymptinhomg}. This proof relies on  two auxiliary lemmata, each concerned with one of the subsums introduced in the decomposition~\eqref{decomSIB}. The notation and assumptions introduced in the statement of the proposition are maintained in   the lemmata. \\

\noindent The integral test for convergence is used repeatedly in the various proofs in the follo\-wing form~: if $\alpha\le\beta$ are elements of the extended real line $\R\cup\left\{\pm\infty\right\}$ and if $f~: \left(\alpha, \beta\right)\rightarrow \R$ is  a nonincreasing map, then (with obvious conventions when $\alpha, \beta\in\left\{\pm\infty\right\})$,  $$\int_{\lceil\alpha\rceil}^{\left[\beta\right]}f\;\le\; \sum_{\alpha< k< \beta}f(k)\;\le\; f\left(\lceil\alpha\rceil\right)\;+\;\int_{ \alpha}^{ \beta}f.$$ Here and throughout, $\lfloor\;\cdot\;\rfloor$ and $\lceil\;\cdot\;\rceil$ denote the floor and ceiling functions, respectively, and  $\left\{\;\cdot\;\right\}$   refers to the fractional part. Furthermore, the notation $\left[\beta\right]$ stands for $\lfloor\beta\rfloor$ if $\beta$ is an integer and $\lfloor\beta\rfloor+1$ otherwise. The integral test for convergence is needed to compare various sums to the  three classical integrals
\begin{align}\label{indetitiesintexp}
\int_0^{\infty}e^{-\alpha x^2}\cdot \textrm{d}x=\sqrt{\frac{\pi}{4\alpha}},  \quad \int_0^{\infty}x\cdot e^{-\alpha x^2}\cdot \textrm{d}x=\frac{1}{2\alpha}\;\;  \& \;\; \int_0^{\infty}x^2\cdot e^{-\alpha x^2}\cdot \textrm{d}x=\frac{\sqrt{\pi}}{4(\alpha^{3/2})},
\end{align}
where $\alpha>0$.

\subsection{The Boundary Contribution} 

\begin{lem}[Asymptotic behavior of the boundary subsum $\partial \mathcal{S}_{\xi}^{(\bm{\gamma})}\left(\bm{\mathcal{T}}; q\right)$]\label{asymptinhomg2} Recall that by definition
\begin{equation}\label{border1}
\partial \mathcal{S}_{\xi}^{(\bm{\gamma})}\left(\bm{\mathcal{T}}; q\right)\;=\; \sum_{\frac{\bm{s}}{q}\in \partial \Q_{\xi}^ {(\bm{\gamma})} (\bm{\mathcal{T}})}\ppi_\xi^{(\bm{\gamma})}\!\left(\bm{\mathcal{T}}; \frac{\bm{s}}{q}\right).
\end{equation}
This subsum meets the asymptotic bound
\begin{align*}
\partial\mathcal{S}_{\xi}^{(\bm{\gamma})}\left(\bm{\mathcal{T}}; q\right)\;\ll \; &\sqrt{\xi(q)}\cdot\left(1+\sqrt{q}\cdot\left(\sqrt{T_2}-\sqrt{T_1}\right)\right)
+\mathcal{E}_{\xi}^{\left(\bm{\gamma}\right)}\left(\bm{\mathcal{T}}, q\right).
\end{align*} 
Here, 
the so--called conditional term $\mathcal{E}_{\xi}^{\left(\bm{\gamma}\right)}\left(\bm{\mathcal{T}}, q\right)$ decomposes into the three terms 

$$\mathcal{E}_{\xi}^{\left(\bm{\gamma}\right)}\left(\bm{\mathcal{T}}, q\right)  \;=\; \Lambda_{\xi}^{\left(\bm{\gamma}, 1\right)}\left(\bm{\mathcal{T}}, q\right) \;+\; \Lambda_{\xi}^{\left(\bm{\gamma}, 2\right)}\left(\bm{\mathcal{T}}, q\right)  \;+\; \Lambda_{\xi}^{\left(\bm{\gamma}, 3\right)}\left(\bm{\mathcal{T}}, q\right),$$ where each $\Lambda_{\xi}^{\left(\bm{\gamma}, i\right)}\left(\bm{\mathcal{T}}, q\right)$, $1\le i\le 3$, is defined as part of Proposition~\ref{asymptinhomg}. 
\end{lem}

\begin{proof}
Decompose the set of boundary rationals $\partial \Q_{\xi}^ {(\bm{\gamma})} (\bm{\mathcal{T}})$ defined in~\eqref{decomCQT0} into four subsets~:
$$ \partial \Q_{\xi}^ {(\bm{\gamma})} (\bm{\mathcal{T}})\;=\;  \bigcup_{i=1}^{4} \partial \Q_{\xi}^ {(\bm{\gamma}, i)} (\bm{\mathcal{T}}).
$$ 
Here, 
\begin{align*}
\partial \Q_{\xi}^ {(\bm{\gamma}, 1)} (\bm{\mathcal{T}}) \;&=\; \left\{\frac{\bm{s}}{q} \in \Q_{\xi}^ {(\gamma_1)} (\bm{\mathcal{T}}) \;:\;   \max\left\{T_1, \; \frac{s_1}{q}+\gamma_1-\frac{\xi(q)}{2q}\right\}\;\le\; \frac{\xi(q)}{q}\right\},\\
 \partial \Q_{\xi}^ {(\bm{\gamma}, 2)} (\bm{\mathcal{T}}) \;&=\; \left\{\frac{\bm{s}}{q}\in \Q_{\xi}^ {(\gamma_1)} (\bm{\mathcal{T}})\;\backslash\; \partial \Q_{\xi}^ {(\bm{\gamma}, 1)} (\bm{\mathcal{T}}) \;:\;   \left|\frac{s_2}{q}+\gamma_2\right| \;\le\; \frac{\xi(q)}{2q}\right\},
 \end{align*}
 \begin{align*}
 \partial \Q_{\xi}^ {(\bm{\gamma}, 3)} (\bm{\mathcal{T}}) \;&=\; \left\{\frac{\bm{s}}{q}\in \Q_{\xi}^ {(\gamma_1)} (\bm{\mathcal{T}})\;\backslash\; \left(\bigcup_{i=1}^{2}\partial \Q_{\xi}^ {(\bm{\gamma}, i)} (\bm{\mathcal{T}}) \right) \;:\; \right. \\ 
 & \qquad  \qquad  \qquad \qquad\qquad  \qquad  \qquad\left. \frac{s_1}{q}+\gamma_1-\frac{\xi(q)}{2q}<T_1<\frac{s_1}{q}+\gamma_1+\frac{\xi(q)}{2q}\right\},
 \end{align*}
 and finally
\begin{align*}
  \partial \Q_{\xi}^ {(\bm{\gamma}, 4)} (\bm{\mathcal{T}}) \;&=\; \left\{\frac{\bm{s}}{q}\in \Q_{\xi}^ {(\gamma_1)} (\bm{\mathcal{T}})\;\backslash\; \left(\bigcup_{i=1}^{3}\partial \Q_{\xi}^ {(\bm{\gamma}, i)} (\bm{\mathcal{T}}) \right) \;:\; \right. \\ 
 & \qquad  \qquad  \qquad \qquad\qquad  \qquad  \qquad\left. \frac{s_1}{q}+\gamma_1-\frac{\xi(q)}{2q}<T_2<\frac{s_1}{q}+\gamma_1+\frac{\xi(q)}{2q}\right\}.
  \end{align*}
 These sets are not guaranteed to be nonempty. Denote by  $\partial \mathcal{S}_{\xi}^{(\bm{\gamma}, i)}\left(\bm{\mathcal{T}}; q\right)$, where $1\le i\le 4$, the corresponding subsums (defined as the restrictions of the sum on the right--hand side of~\eqref{border1}  to each of the above subsets, respectively). Clearly,  
\begin{equation}\label{decoboundsubsum}
\partial \mathcal{S}_{\xi}^{(\bm{\gamma})}\left(\bm{\mathcal{T}}; q\right)\;\le\;  \sum_{i=1}^{4} \partial \mathcal{S}_{\xi}^{(\bm{\gamma}, i)}\left(\bm{\mathcal{T}}; q\right).
\end{equation}
Each of these terms is analysed separately. 

\paragraph{$\bullet$ Sum over the set $\partial \Q_{\xi}^ {(\bm{\gamma}, 1)} (\bm{\mathcal{T}}) $~: } For the set  $  \partial \Q_{\xi}^ {(\bm{\gamma}, 1)} (\bm{\mathcal{T}}) $ to be nonempty, it is necessary that 
\begin{equation}\label{assumptioacse1}
 T_1\;<\;\xi(q)/q \qquad \textrm{and}\qquad \frac{s_1}{q}+\gamma_1\;<\;3\frac{\xi(q)}{2q}\cdotp 
\end{equation}
Since  $0\le \gamma_1<1/q$ (from the assumption~\eqref{eqshift}) and since the map  $\xi$ is throughout assumed to take values in $[0,1]$, this implies in particular that $s_1\in\left\{0,1\right\}$ (recall here that $s_1\ge 0$ from the relations~\eqref{convention}).  The definition of the rectangle $\mathcal{R}_\xi^{(\bm{\gamma})}\!\left(\bm{\mathcal{T}}; \frac{\bm{s}}{q}\right) $ in~\eqref{inhomogrect} then guarantees that whenever $\bm{s}/q\in\partial \Q_{\xi}^ {(\bm{\gamma}, 1)} (\bm{\mathcal{T}})$, 
\begin{equation*}\label{inclurectbord}
\mathcal{R}_\xi^{(\bm{\gamma})}\!\left(\bm{\mathcal{T}}; \frac{\bm{s}}{q}\right) \;\subset\; \left[T_1,\; T_1+ 2\cdot\frac{\xi(q)}{q}\right] \times\,\left[\frac{s_2}{q}+\gamma_2-\frac{\xi(q)}{2q}, \; \frac{s_2}{q}+\gamma_2+\frac{\xi(q)}{2q}\right].
 \end{equation*}
Apply Corollary~\ref{coroutile} with the choice of parameters 
 \begin{align*}
 x^-\;=\; T_1,\qquad x^+\;=\; T_1+\frac{2\xi(q)}{q}, \qquad \alpha=\frac{2\xi(q)}{q}
 \end{align*}
 and also
 \begin{align}\label{casesvert}
 \begin{cases}
 y^\pm\;=\;\left|\frac{s_2}{q}+\gamma_2\right|\pm\frac{\xi(q)}{2q}\quad \textrm{and} \quad \beta=\frac{\xi(q)}{q} \; &\textrm{when} \quad \left|\frac{s_2}{q}+\gamma_2\right|>\frac{\xi(q)}{2q},\\
 y^-= 0, \quad  \textrm{and} \quad  y^+= \beta= \max\left\{\left|\frac{s_2}{q}+\gamma_2\pm \frac{\xi(q)}{2q}\right|\right\} \; &\textrm{when} \quad \left|\frac{s_2}{q}+\gamma_2\right|\le\frac{\xi(q)}{2q}.
 \end{cases}
 \end{align}
Relying on Case~4.(a) in Proposition~\ref{proppasssqbis} to deal with the  second possibility (which may never occur and, when it does, imposes that   $s_2=0$ under the assumption~\eqref{eqshift} and the assumption that the approximation function $\xi$ takes values in   $(0,1)$), one obtains that 
\begin{align*}
\partial \mathcal{S}_{\xi}^{(\bm{\gamma}, 1)}\left(\bm{\mathcal{T}}; q\right)\;&=\; \sum_{\frac{\bm{s}}{q}\in \partial \Q_{\xi}^ {(\bm{\gamma}, 1)} (\bm{\mathcal{T}})}\ppi_\xi^{(\bm{\gamma})}\!\left(\bm{\mathcal{T}}; \frac{\bm{s}}{q}\right)\\
&\ll \; \min\left\{1, \sqrt{\frac{\xi(q)}{q\cdot T_1}}\right\}\cdot \sum_{\underset{\left|\frac{s_2}{q}+\gamma_2\right|>\frac{\xi(q)}{2q}}{s_2\in\Z~:}}\exp\left(-\frac{\left(\left|\frac{s_2}{q}+\gamma_2\right|-\frac{\xi(q)}{2q}\right)^2}{2T_1}\right)\\
& \qquad +\left[ \min\left\{1, \sqrt{\frac{\xi(q)}{q\cdot T_1}}\right\}\;+\;\frac{1}{\sqrt{T_1}}\cdot\max\left\{\left|\gamma_2\pm \frac{\xi(q)}{2q}\right|\right\} \right]_{(\gamma_2, \xi)}^*.
\end{align*}
Here, the bracket-star notation $\left[\;\;\right]^*_{{(\gamma_2, \xi)}}$ means that the corresponding term vanishes unless 
$\left|\gamma_2\right|\le\xi(q)/2q$. 
With the help of the identities~\eqref{indetitiesintexp}, the sum on the right--hand side of the above relation compares as
\begin{align}
&\sum_{\underset{\left|\frac{s_2}{q}+\gamma_2\right|>\frac{\xi(q)}{2q}}{s_2\in\Z~:}}\exp\left(-\frac{\left(\left|\frac{s_2}{q}+\gamma_2\right|-\frac{\xi(q)}{2q}\right)^2}{2T_1}\right)\nonumber\\
&\quad \quad  \le\; \exp\left(-\frac{\left\{\frac{\xi(q)}{2}-q\gamma_2\right\}^2}{2q^2T_1}\right)+\exp\left(-\frac{\left\{\frac{\xi(q)}{2}\;+\; q\gamma_2\right\}^2}{2q^2T_1}\right)+q\sqrt{2\pi T_1}.\label{stimsubsums21}
\end{align}
Note also  that under the assumption that $\left|\gamma_2\right|\le\xi(q)/(2q)$, 
one has that $\max\left\{\left| \gamma_2\pm \xi(q)/(2q)\right|\right\}
\le \xi(q)/q$. As a consequence, since $T_1\le \xi(q)/q$ in this regime,
\begin{align}
\partial \mathcal{S}_{\xi}^{(\bm{\gamma}, 1)}\left(\bm{\mathcal{T}}; q\right)\;& \ll \;  q\sqrt{T_1}+1  + \left[\frac{\xi(q)}{q\sqrt{T_1}}\right]_{(\gamma_2, \xi)}^*,\label{subsbound1}
\end{align}
which upper bound is valid whenever the assumption~\eqref{assumptioacse1} holds.

\paragraph{$\bullet$ Sum over the set $\partial \Q_{\xi}^ {(\bm{\gamma}, 2)} (\bm{\mathcal{T}}) $~: } the set $\partial \Q_{\xi}^ {(\bm{\gamma}, 2)} (\bm{\mathcal{T}}) $ is nonempty precisely when 
$\left|\gamma_2\right|\le\xi(q)/(2q)$, in which case 
\begin{align}
\mathcal{R}_\xi^{(\bm{\gamma})}\!\left(\bm{\mathcal{T}}; \frac{\bm{s}}{q}\right) \;&\subset\; \left[\frac{s_1}{q}+\gamma_1-\frac{\xi(q)}{2q},\; \frac{s_1}{q}+\gamma_1+\frac{\xi(q)}{2q}\right] \times\,\left[\frac{s_2}{q}+\gamma_2-\frac{\xi(q)}{2q}, \; \frac{s_2}{q}+\gamma_2+\frac{\xi(q)}{2q}\right] \nonumber\\
&\underset{\eqref{deinhomogstrip}}{\subset} \left[T_1-\frac{\xi(q)}{q},\; T_2+\frac{\xi(q)}{q}\right] \times\,\left[ \gamma_2-\frac{\xi(q)}{2q}, \;  \gamma_2+\frac{\xi(q)}{2q}\right]. \label{inclusubcase2}
\end{align} 
Apply Corollary~\ref{coroutile} (with the help of Case~4.(a) in Proposition~\ref{proppasssqbis}) with the choice of parameters 
 \begin{align*}
 x^\pm\;=\; \frac{s_1}{q}+\gamma_1\pm\frac{\xi(q)}{2q} \qquad \textrm{ and }\qquad \alpha\;=\;\frac{2\xi(q)}{q},
 \end{align*}
 and also $y^\pm$ and $\beta$ as in the distinction of cases~\eqref{casesvert} with $s_2=0$. Note that $x^->0$ as otherwise the inclusion~\eqref{inclusubcase2} would mean that $ \bm{s}/q\in \partial \Q_{\xi}^ {(\bm{\gamma}, 1)} (\bm{\mathcal{T}}) $, which is ruled out by the definition of the set $  \partial \Q_{\xi}^ {(\bm{\gamma}, 2)} (\bm{\mathcal{T}}) $.  Then, upon setting 
 \begin{align}\label{defs*}
 s^*_1\;=\; \left\lfloor \frac{3}{2}\cdot \xi(q)-q\gamma_1\right\rfloor ,
 \end{align}
one obtains, as before,
 \begin{align}
&\partial \mathcal{S}_{\xi}^{(\bm{\gamma}, 2)}\left(\bm{\mathcal{T}}; q\right)\;=\; \sum_{\frac{\bm{s}}{q}\in \partial \Q_{\xi}^ {(\bm{\gamma}, 2)} (\bm{\mathcal{T}})}\ppi_\xi^{(\bm{\gamma})}\!\left(\bm{\mathcal{T}}; \frac{\bm{s}}{q}\right)\nonumber\\
&\ll \; \sum_{\underset{q\left(T_1-\gamma_1\right)-\xi(q)\le s_1\le q\left(T_2-\gamma_1\right)+\xi(q)}{s_1\ge 0~:}}\left( \min\left\{1, \sqrt{\frac{\xi(q)}{q\cdot \left(\frac{s_1}{q}+\gamma_1-\frac{\xi(q)}{2q}\right)}}\right\}+\frac{\beta}{\sqrt{\frac{s_1}{q}+\gamma_1+\frac{\xi(q)}{2q}}}\right)\nonumber\\
&\ll \left[1\right]^{**}_{\left(q, T_1, \gamma_1, \xi\right)}\;+\; \sqrt{ \xi(q)}\cdot \sum_{\underset{q\left(T_1-\gamma_1\right)-\xi(q)\le s_1\le q\left(T_2-\gamma_1\right)+\xi(q)}{s_1>s_1^*~: }}\frac{1}{\sqrt{s_1+q\gamma_1-\xi(q)/2}}\nonumber\\
&\ll \left[1\right]^{**}_{\left(q, T_1, \gamma_1, \xi\right)}\;+\;  \sqrt{\xi(q)}\cdot \left(1+\sqrt{q}\cdot\left(\sqrt{T_2}-\sqrt{T_1}\right)\right).\label{subsbound2}
\end{align}
 Here, the bracket-double-star notation $\left[\;\;\right]^{**}_{\left(q, T_1, \gamma_1, \xi\right)}$ means that the corresponding term vanishes unless 
 \begin{equation}\label{constras*}
s_1^*\;\ge\; q\left(T_1-\gamma_1\right)-\xi(q).
 \end{equation}
 
 \paragraph{$\bullet$ Sum over the set $\partial \Q_{\xi}^ {(\bm{\gamma}, 3)} (\bm{\mathcal{T}}) $~: } when $\bm{s}/q \in \partial \Q_{\xi}^ {(\bm{\gamma}, 3)} (\bm{\mathcal{T}}) $, 
\begin{align}
\mathcal{R}_\xi^{(\bm{\gamma})}\!\left(\bm{\mathcal{T}}; \frac{\bm{s}}{q}\right) \subset  \left[T_1,\; T_1+ \frac{\xi(q)}{q}\right] \times\,\left[\frac{s_2}{q}+\gamma_2-\frac{\xi(q)}{2q}, \; \frac{s_2}{q}+\gamma_2+\frac{\xi(q)}{2q}\right] \subset \R_{>0} \times\,\left(\R\backslash\left\{0\right\}\right). \label{inclusubcase3}
\end{align} 
Apply Corollary~\ref{coroutile} with the choice of parameters 
 \begin{align*}
 x^-\;=\; T_1,\qquad x^+\;=\; T_1+\frac{\xi(q)}{q}, \qquad \alpha=\frac{\xi(q)}{q}
 \end{align*}
 and also
 \begin{align*}
 y^\pm\;=\;\left|\frac{s_2}{q}+\gamma_2\right|\pm\frac{\xi(q)}{2q}\qquad \textrm{ and } \qquad \beta=\frac{\xi(q)}{q}\cdotp
 \end{align*}
 The calculations carried out for the sum over the set $\partial \Q_{\xi}^ {(\bm{\gamma}, 1)} (\bm{\mathcal{T}}) $ can be repeated  so as to yield that 
\begin{align}
\partial \mathcal{S}_{\xi}^{(\bm{\gamma}, 3)}&\left(\bm{\mathcal{T}}; q\right)\;=\; \sum_{\frac{\bm{s}}{q}\in \partial \Q_{\xi}^ {(\bm{\gamma}, 3)} (\bm{\mathcal{T}})}\ppi_\xi^{(\bm{\gamma})}\!\left(\bm{\mathcal{T}}; \frac{\bm{s}}{q}\right)\nonumber\\
&\ll \; \min\left\{1, \sqrt{\frac{\xi(q)}{q\cdot T_1}}\right\}\cdot \sum_{\underset{\left|\frac{s_2}{q}+\gamma_2\right|>\frac{\xi(q)}{2q}}{s_2\in\Z~:}}\exp\left(-\frac{\left(\left|\frac{s_2}{q}+\gamma_2\right|-\frac{\xi(q)}{2q}\right)^2}{2T_1}\right)\nonumber\\
&\underset{\eqref{stimsubsums21}}{\ll}\; \min\left\{1, \sqrt{\frac{\xi(q)}{q\cdot T_1}}\right\}\cdot \left(q\sqrt{T_1}+\max\left\{\exp\left(-\frac{\left\{\xi(q)/2\pm q\gamma_2\right\}^2}{2q^2T_1}\right)\right\}\right).\label{subsbound3}
\end{align}

 \paragraph{$\bullet$ Sum over the set $\partial \Q_{\xi}^ {(\bm{\gamma}, 4)} (\bm{\mathcal{T}}) $~: } when $\bm{s}/q \in \partial \Q_{\xi}^ {(\bm{\gamma}, 4)} (\bm{\mathcal{T}}) $, 
\begin{align}
\mathcal{R}_\xi^{(\bm{\gamma})}\!\left(\bm{\mathcal{T}}; \frac{\bm{s}}{q}\right) \subset  \left[T_2-\frac{\xi(q)}{q},\; T_2\right] \times\,\left[\frac{s_2}{q}+\gamma_2-\frac{\xi(q)}{2q}, \; \frac{s_2}{q}+\gamma_2+\frac{\xi(q)}{2q}\right] \subset \R \times\,\left(\R\backslash\left\{0\right\}\right). \label{inclusubcase4}
\end{align} 
If, furthermore, $T_2-\frac{\xi(q)}{q}<T_1$, then the inclusions~\eqref{inclusubcase3} part of the previously treated case are here again valid and the corresponding estimate~\eqref{subsbound3} holds. Otherwise, 
\begin{equation}\label{ineqT_2xiT_1}
T_2-\frac{\xi(q)}{q}\ge T_1
\end{equation} 
and, in view of the above inclusions~\eqref{inclusubcase4}, it suffices to reproduce the calculations of the previous case upon switching the role of the pair $\left(T_1, T_1+\xi(q)/q\right)$ to that of the pair $\left(T_2-\frac{\xi(q)}{q}, T_2\right)$. All in all, this gives that 
\begin{align}
&\partial \mathcal{S}_{\xi}^{(\bm{\gamma}, 4)}\left(\bm{\mathcal{T}}; q\right)\;=\; \sum_{\frac{\bm{s}}{q}\in \partial \Q_{\xi}^ {(\bm{\gamma}, 4)} (\bm{\mathcal{T}}) }\ppi_\xi^{(\bm{\gamma})}\!\left(\bm{\mathcal{T}}; \frac{\bm{s}}{q}\right)\nonumber\\
&\ll\; \min\left\{1, \sqrt{\frac{\xi(q)}{q\cdot T_1}}\right\}\cdot \left(q\sqrt{T_1}+\max\left\{\exp\left(-\frac{\left\{\xi(q)/2\pm q\gamma_2\right\}^2}{2q^2T_1}\right)\right\}\right)\;+\; \nonumber\\
&\left[\min\!\left\{1, \sqrt{\frac{\xi(q)}{q\cdot T_2-\xi(q)}}\right\}\cdot \left(q\sqrt{T_2-\frac{\xi(q)}{q}}+\max\!\left\{\exp\left(-\frac{\left\{\xi(q)/2\pm q\gamma_2\right\}^2}{2q\left(qT_2-\xi(q)\right)}\right)\right\}\right)\right]^{\dagger}_{\left(\bm{\mathcal{T}}, q, \xi\right)},
\label{subsbound4}
\end{align}
where the dagger--bracket notation $\left[\;\;\right]^{\dagger}_{\left(\bm{\mathcal{T}}, q, \xi\right)}$ means that the corresponding term va\-ni\-shes unless the inequality~\eqref{ineqT_2xiT_1} is met.

\paragraph{$\bullet$ } The statement of the lemma follows from the decomposition~\eqref{decoboundsubsum} for the boundary sum $\partial \mathcal{S}_{\xi}^{(\bm{\gamma})}\left(\bm{\mathcal{T}}; q\right)$ and from the upper bounds~\eqref{subsbound1} (which is valid under the assumption~\eqref{assumptioacse1}), \eqref{subsbound2}, \eqref{subsbound3} and \eqref{subsbound4} 
for each of the subsums  $\partial \mathcal{S}_{\xi}^{(\bm{\gamma}, i)}\left(\bm{\mathcal{T}}; q\right)$, $1\le i\le 4$, respectively.
\end{proof}

\subsection{The Main Contribution}

\begin{lem}[Asymptotic behavior of the main subsum $ \widehat{\mathcal{S}}_{\xi}^{(\bm{\gamma})}\left(\bm{\mathcal{T}}; q\right)$]\label{asymptinhomg3} Recall that by definition
\begin{equation*}\label{border2}
\widehat{\mathcal{S}}_{\xi}^{(\bm{\gamma})}\left(\bm{\mathcal{T}}; q\right)\;=\; \sum_{\frac{\bm{s}}{q}\in \widehat{\Q}_{\xi}^ {(\gamma)}( \bm{\mathcal{T}})}\ppi_\xi^{(\bm{\gamma})}\!\left(\bm{\mathcal{T}}; \frac{\bm{s}}{q}\right),
\end{equation*} 
where the set of rationals $\widehat{\Q}_{\xi}^ {(\gamma)}( \bm{\mathcal{T}})$ is defined in~\eqref{decomCQT1}. 
Then, this subsum meets the asymptotic expansion
\begin{align*}
\widehat{\mathcal{S}}_{\xi}^{(\bm{\gamma})}\left(\bm{\mathcal{T}}; q\right)\;&=\; 2\sqrt{\frac{2}{\pi}}\cdot \sqrt{q^3\xi(q)}\cdot \left(T_2-T_1\right)  \;+\; O\!\!\left(\Lambda_{\xi}^{\left(\bm{\gamma}, 4\right)}\left(\bm{\mathcal{T}}, q\right)\right)\\
&+O\!\!\left(\sqrt{q\xi(q)}\cdot\left(\left(\sqrt{q\xi(q)}+1\right)\cdot\left(T_2-T_1+1\right)+\left|\log\left(\frac{T_2}{\max\left\{T_1, 1/q\right\}}\right)\right|\right)\right),
\end{align*}
where the conditional term $\Lambda_{\xi}^{\left(\bm{\gamma}, 4\right)}\left(\bm{\mathcal{T}}, q\right)$ is the one introduced in the statement of Proposition~\ref{asymptinhomg}.
\end{lem}

\begin{proof}
Decompose the sum $ \widehat{\mathcal{S}}_{\xi}^{(\bm{\gamma})}\left(\bm{\mathcal{T}}; q\right)$ as 
\begin{equation}\label{decompospmhat} 
\widehat{\mathcal{S}}_{\xi}^{(\bm{\gamma})}\left(\bm{\mathcal{T}}; q\right)\;=\;  \widehat{\mathcal{S}}_{\xi}^{(\bm{\gamma}, +)}\left(\bm{\mathcal{T}}; q\right)\;+\;  \widehat{\mathcal{S}}_{\xi}^{(\bm{\gamma}, -)}\left(\bm{\mathcal{T}}; q\right), 
\end{equation}
 where 
\begin{equation}\label{defspmhat}
 \widehat{\mathcal{S}}_{\xi}^{(\bm{\gamma}, \pm)}\left(\bm{\mathcal{T}}; q\right)\;=\; \sum_{\underset{s_2\square 0}{\frac{\bm{s}}{q}\in  \widehat{\Q}_{\xi}^ {(\gamma)}( \bm{\mathcal{T}})}}\ppi_\xi^{(\bm{\gamma})}\!\left(\bm{\mathcal{T}}; \frac{\bm{s}}{q}\right).
\end{equation}
Here, $\square$ is the symbol $>$ if the sign is positive and is the symbol $<$ otherwise. It should be clear that in view of the definition of the set $\widehat{\Q}_{\xi}^ {(\gamma)}( \bm{\mathcal{T}})$  that 
\begin{equation}\label{s+=s-hat} 
\widehat{\mathcal{S}}_{\xi}^{(\bm{\gamma}, +)}\left(\bm{\mathcal{T}}; q\right)\;=\;   \widehat{\mathcal{S}}_{\xi}^{(\bm{\gamma}, -)}\left(\bm{\mathcal{T}}; q\right).
\end{equation}
Consider therefore only the sum $\widehat{\mathcal{S}}_{\xi}^{(\bm{\gamma}, +)}\left(\bm{\mathcal{T}}; q\right)$ and set 
 \begin{align*}
T_1< x^\pm= \frac{s_1}{q}+\gamma_1\pm\frac{\xi(q)}{2q}<T_2, \quad  \alpha=\beta= \frac{\xi(q)}{q}\qquad \textrm{and}\qquad y^\pm=\frac{s_2}{q}+\gamma_2\pm\frac{\xi(q)}{2q}>0.
 \end{align*}
When $\bm{s}/q \in \widehat{\Q}_{\xi}^ {(\gamma)}( \bm{\mathcal{T}}) $ and $s_2>0$, the refined estimate established in  Corollary~\ref{coroutile} reduces with this choice of parameters to 
\begin{equation}\label{probatsqgxi}
\ppi_\xi^{(\bm{\gamma})}\!\left(\bm{\mathcal{T}}; \frac{\bm{s}}{q}\right)= \frac{2}{\pi}\sqrt{\frac{\alpha}{x^-}}\cdot \exp\left(-\frac{\left(y^-\right)^2}{2x^+}\right)\cdot\left(1+O\!\!\left(\sqrt{\alpha}\left(1+\frac{\sqrt{\alpha}}{x^-}\left(1+y^++\frac{\left(y^+\right)^2}{x^-}\right)\right)\right)\right).
\end{equation}
The sum $\widehat{\mathcal{S}}_{\xi}^{(\bm{\gamma}, +)}\left(\bm{\mathcal{T}}; q\right)$ is then estimated in several steps from its definition in~\eqref{defspmhat}.

\paragraph{$\bullet$} The sum over the integers $s_2>0$ involves three subsums (related to the quantities $y^\pm$), each of which compares to an integral in the identities~\eqref{indetitiesintexp}. 
Given an integer $s_1>0$, taking into account the definition of the set $\widehat{\Q}_{\xi}^ {(\gamma)}( \bm{\mathcal{T}})$ in~\eqref{decomCQT1}, and upon setting for the sake of simplicity of notation 
\begin{equation*}
f(s_2)\;=\; \exp\left(-\frac{\left(s_2/q+\gamma_2 - \xi(q)/(2q)\right)^2}{2\left( s_1/q+\gamma_1+\xi(q)/(2q)\right)}\right),
\end{equation*} 
the first one is
\begin{align}
\sum_{\underset{\eqref{decomCQT1}}{s_2>0~:}} f(s_2)\;&=\; \int_{0}^{\infty} f \;+\;O\left( f(0)\right)
& \underset{\eqref{indetitiesintexp}}{=}\; \sqrt{\frac{\pi}{2}\cdot q\cdot\left(s_1+q\gamma_1+\frac{\xi(q)}{2}\right)} \;+\;O\left(1  \right).
\label{subestim1princi}
\end{align} 

\noindent The second and third subsums are  
\begin{align*}
\sum_{\underset{\eqref{decomCQT1}}{s_2>0~:}}  \left(\frac{s_2}{q}+\gamma_2 + \frac{\xi(q)}{2q}\right)\cdot f\left(s_2\right) \qquad \textrm{and}\qquad \sum_{\underset{\eqref{decomCQT1}}{s_2>0~:}}  \left(\frac{s_2}{q}+\gamma_2 + \frac{\xi(q)}{2q}\right)^2\cdot f\left(s_2\right).
\end{align*} 
A comparison with the second and third integrals in the identities~\eqref{indetitiesintexp} yields after a short calculation that 
\begin{align}
\sum_{\underset{\eqref{decomCQT1}}{s_2>0~:}}  \left(\frac{s_2}{q}+\gamma_2 + \frac{\xi(q)}{2q}\right)\cdot f\left(s_2\right)\;\ll\; 1+q\left(\frac{s_1}{q}+\gamma_1+\frac{\xi(q)}{q}\right).
\label{subestim2princi}
\end{align} 
and 
\begin{align}
\sum_{\underset{\eqref{decomCQT1}}{s_2>0~:}} \left(\frac{s_2}{q}+\gamma_2 + \frac{\xi(q)}{2q}\right)^2\cdot f\left(s_2\right)\;\ll\;  1+q\left(\frac{s_1}{q}+\gamma_1+\frac{\xi(q)}{q}\right)^{\frac{3}{2}}. 
\label{subestim3princi}
\end{align} 

\paragraph{$\bullet$} In view of the asymptotic expansion~\eqref{probatsqgxi} for the probability $\ppi_\xi^{(\bm{\gamma})}\!\left(\bm{\mathcal{T}}; \frac{\bm{s}}{q}\right)$, and given the just established estimations~\eqref{subestim1princi}, \eqref{subestim2princi}  and~\eqref{subestim3princi}, the sum $\widehat{\mathcal{S}}_{\xi}^{(\bm{\gamma}, +)}\left(\bm{\mathcal{T}}; q\right)$ defined in~\eqref{defspmhat} becomes 
\begin{align*}
\widehat{\mathcal{S}}_{\xi}^{(\bm{\gamma}, +)}\left(\bm{\mathcal{T}}; q\right)=& \sqrt{\frac{2}{\pi}\cdot q\xi(q)}\cdot\sum_{\underset{\eqref{decomCQT1}}{s_1\neq 0~:}}\sqrt{\frac{s_1+q\gamma_1+\xi(q)/2}{s_1+q\gamma_1-\xi(q)/2}}+O\!\!\left(\sum_{\underset{\eqref{decomCQT1}}{s_1\neq 0~:}}\sqrt{\frac{\xi(q)}{s_1+q\gamma_1-\xi(q)/2}}\right)\\
&+O\!\!\left(\sum_{\underset{\eqref{decomCQT1}}{s_1\neq 0~:}}\xi(q)\cdot\left(1+\frac{\sqrt{q\xi(q)}}{s_1+q\gamma_1-\xi(q)/2}\right)\right) +O\!\!\left(\Lambda_{\xi}^{\left(\bm{\gamma}, 4\right)}\left(\bm{\mathcal{T}}, q\right)\right),
\end{align*}
where the conditional term $\Lambda_{\xi}^{\left(\bm{\gamma}, 4\right)}\left(\bm{\mathcal{T}}, q\right)$ is the one defined in Proposition~\ref{asymptinhomg}. 
It corresponds to the contribution of the term $s_1=0$, whenever it belongs to the set $\widehat{\Q}_{\xi}^ {(\gamma)}( \bm{\mathcal{T}})$. From the definition of this set in~\eqref{decomCQT1}, when $s_1\neq 0$, 
\begin{align*}
 \max\left\{1, \; q\left(\max\left\{T_1, \frac{\xi(q)}{q}\right\}-\gamma_1\right)+\frac{\xi(q)}{2}\right\}\;
 \le\; s_1\;\le\; q\left(T_2-\gamma_1\right)-\frac{\xi(q)}{2}\cdot
\end{align*}
As a consequence, after calculations, 
\begin{align*}
\widehat{\mathcal{S}}_{\xi}^{(\bm{\gamma}, +)}\left(\bm{\mathcal{T}}; q\right)\;=&\; \sqrt{\frac{2}{\pi}\cdot q^3\xi(q)}\cdot \left(T_2-T_1\right)  \;+\; O\!\!\left(\Lambda_{\xi}^{\left(\bm{\gamma}, 4\right)}\left(\bm{\mathcal{T}}, q\right)\right)\\
&+O\!\!\left(\sqrt{q\xi(q)}\cdot\left(\left(\sqrt{q\xi(q)}+1\right)\cdot\left(T_2-T_1+1\right)+\left|\log\left(\frac{T_2}{\max\left\{T_1, 1/q\right\}}\right)\right|\right)\right).
\end{align*}
This completes the proof of the lemma in view of the decomposition~\eqref{decompospmhat} and of the identity~\eqref{s+=s-hat}.
\end{proof}

\subsection{Asymptotic Expansions of the Inhomogeneous Sums.}\label{subsubasymexpains} 

\begin{proof}[Completion of the proof of Proposition~\ref{asymptinhomg}]
This is just a restatement of  the conclusions of Lemmata~\ref{asymptinhomg2} and~\ref{asymptinhomg3}.
\end{proof}

\begin{proof}[Proof of Corollary~\ref{asymptinhomgbis}] All the  calculations in the proof of Proposition~\ref{asymptinhomg} can be carried out  in the simpler case where the inhomogeneous term $\bm{\gamma}$ is set to vanish upon making the following changes~:
\begin{itemize}
\item from their very definitions, each term in the sum defining $\mathcal{O}_{ \xi}\!\left(\bm{\mathcal{T}}; q\right)$ is at most $T_2$ times the corresponding term in the sum defining $ \mathcal{S}_{\xi}^{(\bm{0})}\!\left(\bm{\mathcal{T}}; q\right)$. When reproducing the above calculations, the error term for $\mathcal{O}_{ \xi}\!\left(\bm{\mathcal{T}}; q\right)$ is thus plainly at most $T_2$ times the error term established for $ \mathcal{S}_{\xi}^{(\bm{0})}\!\left(\bm{\mathcal{T}}; q\right)$;\\
\item as far as the leading term is concerned, following the above calculations, it equals twice the quantity 
\begin{align*}
\sqrt{\frac{2}{\pi}\cdot q\xi(q)}\cdot\sum_{\underset{\eqref{decomCQT1}}{s_1\neq 0~:}}&\sqrt{\frac{s_1+q\gamma_1+\xi(q)/2}{s_1+q\gamma_1-\xi(q)/2}}\cdot \frac{s_1}{q}\;= \frac{1}{2}\cdot\frac{\sqrt{2}}{\sqrt{\pi}}\cdot\sqrt{q^3\xi(q)}\cdot\left(T_2^2-T_1^2\right)\\ 
& +\; \frac{1}{2}\cdot\frac{\sqrt{2}}{\sqrt{\pi}}\cdot\sqrt{q\xi(q)^3}\cdot\left(T_2-T_1\right)+O\left(\sqrt{q\cdot \xi(q)^5}\cdot \log\left(\frac{T_2}{T_1}\right)\right),
\end{align*}
where the implicit constant is absolute as long as $q\ge 1/T_1$.
\end{itemize}
The statement follows.
\end{proof}

\section{Asymptotics of the Off--Diagonal Probability Term}\label{asymoffdiag} 

This section is devoted to the proof of Proposition~\ref{offdiagestimterm}. \\

\noindent To this end, keep all notation introduced in the previous section and, given integers $0\le Q_1<Q_2$,  define the auxiliary counting function 
\begin{equation}\label{sommeglobbis} 
\Delta_{\xi}^{(\bm{\gamma})}\left(\bm{\mathcal{T}}; \bm{\mathcal{Q}}\right)\;=\; \sum_{Q_1\le q\le Q_2} \mathcal{S}_{\xi}^{(\bm{\gamma})}\left(\bm{\mathcal{T}}; q\right).
\end{equation}
Here,  $\mathcal{S}_{\xi}^{(\bm{\gamma})}\left(\bm{\mathcal{T}}; q\right)$ is the inhomogeneous sum introduced in~\eqref{defaxilfixq} whose  asymptotic behavior is determined by Proposition~\ref{asymptinhomg}. When $\bm{\gamma}=\bm{0}$, one has that 
\begin{equation}\label{sommeglob} 
\Delta_{\xi}^{(\bm{0})}\left(\bm{\mathcal{T}}; \bm{\mathcal{Q}}\right)\;=\; \Delta_{\xi}\left(\bm{\mathcal{T}}; \bm{\mathcal{Q}}\right),
\end{equation}
where $\Delta_{\xi}\left(\bm{\mathcal{Q}}; \bm{\mathcal{T}}\right)$ is the diagonal probability term defined in~\eqref{DPT}. \\

\noindent Recall that the quantity of interest is the off--diagonal probability term $\Theta_{ \xi}\left(\bm{\mathcal{T}}; \bm{\mathcal{Q}}\right)$ --- see~\eqref{ODPT} for its definition. Its domain of summation is split into four subdomains depending on the choice of  a parameters $\mu>0$. 
More precisely, given integers $q',q\ge Q_1$, let 
\begin{align}\label{defhH} 
H_{\xi}(q,q')=  2\max\left\{\frac{\xi(q)}{q}, \frac{\xi(q')}{q'}\right\}\qquad \textrm{and}\qquad h_\xi^{(\mu)}(q,q')= H_{\xi}(q,q')^{1/2-\mu}.
\end{align}
Consider then the following four regimes and the corresponding subsums~:

\begin{itemize}
\item \textbf{Regime 1~:}  the summation is restricted to those rationals with bounded deno\-mi\-nators corresponding to pairs of squares sufficiently far apart from each other along the horizontal line. Specifically, the sum under consideration in this case is 
\begin{equation}\label{subsm1}
\Theta_{ \xi}^{(1)}\left(\bm{\mathcal{Q}}; \bm{\mathcal{T}}\right)\;=\; \sum_{Q_1\le q, q'\le Q_2}\; \sum \ppi_\xi\left(\bm{\mathcal{T}}; \left(\frac{\bm{s}}{q}, \frac{\bm{s}'}{q'}\right)\right),
\end{equation}
where   $ \ppi_\xi\left(\bm{\mathcal{T}}; \left(\bm{s}/q, \bm{s}'/q'\right)\right)$ is the probability of passing through the two squares $\mathcal{R}_{\xi}\!\left(\bm{\mathcal{T}}; \bm{s}/q\right)$ and $\mathcal{R}_{\xi}\!\left(\bm{\mathcal{T}}; \bm{s}'/q'\right)$ introduced and defined in~\eqref{defpiprobbis}, and where the inner sum is over the set of integers $\bm{s}, \bm{s}'$ such that $$\frac{\bm{s}}{q}=\left(\frac{s_1}{q}, \frac{s_2}{q}\right), \; \; \frac{\bm{s}'}{q'}=\left(\frac{s'_1}{q'}, \frac{s'_2}{q'}\right)\in \Q_{\xi}(\bm{\mathcal{T}})$$
and such that 
\begin{equation}\label{csubsm1}
\left|\frac{s'_1}{q'}- \frac{s_1}{q}\right|\;\ge\; H_{\xi}(q,q')\cdotp
\end{equation}
This is illustrated in Figure~\ref{R1}.

\item  \textbf{Regime 2~:}  the summation is restricted to those rationals with bounded deno\-mi\-nators corresponding to pairs of squares  which are close enough to each other along the horizontal axis but which are vertically well--spaced.  Specifically, the sum under consideration in this case is 
\begin{equation}\label{subsm2}
\Theta_{ \xi}^{(2)}\left(\bm{\mathcal{Q}}; \bm{\mathcal{T}}\right)\;=\; \sum_{Q_1\le q,q'\le Q_2}\; \sum \ppi_\xi\left(\bm{\mathcal{T}}; \left(\frac{\bm{s}}{q}, \frac{\bm{s}'}{q'}\right)\right),
\end{equation}
where the inner sum is over the set of integers $\bm{s}, \bm{s}'$ such that $\bm{s}/q, \bm{s'}/q'\in \Q_{\xi}(\bm{\mathcal{T}})$ and such that 
\begin{equation}\label{csubsm2}
\left|\frac{s_1}{q}-\frac{s'_1}{q'}\right|\;<\;  H_{\xi}(q,q')\qquad\textrm{and}\qquad \left|\frac{s'_2}{q'}-\frac{s_2}{q}\right|\;\ge\; h_\xi^{(\mu)}(q,q').
%
\end{equation}
This is illustrated in Figure~\ref{R2}.

\item \textbf{Regime 3~:}  the summation is restricted to those rationals with bounded deno\-mi\-nators corresponding to pairs of squares  which are close enough to each other without sharing the same center.  Specifically, the sum under consideration in this case is 
\begin{equation}\label{subsm3}
\Theta_{\xi}^{(3)}\left(\bm{\mathcal{Q}}; \bm{\mathcal{T}}\right)\;=\; \sum_{Q_1\le q,q'\le Q_2}\; \sum \ppi_\xi\left(\bm{\mathcal{T}}; \left(\frac{\bm{s}}{q}, \frac{\bm{s}'}{q'}\right)\right),
\end{equation}
where the inner sum is over the set of integers $\bm{s}, \bm{s}'$ such that the rationals $\bm{s}/q, \bm{s'}/q'\in \Q_{\xi}(\bm{\mathcal{T}})$ meet the three conditions
\begin{equation}\label{csubsm3} 
\frac{\bm{s}}{q}\neq\frac{\bm{s}'}{q'}, \qquad \left|\frac{s_1}{q}-\frac{s_1'}{q'}\right| < H_{\xi}(q,q') \qquad \textrm{and}\qquad   \left|\frac{s_2}{q}-\frac{s_2'}{q'}\right|< h_\xi^{(\mu)}(q,q')\cdot
\end{equation}
This  is illustrated in Figure~\ref{R3}.

\item \textbf{Regime 4~:} in this final situation, the summation is concerned with the intersection of squares centered at the same rational expressed in different forms. Specifically,  the sum under consideration  is 
\begin{equation}\label{subsm4} 
\Theta_{\xi}^{(4)}\left(\bm{\mathcal{Q}}; \bm{\mathcal{T}}\right)\;=\; \sum_{Q_1\le q,q'\le Q_2}\; \sum \ppi_\xi\left(\bm{\mathcal{T}}; \left(\frac{\bm{s}}{q}, \frac{\bm{s}'}{q'}\right)\right),
\end{equation}
where the inner sum is over the set of integers $\bm{s}, \bm{s}'$ such that $\bm{s}/q, \bm{s'}/q'\in \Q_{\xi}(\bm{\mathcal{T}})$ and such that 
\begin{equation*}\label{csubsm4}
\frac{\bm{s}}{q}\;=\;  \frac{\bm{s'}}{q'}\cdotp
\end{equation*}
This  is illustrated in Figure~\ref{R4}.
\end{itemize}

\begin{figure}[h!]
    \centering
    \begin{subfigure}{0.48\textwidth}
        \centering
        \begin{tikzpicture}[scale=1.8] 
    \draw[thick,->] (-0.5,0) -- (3.5,0) node[right] {\textit{u}};
    \draw[thick,->] (0,-0.5) -- (0,4) node[above] {\textit{v}};
%

    \draw[thick,blue] (0.7,1) -- (0.7,2);
    \draw[thick,blue] (1.7,1) -- (1.7,2);

    \draw[thick,blue] (0.7,1) -- (1.7,1);
    \draw[thick,blue] (0.7,2) -- (1.7,2);

    \fill[blue!15,opacity=0.5] (0.7,1) -- (0.7,2) -- (1.7,2) -- (1.7, 1) -- cycle;
    
    \draw[thick,blue] (2.4,1.95) -- (2.4,2.55);
    \draw[thick,blue] (3,1.95) -- (3,2.55);

    \draw[thick,blue] (2.4,1.95) -- (3,1.95);
    \draw[thick,blue] (2.4,2.55) -- (3,2.55);

    \fill[blue!15,opacity=0.5] (2.4,1.95) -- (2.4,2.55) -- (3,2.55) -- (3, 1.95) -- cycle;
    
    \draw[densely dotted,blue] (1.2,3.5) -- (1.2,0);    
    \draw[densely dotted,blue] (1.2,1.5) -- (0,1.5);   
    \draw[densely dotted,blue] (2.7,3.5) -- (2.7,0);    
    \draw[densely dotted,blue] (2.7,2.25) -- (0,2.25);  
                        
    \node[below] at (1.2,0) {$\frac{s_1}{q}$};
    \node[left] at (0,1.5) {$\frac{s_2}{q}$};
    \node[below] at (2.7,0) {$\frac{s'_1}{q'}$};    
    \node[left] at (0, 2.25) {$\frac{s'_2}{q'}$};    
    \node[right] at (1.9, 1.5) {$\frac{\xi(q)}{q}$}; 
    \node[right] at (3.2, 2.25) {$\frac{\xi(q')}{q'}$}; 
    
    \draw [<->] (1.9, 1) -- (1.9, 2);
    \draw [<->] (3.2, 1.95) -- (3.2, 2.55);
        
	\draw [decorate, decoration = {calligraphic brace,amplitude=10pt}] (1.2,3.6) --  (2.7,3.6);    
	\node[above] at (1.95, 3.7) {$\ge H_{\xi}(q,q')$};    
        \end{tikzpicture}
        \caption{Disposition of squares in Regime 1}
        \label{R1}
    \end{subfigure}
    \hfill
    \begin{subfigure}{0.48\textwidth}
        \centering
        \begin{tikzpicture}[scale=1.8] 
    \draw[thick,->] (-0.5,0) -- (3.5,0) node[right] {\textit{u}};
    \draw[thick,->] (0,-0.5) -- (0,4) node[above] {\textit{v}};
%

    \draw[thick,blue] (0.7,0.5) -- (0.7,1.5);
    \draw[thick,blue] (1.7,0.5) -- (1.7,1.5);

    \draw[thick,blue] (0.7,0.5) -- (1.7,0.5);
    \draw[thick,blue] (0.7,1.5) -- (1.7,1.5);

    \fill[blue!15,opacity=0.5] (0.7,0.5) -- (0.7,1.5) -- (1.7,1.5) -- (1.7, 0.5) -- cycle;
    
    \draw[thick,blue] (1.45,1.95) -- (1.45,2.55);
    \draw[thick,blue] (2.05,1.95) -- (2.05,2.55);

    \draw[thick,blue] (1.45,1.95) -- (2.05,1.95);
    \draw[thick,blue] (1.45,2.55) -- (2.05,2.55);

    \fill[blue!15,opacity=0.5] (1.45,1.95) -- (1.45,2.55) -- (2.05,2.55) -- (2.05, 1.95) -- cycle;

    \draw[densely dotted,blue] (1.75,3.5) -- (1.75,0);        
    \draw[densely dotted,blue] (1.2,1.5) -- (1.2,0);    
    \draw[densely dotted,blue] (1.2,3.5) -- (1.2,0);   
    \draw[densely dotted,blue] (3,1) -- (0,1);   
    \draw[densely dotted,blue] (1.2,1) -- (1.2,0);   
    \draw[densely dotted,blue] (1.75,2.25) -- (1.75,0);    
    \draw[densely dotted,blue] (3,2.25) -- (0,2.25);  
   \node[below] at (1.2,0) {$\frac{s_1}{q}$};
    \node[left] at (0,1) {$\frac{s_2}{q}$};
    \node[below] at (1.75,0) {$\frac{s'_1}{q'}$};    
    \node[left] at (0, 2.25) {$\frac{s'_2}{q'}$};    
    \node[right] at (1.9, 1) {$\frac{\xi(q)}{q}$};
    \node[right] at (2.25, 2.25) {$\frac{\xi(q')}{q'}$};
    \node[above] at (1.4, 3.7) {$< H_{\xi}(q,q')$};     
    \node[right] at (3.2, 1.625) {$\ge h^{(\mu)}_{\xi}(q,q')$};    
	                  
	\draw [decorate, decoration = {calligraphic brace,amplitude=5pt}] (1.2,3.6) --  (1.75,3.6);    
	\draw [decorate, decoration = {calligraphic brace,amplitude=5pt}] (3.1, 2.25) --  (3.1,1);    	

    \draw [<->] (1.9, 0.5) -- (1.9, 1.5);
    \draw [<->] (2.25, 1.95) -- (2.25, 2.55);
        \end{tikzpicture}
        \caption{Disposition of squares in Regime 2}
        \label{R2}
    \end{subfigure}

    \begin{subfigure}{0.48\textwidth}
        \vspace{8mm}
        \centering
        \begin{tikzpicture}[scale=1.8] 
    \draw[thick,->] (-0.5,0) -- (3.5,0) node[right] {\textit{u}};
    \draw[thick,->] (0,-0.5) -- (0,4) node[above] {\textit{v}};
%

    \draw[thick,blue] (0.9,0.6) -- (0.9,1.6);
    \draw[thick,blue] (1.9,0.6) -- (1.9,1.6);

    \draw[thick,blue] (0.9,0.6) -- (1.9,0.6);
    \draw[thick,blue] (0.9,1.6) -- (1.9,1.6);

    \fill[blue!15,opacity=0.5] (0.9,0.6) -- (0.9,1.6) -- (1.9,1.6) -- (1.9, 0.6) -- cycle;
    
    \draw[thick,blue] (1.45,1.7) -- (1.45,2.3);
    \draw[thick,blue] (2.05,1.7) -- (2.05,2.3);

    \draw[thick,blue] (1.45,1.7) -- (2.05,1.7);
    \draw[thick,blue] (1.45,2.3) -- (2.05,2.3);

    \fill[blue!15,opacity=0.5] (1.45,1.7) -- (1.45,2.3) -- (2.05,2.3) -- (2.05, 1.7) -- cycle;
    
    \draw[densely dotted,blue] (1.4,3.5) -- (1.4,0);    
    \draw[densely dotted,blue] (3,1) -- (0,1);   


    \draw[densely dotted,blue] (1.75,3.5) -- (1.75,0);    
    \draw[densely dotted,blue] (3,2) -- (0,2);  
                        
   \node[below] at (1.4,0) {$\frac{s_1}{q}$};
    \node[left] at (0,1) {$\frac{s_2}{q}$};
    \node[below] at (1.75,0) {$\frac{s'_1}{q'}$};    
    \node[left] at (0, 2) {$\frac{s'_2}{q'}$};    
    \node[right] at (2.1, 1) {$\frac{\xi(q)}{q}$};
    \node[right] at (2.25, 2) {$\frac{\xi(q')}{q'}$};
    \node[above] at (1.6, 3.65) {$< H_{\xi}(q,q') $};    
    \node[right] at (3.2, 1.5) {$< h^{(\mu)}_{\xi}(q,q')$};    
                	                  
	\draw [decorate, decoration = {calligraphic brace,amplitude=3pt}] (1.4,3.6) --  (1.75,3.6);    
	\draw [decorate, decoration = {calligraphic brace,amplitude=5pt}] (3.1, 2) --  (3.1,1);    	

    \draw [<->] (2.1, 0.6) -- (2.1, 1.6);
    \draw [<->] (2.25, 1.7) -- (2.25, 2.3);
        \end{tikzpicture}
        \caption{Disposition of squares in Regime 3}
        \label{R3}
    \end{subfigure}
        \hfill
    \begin{subfigure}{0.48\textwidth}
        \vspace{8mm}
        \centering
        \begin{tikzpicture}[scale=1.8] 
    \draw[thick,->] (-0.5,0) -- (3.5,0) node[right] {\textit{u}};
    \draw[thick,->] (0,-0.5) -- (0,4) node[above] {\textit{v}};
%



    
    \draw[thick,blue] (1.45,1.7) -- (1.45,2.3);
    \draw[thick,blue] (2.05,1.7) -- (2.05,2.3);

    \draw[thick,blue] (1.45,1.7) -- (2.05,1.7);
    \draw[thick,blue] (1.45,2.3) -- (2.05,2.3);

    \fill[blue!15,opacity=0.5] (1.45,1.7) -- (1.45,2.3) -- (2.05,2.3) -- (2.05, 1.7) -- cycle;

    \draw[thick,blue] (0.75,1) -- (0.75,3);
    \draw[thick,blue] (2.75,1) -- (2.75,3);

    \draw[thick,blue] (0.75,1) -- (2.75,1);
    \draw[thick,blue] (0.75,3) -- (2.75,3);

    \fill[blue!15,opacity=0.5] (0.75,1) -- (2.75,1) -- (2.75,3) -- (0.75, 3) -- cycle;
        


    \draw[densely dotted,blue] (1.75,2) -- (1.75,0);    
    \draw[densely dotted,blue] (1.75,2) -- (0,2);  
                        
    \node[below] at (1.75,0) {$\frac{s_1}{q}=\frac{s'_1}{q'}$};    
    \node[left] at (0, 2) {$\frac{s_2}{q}=\frac{s'_2}{q'}$};    
    \node[right] at (3.1, 2) {$\frac{\xi(q)}{q}$};
    \node[right] at (2.2, 2) {$\frac{\xi(q')}{q'}$};
                	                  

    \draw [<->] (3.05, 1) -- (3.05, 3);
    \draw [<->] (2.25, 1.7) -- (2.25, 2.3);
        \end{tikzpicture}
        \caption{Disposition of squares in Regime 4}
        \label{R4}
    \end{subfigure}    
    
    \caption{Disposition of squares in the four regimes under consideration. 
    The quantities $H_{\xi}(q,q')$ and $h^{(\mu)}_{\xi}(q,q')$ are defined  in  \eqref{defhH}.}
    \hfill    
\end{figure}
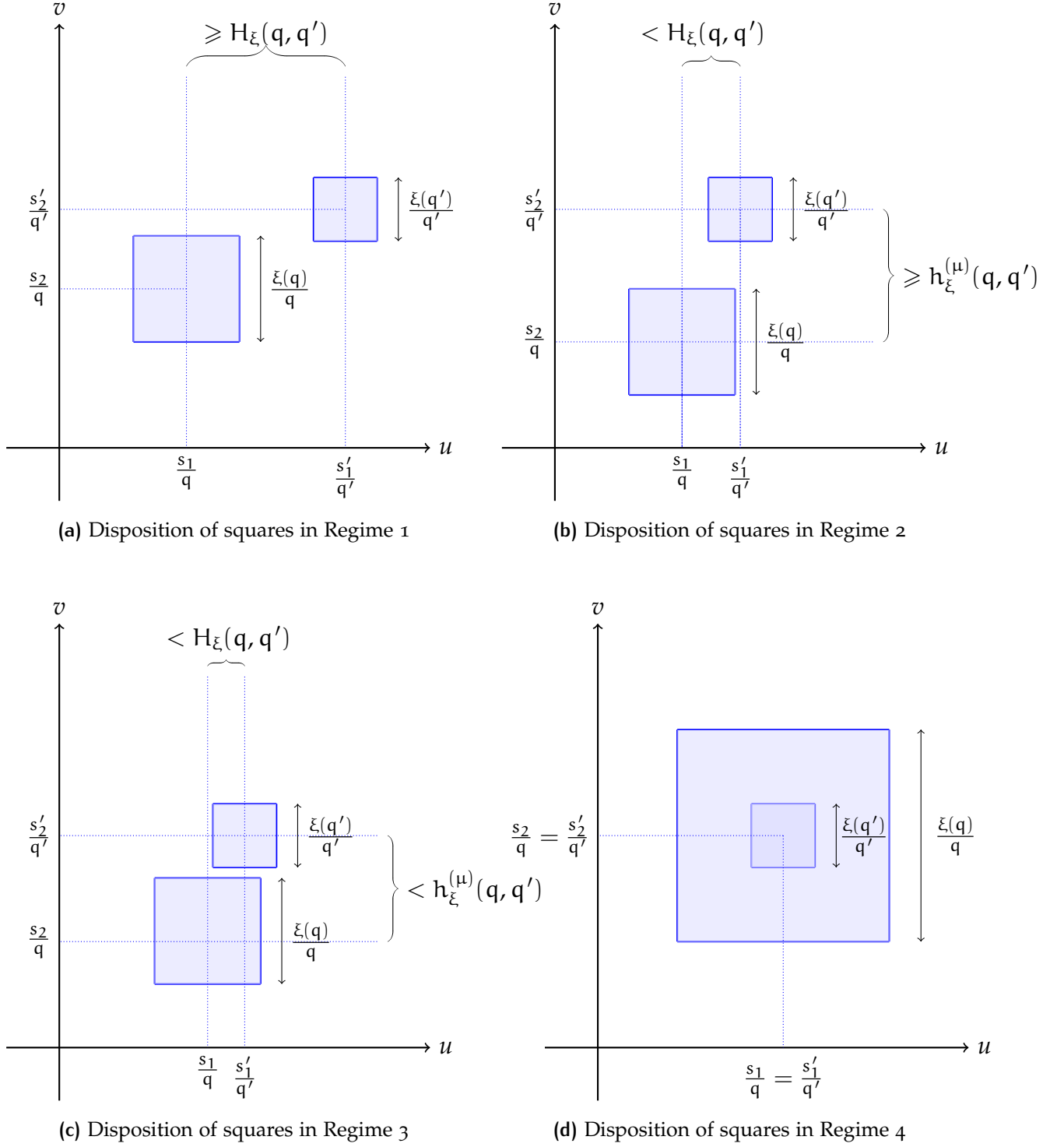

\noindent It should be clear that 
\begin{align}\label{decomthetabig}
\Theta_{\xi}\left(\bm{\mathcal{Q}}; \bm{\mathcal{T}}\right)\;=\; \Theta_{\xi}^{(1)}\left(\bm{\mathcal{Q}};\bm{\mathcal{T}}\right)\;+\; \Theta_{\xi}^{(2)}\left(\bm{\mathcal{Q}}; \bm{\mathcal{T}}\right)\;+\; \Theta_{\xi}^{(3)}\left(\bm{\mathcal{Q}};\bm{\mathcal{T}}\right)\;+\; \Theta_{\xi}^{(4)}\left(\bm{\mathcal{Q}};\bm{\mathcal{T}}\right).
\end{align}
With this identity in mind, the main statement of this section reads as follows~: 

\begin{prop}[Estimation of the subsums in the various regimes]\label{propestimregi} Keep the above definitions and notation. In particular, recall the definitions of the quantity $V\!\left(\bm{\mathcal{T}}\right)$ and of the sum $\Delta_{\xi}\left(\bm{\mathcal{Q}}; \bm{\mathcal{T}}\right)$  in~\eqref{defVT} and~\eqref{sommeglob}, respectively. Assume furthermore as in Proposition~\ref{diagestimterm} --- see the inequality~\eqref{asymptlageT_1T_2} --- that
\begin{equation}\label{asymptlageT_1T_2bis} 
Q_1\;\ge\; \max\left\{\frac{4}{T_1},\; \frac{1}{T_2-T_1},\; \frac{2}{T_2}\right\}.
\end{equation}
Then, 
the following estimates hold~: 
\begin{itemize}
\item (Asymptotic estimate for the subsum in Regime 1)~: 
\begin{align}
&\Theta_{\xi}^{(1)}\left(\bm{\mathcal{Q}}; \bm{\mathcal{T}}\right) \;=\; \nonumber\\
&\left(\Delta_{\xi}\left(\bm{\mathcal{Q}}; \bm{\mathcal{T}}\right)\right)^2 +O\!\left(V\!\left(\bm{\mathcal{T}}\right)^3\cdot\Delta_{\xi}\left(\bm{\mathcal{Q}}; \bm{\mathcal{T}}\right)\cdot\left(\sum_{q=Q_1}^{Q_2} \left(\sqrt{q\xi(q)}\cdot\left(1+\sqrt{q\xi(q)}+\log q\right)\right)\right)\right). \label{estimsub1} 
\end{align}

\item (Upper bound for the subsum in Regime 2)~: given a parameter $\mu>0$ and under the additional assumption that $Q_1\ge 12^{1/(2\mu)}$,
\begin{equation}\label{estimsub2} 
\Theta_{ \xi}^{(2)}\left(\bm{\mathcal{Q}}; \bm{\mathcal{T}}\right) \;\ll\;V\!\left(\bm{\mathcal{T}}\right)^2\cdot\Delta_{\xi}\left(\bm{\mathcal{Q}}; \bm{\mathcal{T}}\right),
\end{equation}
where the implicit constant depends on $\mu$;

\item (Upper bound for the subsum in Regime 3)~: 
\begin{align}\label{estimsub3} 
\Theta_{\xi}^{(3)}\left(\bm{\mathcal{Q}}; \bm{\mathcal{T}}\right)\;& \ll\; \left(1+T_2-T_1\right)\times\nonumber \\
& \sum_{Q_1\le q,q'\le Q_2}\sqrt{\xi(q)\cdot \xi(q')}\times\left(\left(qq'\ \right)^{3/2}\cdot \sqrt{H_{\xi}(q,q')}\cdot  h_\xi^{(\mu)}(q,q')\right.\nonumber \\
&\left.\qquad \quad   + \left(1+\frac{ h_\xi^{(\mu)}(q,q')}{H_{\xi}(q,q')}\right)\cdot \sqrt{qq'\cdot H_{\xi}(q,q')}\cdot  \min\left\{q,q'\right\} \right).
\end{align}

\item (Upper bound for the subsum in Regime 4)~: recall  the notation  
\begin{equation}\label{sumdiv}
\sigma(q)\;=\; \sum_{d|q}d
\end{equation}
introduced in~\eqref{sumdiv0} to denote the sum of the divisors of an integer $q\ge 1$ and the notation $ \Xi(Q)$ defined  in~\eqref{estimsub4asym0} for the sum
\begin{equation}\label{estimsub4asym}
 \Xi(Q)\;=\; \sum_{q=1}^{Q}\sqrt{q^3\cdot \xi(q)}.
\end{equation}
Then, under the  assumption that the map $q\mapsto \xi(q)/q$ is nonincreasing, 
\begin{align}\label{estimsub4} 
\Theta_{ \xi}^{(4)}\left( \bm{\mathcal{Q}}; \bm{\mathcal{T}}\right)  \ll V\!\left(\bm{\mathcal{T}}\right)\cdot \left(\sum_{q=Q_1}^{Q_2}\sqrt{q\cdot \xi(q)}\cdot \sigma(q)\right).
\end{align}
Furthermore, the sum $ \Xi(Q)$ admits the asymptotic expansion
\begin{equation}\label{defMxibis}
\sum_{q=1}^{Q}\sqrt{q\cdot \xi(q)}\cdot \sigma(q)\;=\;   \frac{\pi^2}{6}\cdotp \Xi(Q)\;+\; O\left(\sum_{q=1}^{Q}\sqrt{q\xi(q)}\cdot\log q\right)
\end{equation}
as $Q$ tends to infinity.  
\color{black}
\end{itemize}
\end{prop}

\noindent The statement is proved for each of the four regimes separately from Section~\ref{reg1} to~\ref{reg4}. In particular, the proof of the bound~\eqref{estimsub4} in Regime~4 (see Section~\ref{reg4}) makes it clear that the assumption of the monotonicity under which it holds can be substantially weakened at the cost of some further developments.  Details are left to the interested reader. \\

\noindent Proposition~\ref{offdiagestimterm} is first derived from Proposition~\ref{propestimregi} in the following section.

\subsection{Estimation of the Off--Diagonal Probability Term} 

\begin{proof}[Completion of the  proof of Proposition~\ref{offdiagestimterm} modulo Proposition~\ref{propestimregi}]  For the sake of this proof, generalise the definition of the quantity $ \Xi(Q)$ in~\eqref{estimsub4asym} by setting  
\begin{equation*}
 \Xi(\bm{\mathcal{Q}})\;=\; \sum_{q=Q_1}^{Q_2}\sqrt{q^3\cdot \xi(q)}.
\end{equation*}
Proposition~\ref{diagestimterm} then makes it clear that 
\begin{equation}\label{estimxiQdelta}
 \Xi(\bm{\mathcal{Q}})\;\ll\; V\!\left(\bm{\mathcal{T}}\right)^2\cdot \Delta_{\xi}\left(\bm{\mathcal{Q}}; \bm{\mathcal{T}}\right).
\end{equation}

\noindent The first step in the proof is to reduce the upper bound for $\Theta_{\xi}^{(3)}\left(\bm{\mathcal{Q}};\bm{\mathcal{T}}\right)$ stated in~\eqref{estimsub3} to a simpler one under the assumption of the monotonicity of the map  $q\mapsto \xi(q)/q$. To this end, one reformulates the upper bound under consideration as
\begin{align*}
\Theta_{\xi}^{(3)}\left(\bm{\mathcal{Q}}; \bm{\mathcal{T}}\right)\; \underset{\eqref{defhH}}{\ll}\; &V\!\left(\bm{\mathcal{T}}\right)\cdot \left( \sum_{q=Q_1}^{Q_2} \xi(q)^{3/2-\mu}\cdot q^{1/2+\mu}\cdot \left(\sum_{q'=q}^{Q_2}\sqrt{(q')^3\cdot \xi(q')}\right)\right) \nonumber \\
&+V\!\left(\bm{\mathcal{T}}\right)\cdot\left( \sum_{q=Q_1}^{Q_2} \left(\xi(q)\cdot q+q^{3/2+\mu}\cdot \xi(q)^{1/2-\mu}\right)\cdot   \left(\sum_{q'=q}^{Q_2}\sqrt{q'\cdot \xi(q')}\right) \right).
\end{align*}
Under the assumption of the monotonicity, the second sum meets the inequality 
\begin{align*}
\sum_{q=Q_1}^{Q_2} &\left(\xi(q)\cdot q+q^{3/2+\mu}\cdot \xi(q)^{1/2-\mu}\right)\cdot   \left(\sum_{q'=q}^{Q_2}\sqrt{q'\cdot \xi(q')}\right)\\
&\qquad\qquad \ll\; \sum_{q'=Q_1}^{Q_2}\sqrt{q'\cdot \xi(q')}\cdot \sum_{q=Q_1}^{q'}q^{3/2+\mu}\cdot \xi(q)^{1/2-\mu}\\
&\qquad\qquad \ll\; \sum_{q'=Q_1}^{Q_2} \left(q'\right)^{1/2+\mu}\cdot \xi(q')^{1/2-\mu}\cdot \left(\sum_{q=Q_1}^{q'}q^{3/2}\cdot \xi(q)^{1/2}\right).
\end{align*}
As a consequence, 
\begin{align*}
\Theta_{\xi}^{(3)}\left(\bm{\mathcal{Q}}; \bm{\mathcal{T}}\right)\; &\ll\; V\!\left(\bm{\mathcal{T}}\right)\cdot   \Xi(\bm{\mathcal{Q}})\cdot \left( \sum_{q=Q_1}^{Q_2} q^{1/2+\mu}\cdot \xi(q)^{1/2-\mu}\right)\\
&\underset{\eqref{estimxiQdelta}}{\ll} V\!\left(\bm{\mathcal{T}}\right)^3\cdot   \Delta\left(\bm{\mathcal{Q}}, \bm{\mathcal{T}}\right)\cdot \left( \sum_{q=Q_1}^{Q_2} q^{1/2+\mu}\cdot \xi(q)^{1/2-\mu}\right).
\end{align*}
\noindent  Proposition~\ref{offdiagestimterm} then becomes an immediate consequence of the subsum estimates stated in Proposition~\ref{propestimregi} and of the decomposition~\eqref{decomthetabig}. Indeed, they altogether imply that 
\begin{align*}
&\Theta_{\xi}\left(\bm{\mathcal{Q}}; \bm{\mathcal{T}}\right)\;=\; \left(\Delta_{\xi}\left(\bm{\mathcal{Q}}; \bm{\mathcal{T}}\right)\right)^2\\
&+\!O\!\left(V\!\left(\bm{\mathcal{T}}\right)^3\cdot \Delta_{\xi}\left(\bm{\mathcal{Q}}; \bm{\mathcal{T}}\right)\cdot\left(1\!+\!\sum_{q=Q_1}^{Q_2}\!\left(\sqrt{q\xi(q)}\cdot\left(1\!+\!\log q \!+\!\sqrt{q\xi(q)}\right)+q^{\frac{1}{2}\!+\!\mu}\cdot\xi(q)^{\frac{1}{2}-\mu}\right)\right)\right)\\
&+\!O\!\left(V\!\left(\bm{\mathcal{T}}\right)\cdot \sum_{q=Q_1}^{Q_2}\sqrt{q\cdot \xi(q)}\cdot \sigma(q)\right),
\end{align*} 
whence the statement. 
\end{proof}

\noindent The goal is now to estimate each of the subsums intervening in Proposition~\ref{propestimregi} separately. This involves a blend of ideas from Analysis, Arithmetic and Probability.

\subsection{Regime 1~: Squares Horizontally Distant}\label{reg1} 

\begin{proof}[Proof of the expansion~\eqref{estimsub1}.] Given a rational vector $\bm{s}/q=\left(s_1/q, s_2/q\right)\in \Q_{\xi}(\bm{\mathcal{T}}) $, denote by $\partial_{\xi}\!\left(\bm{\mathcal{T}}; \bm{s}/q\right)$ the boundary of the square $\mathcal{R}_{\xi}\!\left(\bm{\mathcal{T}}; \bm{s}/q\right)$ defined in~\eqref{defract} with the open right side removed. \\ 

\noindent The Strong Markov Property satisfied by Brownian motion (see~\cite[Theorem~2.16]{MP})  implies that if it hits the square $\mathcal{R}_{\xi}\!\left(\bm{\mathcal{T}}; \bm{s}/q\right)$ at a time $t\in\left[ T_1, T_2\right]$ and at a point $$\bm{p}\;=\;(p_1,p_2)\in \partial_{\xi}\!\left(\bm{\mathcal{T}}; \bm{s}/q\right),\qquad \textrm{where} \; t=p_1,$$ then, conditionally on this event, it restarts from this point.  Accordingly, define for the sake of this proof  $\ppi_{\xi}\left( \bm{\mathcal{T}};\left( \bm{s}/q\middle| \bm{p}\right)\right)$ as the probability that  Brownian motion hits at time $t =p_1$ the square $\mathcal{R}_{\xi}\left(\bm{\mathcal{T}}; \bm{s}/q\right)$ at the point $\bm{p}$ for the first time, conditionally on the event that it hits this square. Given points $\bm{s}/q,\; \bm{s}'/q'\in \Q_{\xi}(\bm{\mathcal{T}}) $, one thus has that  
\begin{equation}\label{decocondipasssq}
 \ppi_\xi\left(\bm{\mathcal{T}}; \left(\frac{\bm{s}}{q}, \frac{\bm{s}'}{q'}\right)\right)\;=\; \ppi_{\xi}\left(\bm{\mathcal{T}}; \frac{\bm{s}}{q}\right)\cdot \bigintssss_{\partial_{\xi}\left(\bm{\mathcal{T}}; \frac{\bm{s}}{q}\right)} \ppi_{\xi}\left(\bm{\mathcal{T}}; \left( \frac{\bm{s}'}{q'}\middle|\; \bm{p}\right)\right)\cdot \textrm{d}\bm{p}\qquad \textrm{if}\qquad \frac{s_1}{q}\le \frac{s'_1}{q'}
\end{equation}
and \emph{mutatis mutandis} when $s'_1/q'\le s_1/q$. Here, $\textrm{d}\bm{p}$ denotes an infinitesimal probability density (the explicit form of which is not of interest here) such that 
\begin{equation}\label{intdensinfini}
\bigintssss_{\partial_{\xi}\left(\bm{\mathcal{T}}; \frac{\bm{s}}{q}\right)}   \textrm{d}\bm{p}\;=\; 1.
\end{equation} 
As a consequence, 
\begin{align}
&\Theta_{ \xi}^{(1)}\left( \bm{\mathcal{Q}}; \bm{\mathcal{T}}\right) \underset{\eqref{subsm1}}{=}\sum_{\underset{\bm{s}/q\in \Q_{\xi}(\bm{\mathcal{T}}) }{Q_1\le q,q'\le Q_2}} \ppi_{\xi}\left(\bm{\mathcal{T}}; \frac{\bm{s}}{q}\right)\cdot \bigintssss_{\partial_{\xi}\left(\bm{\mathcal{T}}; \frac{\bm{s}}{q}\right)} \left(\sum_{\underset{\underset{s_1/q\le s'_1/q'}{\eqref{csubsm1}\; \;\&}}{\bm{s}'/q'\in \Q_{\xi}(\bm{\mathcal{T}})~:}} \ppi_{\xi}\left(\bm{\mathcal{T}}; \left( \frac{\bm{s}'}{q'}\middle|\; \bm{p}\right)\right)\right)\cdot \textrm{d}\bm{p}\nonumber\\
&\qquad + \sum_{\underset{\bm{s}'/q'\in \Q_{\xi}(\bm{\mathcal{T}}) }{Q_1\le q,q'\le Q_2}} \ppi_{\xi}\left(\bm{\mathcal{T}}; \frac{\bm{s}'}{q'}\right)\cdot \bigintssss_{\partial_{\xi}\left(\bm{\mathcal{T}}; \frac{\bm{s}'}{q'}\right)} \left(\sum_{\underset{\underset{s_1/q > s'_1/q'}{\eqref{csubsm1}\;\; \&}}{\bm{s}/q\in \Q_{\xi}(\bm{\mathcal{T}})~:}} \ppi_{\xi}\left(\bm{\mathcal{T}}; \left( \frac{\bm{s}}{q}\middle|\; \bm{p}\right)\right)\right)\cdot \textrm{d}\bm{p}\label{completedectheta1}
\end{align}
with the meaning that the inner sums are over the set of integer vectors $\bm{s}'$ and $\bm{s}$, respectively, meeting in particular the assumptions~\eqref{csubsm1} defining this regime. Note that in the above decomposition, the two terms differ  inasmuch as the first one imposes the condition $s_1/q\le s'_1/q'$ in the inner sum and the second one the complementary condition $s_1/q> s'_1/q'$.\\

\noindent The key observation here is that the inner sums turn out to be expressible in terms of the global inhomogeneous sum $ \mathcal{S}_{\xi}^{(\bm{\gamma})}\left(\bm{\mathcal{T}}; q\right)$ introduced in~\eqref{defaxilfixq}. To see this, consider first the inner sum inside the first integral in the equation~\eqref{completedectheta1}. Set then
\begin{equation*}
\bm{\mathcal{T}}_{\le}\;=\;\left(T_1^{(\le)}, T_2^{(\le)}\right)\quad \textrm{with}\quad T_1^{(\le)}\;=\; \frac{s_1}{q}+H_\xi\left(q,q'\right)-p_1 \quad \textrm{and}\quad T_2^{(\le)}\;=\; T_2-p_1
\end{equation*}
and also
\begin{equation*}
\bm{\gamma}_{\le}= \left(\gamma_1^{(\le)}, \gamma_2^{(\le)}\right)\quad \textrm{with}\quad \gamma_2^{(\le)}= \textrm{sgn}\left(\frac{s_2}{q}-p_2\right)\cdot \frac{\left\langle qp_2\right\rangle}{q}
\quad \textrm{and}\quad \gamma_1^{(\le)}\;=\; \left\{\frac{\left\lfloor qp_1\right\rfloor}{q}\right\}.
\end{equation*}
Here, $\textrm{sgn}$ denotes the sign function extended to the origin by setting $\textrm{sgn}(0)=1$, $\left\{\;\cdot\;\right\}$ stands for the fractional part, $\left\lfloor \;\cdot\; \right\rfloor$ for the integer part and $\left\langle\;\cdot\; \right\rangle$ for the distance to the nearest integer. Then, 
\begin{equation*}
\sum_{\underset{\underset{s_1/q\le s'_1/q'}{\eqref{csubsm1}\quad \&}}{\bm{s}'/q'\in \Q_{\xi}(\bm{\mathcal{T}})~:}} \ppi_{\xi}^{(t)}\left(\bm{\mathcal{T}}; \left( \frac{\bm{s}'}{q'}\middle|\; \bm{p}\right)\right)\;=\; \mathcal{S}_{\xi}^{(\bm{\gamma}_{\le})}\left(\bm{\mathcal{T}}_{\le}\;; q'\right).
\end{equation*}

\noindent From Proposition~\ref{asymptinhomg}, taking into account the assumption~\eqref{asymptlageT_1T_2bis} to simplify the bounds, one infers that 
\begin{align}
& \mathcal{S}_{\xi}^{(\bm{\gamma}_{\le})}\left(\bm{\mathcal{T}}_{\le}\;; q'\right)= \frac{2\sqrt{2}}{\sqrt{\pi}}\cdot \left(T_2-\frac{s_1}{q}-H_\xi\left(q,q'\right)\right)\cdot\sqrt{\left(q'\right)^3\xi(q')}+ O\!\left(K_\xi\left(\bm{\mathcal{T}}; q'\right)\right), \label{asymp<sumgat}
\end{align}
where 
\begin{align*}
 K_\xi\left(\bm{\mathcal{T}}; q'\right)\;=\;  \sqrt{q' \xi(q')}&\cdot\left(\left(1+\sqrt{q'\xi(q')}\right)\cdot \left(1+T_2-\frac{s_1}{q}+\sqrt{T_2}\right)\right.\\
&\qquad \qquad  \left. +\left|\log\left(\frac{T_2}{\max\left\{\frac{1}{q'}, H_\xi\!\left(q,q'\right)-\frac{\xi(q')}{2q'}\right\}}\right)\right|\right).
\end{align*}
This relation still holds for the second inner sum in~\eqref{completedectheta1} upon switching the roles of the rationals $\bm{s}'/q'$ and $\bm{s}/q$~: this is saying that the above expansion is also valid for the similarly defined quantity $ \mathcal{S}_{\xi}^{(\bm{\gamma}_{>})}\left(\bm{\mathcal{T}}_{>}\;; q\right)$ upon replacing $q$ with $q'$, and conversely. Furthermore, in both cases, the asymptotic expansions are left unchanged after integration against the infinitesimal probability density $d\bm{p}$ as can be seen from relation~\eqref{intdensinfini}. As a consequence,   it follows from the definitions of the sums $\mathcal{S}_{\xi}^{(\bm{0})}\!\left(\bm{\mathcal{T}}; q\right)$ and $\mathcal{O}_{ \xi}\!\left(\bm{\mathcal{T}}; q\right)$ in~\eqref{defaxilfixq} and~~\eqref{defaxilfixqbis}, respectively,  that
\begin{align}
&\Theta_{ \xi}^{(1)}\left( \bm{\mathcal{Q}}; \bm{\mathcal{T}}\right)\;\underset{\eqref{completedectheta1}\&\eqref{asymp<sumgat}}{=}\; 2\cdot \frac{2\sqrt{2}}{\sqrt{\pi}}\sum_{Q_1\le q,q'\le Q_2} \left(T_2\cdot \mathcal{S}_{\xi}^{(\bm{0})}\!\left(\bm{\mathcal{T}}; q\right)-\mathcal{O}_{ \xi}\!\left(\bm{\mathcal{T}}; q\right)\right)\cdot\sqrt{\left(q'\right)^3\xi(q')}\nonumber\\
&+\!O\!\left(\sum_{Q_1\le q,q'\le Q_2} \mathcal{S}_{\xi}^{(\bm{0})}\!\left(\bm{\mathcal{T}}; q\right)\cdot H_\xi\!\left(q, q'\right)\cdot\sqrt{\left(q'\right)^3\xi(q')}\right)\nonumber\\
&+\!O\!\left(\sum_{Q_1\le q,q'\le Q_2}\!\!\!\sqrt{q' \xi(q')}\cdot \mathcal{S}_{\xi}^{(\bm{0})}\!\left(\bm{\mathcal{T}}; q\right) \cdot\left(\left(1\!+\! \sqrt{q'\xi(q')}\right)\cdot\left(1\!+\! T_2+\sqrt{T_2}\right)\! +\! \left|\log\left(q'T_2\right)\right|\right) \right)\nonumber\\
&+\!O\!\left(\sum_{Q_1\le q,q'\le Q_2}\sqrt{q' \xi(q')}\cdot\left(1+\sqrt{q'\xi(q')}+\log q' \right)\cdot \mathcal{O}_{ \xi}\!\left(\bm{\mathcal{T}}; q\right) \right). \label{defbistheta1tqff}
\end{align}
The asymptotic expansions of the quantities  $\mathcal{S}_{\xi}^{(\bm{0})}\!\left(\bm{\mathcal{T}}; q\right)$ and $\mathcal{O}_{ \xi}\!\left(\bm{\mathcal{T}}; q\right)$ are established in Proposition~\ref{asymptinhomg} and Corollary~\ref{asymptinhomgbis}, respectively. Under the assumption~\eqref{asymptlageT_1T_2bis}, they imply that 
\begin{align}
\mathcal{O}_{ \xi}\!\left(\bm{\mathcal{T}}; q\right)\;&=\; \frac{\sqrt{2}}{\sqrt{\pi}}\cdot\left(T_2^2-T_1^2\right)\cdot\sqrt{q^3\xi(q)}\nonumber \\
&\qquad \qquad\qquad \qquad +O\!\left(\left(1+T_2\right)\cdot V\!\left(\bm{\mathcal{T}}\right)\cdot\sqrt{q\xi(q)}\cdot\left(1+\sqrt{q\xi(q)}+\log q\right)\right)\nonumber\\
&\underset{\eqref{defVT}}{=}\; \frac{T_1+T_2}{2}\cdot \mathcal{S}_{\xi}^{(\bm{0})}\!\left(\bm{\mathcal{T}}; q\right)\;+\; O\!\left(V\!\left(\bm{\mathcal{T}}\right)^2\cdot \sqrt{q\xi(q)}\cdot\left(1+\sqrt{q\xi(q)}+\log q\right)\right)\nonumber\\
&\underset{\eqref{defVT}}{\ll}\; V\!\left(\bm{\mathcal{T}}\right)^2\cdot \mathcal{S}_{\xi}^{(\bm{0})}\!\left(\bm{\mathcal{T}}; q\right).\label{upebOSreg1}
\end{align}
\noindent Also, from its definition in~\eqref{DPT}, the diagonal probability term $\Delta_{\xi}\left(\bm{\mathcal{Q}}; \bm{\mathcal{T}}\right)$ meets the equation $\Delta_{\xi}\left(\bm{\mathcal{Q}}; \bm{\mathcal{T}}\right)= \sum_{q=Q_1}^{Q_2} \mathcal{S}_{\xi}^{(\bm{0})}\left(\bm{\mathcal{T}}; q\right)$, and its asymptotic expansion is determined in  Proposition~\ref{diagestimterm}. This is readily seen to imply that the leading term in the above identity~\eqref{defbistheta1tqff} can be expressed as
\begin{align*}
&2\cdot \frac{2\sqrt{2}}{\sqrt{\pi}}\sum_{Q_1\le q,q'\le Q_2} \left(T_2\cdot \mathcal{S}_{\xi}^{(\bm{0})}\!\left(\bm{\mathcal{T}}; q\right)-\mathcal{O}_{ \xi}\!\left(\bm{\mathcal{T}}; q\right)\right)\cdot\sqrt{\left(q'\right)^3\xi(q')}\\
&\underset{\eqref{defVT}}{=}\; \left(\Delta_{\xi}\left(\bm{\mathcal{Q}}; \bm{\mathcal{T}}\right)\right)^2\;+\; O\!\left(V\!\left(\bm{\mathcal{T}}\right)^3\cdot\Delta_{\xi}\left(\bm{\mathcal{Q}}; \bm{\mathcal{T}}\right)\cdot\left(\sum_{q=Q_1}^{Q_2} \sqrt{q\xi(q)}\cdot\left(1+\sqrt{q\xi(q)}+\log q\right)\right)\right).
\end{align*}
The analysis of the first error term in~\eqref{defbistheta1tqff}   requires some careful consideration. From the very definition of the quantity $H_\xi\!\left(q, q'\right)$ in~\eqref{defhH}, 
\begin{align*}
&\sum_{Q_1\le q,q'\le Q_2} \mathcal{S}_{\xi}^{(\bm{0})}\!\left(\bm{\mathcal{T}}; q\right)\cdot H_\xi\!\left(q, q'\right)\cdot\sqrt{\left(q'\right)^3\xi(q')}\\
&\qquad\qquad\le\; \sum_{Q_1\le q,q'\le Q_2} \mathcal{S}_{\xi}^{(\bm{0})}\!\left(\bm{\mathcal{T}}; q\right)\cdot \left(\frac{\xi(q)}{q}+ \frac{\xi(q')}{q'}\right)\cdot\sqrt{\left(q'\right)^3\xi(q')}\\
&\qquad\qquad\le\; \Delta_{\xi}\left(\bm{\mathcal{Q}}; \bm{\mathcal{T}}\right)\cdot\left(\sum_{Q_1\le q'\le Q_2} \sqrt{\left(q'\right)^3\xi(q')}\cdot \left(\frac{\xi(q')}{q'}\right)\right)\\
&\qquad\qquad \qquad\qquad  +\left(\sum_{Q_1\le q'\le Q_2}\sqrt{\left(q'\right)^3\xi(q')}\right)\cdot\left(\sum_{Q_1\le q\le Q_2}\mathcal{S}_{\xi}^{(\bm{0})}\!\left(\bm{\mathcal{T}}; q\right)\cdot \frac{\xi(q)}{q}\right).
\end{align*}
Substituting $\mathcal{S}_{\xi}^{(\bm{0})}\!\left(\bm{\mathcal{T}}; q\right)$ with its asymptotic expansion given in Proposition~\ref{asymptinhomg}  and ta\-king into account that of the quantity $\Delta_{\xi}\left(\bm{\mathcal{Q}}; \bm{\mathcal{T}}\right)$ stated in Proposition~\ref{diagestimterm}, the second term in the above bound becomes 
\begin{align*}
&\left(\sum_{Q_1\le q'\le Q_2}\sqrt{\left(q'\right)^3\xi(q')}\right)\cdot\left(\sum_{Q_1\le q\le Q_2}\mathcal{S}_{\xi}^{(\bm{0})}\!\left(\bm{\mathcal{T}}; q\right)\cdot\frac{\xi(q)}{q}\right)\\
&\qquad \qquad \qquad\qquad \qquad  \underset{\eqref{defVT}}{\ll}\;  V\!\left(\bm{\mathcal{T}}\right)^3\cdot \Delta_{\xi}\left(\bm{\mathcal{Q}}; \bm{\mathcal{T}}\right)\cdot\left(\sum_{Q_1\le q'\le Q_2} \sqrt{\left(q'\right)^3\xi(q')}\cdot \frac{\xi(q')}{q'}\right).
\end{align*}
The first error term can thus be bounded as 
\begin{align*}
\sum_{Q_1\le q,q'\le Q_2} \mathcal{S}_{\xi}^{(\bm{0})}\!\left(\bm{\mathcal{T}}; q\right)\cdot H_\xi\!\left(q, q'\right)\cdot&\sqrt{\left(q'\right)^3\xi(q')}\\ 
& \ll\;  V\!\left(\bm{\mathcal{T}}\right)^3\cdot \Delta_{\xi}\left(\bm{\mathcal{Q}}; \bm{\mathcal{T}}\right)\cdot\left(\sum_{Q_1\le q\le Q_2} q^{1/2}\xi(q)^{3/2}\right).
\end{align*}
\noindent With the help of the inequality~\eqref{upebOSreg1}, the analysis of the last two error terms in the expansion~\eqref{defbistheta1tqff} does not cause any difficulty. This yields the sought relation~\eqref{estimsub1}. 
\end{proof}

\subsection{Regime 2~: Squares Vertically Distant} \label{reg2}

\begin{proof}[Proof of the inequality~\eqref{estimsub2}.] Let there be rationals $\bm{s}/q, \bm{s'}/q'\in\Q_\xi\left(\bm{\mathcal{T}}\right)$ with denominators $Q_1\le q, q'\le Q_2$. Since the sought relation~\eqref{estimsub2} involves only an upper bound, it suffices, by the symmetry of the problem, to consider the case where Brownian motion hits first the square $\mathcal{R}_{\xi}\left(\bm{\mathcal{T}}; \bm{s}/q\right)$  and then the square $\mathcal{R}_{\xi}\left(\bm{\mathcal{T}}; \bm{s'}/q'\right)$, and to furthermore assume that $s_1/q\le s'_1/q'$ and that $s_2/q\le s'_2/q'$ (all other cases are treated with obvious modifications). \\

\noindent The proof relies on an extension of the idea   intervening in the proof of the  just--established equation~\eqref{estimsub1}, and in particular of the decomposition~\eqref{decocondipasssq}. Denote by $\partial_{\xi}\!\left(\bm{\mathcal{T}}; \bm{s}/q\right)$ and by $\partial_{\xi}\!\left(\bm{\mathcal{T}}; \bm{s}'/q'\right)$  the boundaries of the squares $\mathcal{R}_{\xi}\!\left(\bm{\mathcal{T}}; \bm{s}/q\right)$ and $\mathcal{R}_{\xi}\!\left(\bm{\mathcal{T}}; \bm{s}'/q'\right)$, respectively, with the open right sides removed. Let $$\bm{p}\;=\;(p_1,p_2)\in \partial_{\xi}\!\left(\bm{\mathcal{T}}; \bm{s}/q\right)\qquad \textrm{and} \qquad \bm{p}'\;=\;(p'_1,p'_2)\in \partial_{\xi}\!\left(\bm{\mathcal{T}}; \bm{s}'/q'\right)$$ be the points where  Brownian motion hits first the squares $\mathcal{R}_{\xi}\left(\bm{\mathcal{T}}; \bm{s}/q\right)$ and $\mathcal{R}_{\xi}\left(\bm{\mathcal{T}}; \bm{s}'/q'\right)$, respectively (in particular, $p'_1>p_1$). Define $\ppi_{\xi}\left( \bm{\mathcal{T}};\left( \bm{s}'/q' :   \bm{p}'\middle| \bm{p}\right)\right)$ as the probability that Brownian motion hits at time $t =p'_1$ the square $\mathcal{R}_{\xi}\left(\bm{\mathcal{T}}; \bm{s}'/q'\right)$ at the point $\bm{p}'$ for the first time, conditionally on the event that it hits first  the square $\mathcal{R}_{\xi}\left(\bm{\mathcal{T}}; \bm{s}/q\right)$ at the point $\bm{p}$  at time $t =p_1<p'_1$ for the first time. Then, from the Strong Markovian Property of Brownian motion, 
 \begin{equation}\label{decocondipasssqbis}
 \ppi_\xi\left(\bm{\mathcal{T}}; \left(\frac{\bm{s}}{q}, \frac{\bm{s}'}{q'}\right)\right)= \ppi_{\xi}\left(\bm{\mathcal{T}}; \frac{\bm{s}}{q}\right)\cdot \bigintssss_{\partial_{\xi}\left(\bm{\mathcal{T}}; \frac{\bm{s}}{q}\right)}\!\! \left( \bigintssss_{\partial_{\xi}\left(\bm{\mathcal{T}}; \frac{\bm{s}'}{q'}\right)} \ppi_{\xi}\left(\bm{\mathcal{T}}; \left( \frac{\bm{s}'}{q'} :   \bm{p}'\middle|\; \bm{p}\right)\right)\cdot \textrm{d}^{(\bm{p})}\bm{p}'\right)\! \!\textrm{d}\bm{p}. 
 \end{equation}
Here, $\textrm{d}\bm{p}$ and $\textrm{d}^{(\bm{p})}\bm{p}'$ are infinitesimal probability densities (the explicit forms of which are not of interest here) such that 
\begin{equation*}\label{intdensinfinibis}
\bigintssss_{\partial_{\xi}\left(\bm{\mathcal{T}}; \frac{\bm{s}}{q}\right)} \bigintssss_{\partial_{\xi}\left(\bm{\mathcal{T}}; \frac{\bm{s}'}{q'}\right)}  \textrm{d}^{(\bm{p})}\bm{p}' \cdot  \textrm{d}\bm{p}\;=\; 1.
\end{equation*} 
Set
 \begin{alignat}{2}\label{boundreg2a}
&x^+\;=\; \frac{s'_1}{q'}+\frac{\xi(q')}{2q'}-p_1\underset{\eqref{defhH}\&\eqref{csubsm2}}{\le} 3H_{\xi}(q,q'), \qquad &&x^-\;=\;p'_1-p_1\;>\; 0,\\
&y^-\;=\;p'_2-p_2, 
\qquad  &&y^+\;=\; \frac{s'_2}{q'}\pm \frac{\xi(q')}{2q'}-p_2
\end{alignat} 
and also  
\begin{alignat}{2}\label{boundreg2}
&\alpha\;=\; \min\left\{\frac{\xi(q')}{q'}, \;\frac{s'_1}{q'}+\frac{\xi(q')}{2q'}-p'_1\right\}\qquad  \textrm{and}\qquad &&\beta\;=\;  \frac{\xi(q')}{q'}\cdotp
\end{alignat}
Here, from the relations~\eqref{defhH} and~\eqref{csubsm2}, one easily checks that 
\begin{equation}\label{lasty^-} 
\left|y^-\right|\;\ge\; \frac{1}{2}\cdot h^{(\mu)}_{\xi}(q,q')\;+\; \frac{1}{4}\cdot \left|\frac{s'_2}{q'}-p_2\right|. 
\end{equation}
Also,  the upper bound~\eqref{mainupbprob} stated in Corollary~\ref{coroutile} implies that 
\begin{align}
\ppi_{\xi}\left(\bm{\mathcal{T}}; \left( \frac{\bm{s}'}{q'} :   \bm{p}'\middle|\; \bm{p}\right)\right) \;\ll \; \min&\left\{ \sqrt{\frac{\alpha}{x^-}}, \frac{\sqrt{x^+}}{y^-}\right\}\cdot\exp\left(-\frac{\left(y^-\right)^2}{2x^+}\right)\nonumber\\
& +\;\max\left\{\frac{\beta}{\sqrt{x^\pm}}\cdot  \exp\left(-\frac{\left(y^-\right)^2}{2x^\pm}\right)\right\}.\label{ezggze}
\end{align}
To simplify the right--hand side, make the following observations~: 
\begin{itemize}
\item given a real $y>0$, the map $x>0\mapsto \exp(-y^2/(2x))/\sqrt{x}$ is increasing over $\left(0, y^2\right]$ and decreasing over $\left[ y^2, \infty\right)$. The lower bound $Q_1\ge 12^{1/(2\mu)}$ part of the assumption ensures that  $3H_{\xi}(q,q')\le \left(h^{(\mu)}_{\xi}(q,q')\right)^2/4$, and therefore that 
\begin{equation}\label{defe}
x^-\;\le\; x^+\;\le\; 3\cdot H_{\xi}(q,q')\;\le\; \left(\frac{h^{(\mu)}_{\xi}(q,q')}{2}\right)^2\;\le\;  \left(y^-\right)^2.
\end{equation} 
Consequently, the maximum on the right--hand side of the upper bound~\eqref{ezggze} is attained at the positive sign, i.e.~at $x^+$;
 
\item if $x^-<\xi(q')/q'=\beta$, then $x^+\le x^-+\xi(q')/q'\le 2\xi(q')/q'= 2\beta$ and   $$\frac{\sqrt{x^+}}{\left|y^-\right|}\;\le\; \frac{4\sqrt{\beta}}{h^{(\mu)}_{\xi}(q,q')+\left| s'_2/q'-p_2\right|};$$
\item if $x^-\ge \xi(q')/q'=\beta$, then $x^+\le x^-+\xi(q')/q'\le 2x^-$ and $$\sqrt{\frac{\alpha}{x^-}}\;\ll\; \frac{\sqrt{\alpha}}{\sqrt{x^+}}\;\ll\; \frac{\sqrt{\xi(q')/q'}}{\sqrt{x^+}}\;=\; \sqrt{\frac{\beta}{x^+}}.$$
\end{itemize}
This analysis shows that 
\begin{align}
\ppi_{\xi}\left(\bm{\mathcal{T}}; \left( \frac{\bm{s}'}{q'} :   \bm{p}'\middle|\; \bm{p}\right)\right) \;&\ll \;  \sqrt{\beta}\cdot \left(\frac{1}{h^{(\mu)}_{\xi}(q,q')}+\frac{1}{\sqrt{x^+}}\right)\cdot  \exp\left(-\frac{\left(y^-\right)^2}{2x^+}\right)\label{majreg2codn0}\\
&\underset{\eqref{defhH}}{\ll}\; \sqrt{\frac{\beta}{H_{\xi}(q,q')}}\cdot \exp\left(-\frac{1}{48\cdot H_{\xi}(q,q')^{2\mu}}-\frac{\left( s'_2/q'-p_2\right)^2}{48\cdot H_{\xi}(q,q')}\right), \label{majreg2codn}
\end{align}
where this last bound follows again by monotonicity  from the inequalities~\eqref{defe}. \\

\noindent Fix a value of the rational $\bm{s}/q\in\Q_\xi\left(\bm{\mathcal{T}}\right)$ and of the denominator $Q_1\le q'\le Q_2$. From the decomposition~\eqref{decocondipasssqbis} and then by comparison with an integral (with the help, e.g., of the estimates for the complementary error function given in~\eqref{ineqerrfct1}), one thus obtains  that 
\begin{align}
&\sum_{\underset{\eqref{csubsm2}}{s_2'\in\Z~: }} \ppi_\xi\left(\bm{\mathcal{T}}; \left(\frac{\bm{s}}{q}, \frac{\bm{s}'}{q'}\right)\right)\nonumber\\
&\underset{\eqref{majreg2codn}}{\ll}  \ppi_{\xi}\left(\bm{\mathcal{T}}; \frac{\bm{s}}{q}\right)\cdot \sqrt{\frac{\xi(q')/q'}{H_{\xi}(q,q')}}\cdot \exp\left(-\frac{1/48}{ H_{\xi}(q,q')^{2\mu}}\right)\cdot \left(\sum_{s'_2\in\Z}  \exp\left(-\frac{\left( s'_2/q'-p_2\right)^2}{48\cdot H_{\xi}(q,q')}\right)\right) \nonumber \\
&\ll  \ppi_{\xi}\left(\bm{\mathcal{T}}; \frac{\bm{s}}{q}\right)\cdot \sqrt{\frac{\xi(q')/q'}{H_{\xi}(q,q')}}\cdot \exp\left(-\frac{1}{48\cdot H_{\xi}(q,q')^{2\mu}}\right)\cdot  \left(1+q'\sqrt{H_\xi(q,q')}\right).\label{ineqproofreg2}
\end{align}
For fixed values of $s_1,q$ and $q'$, the number of integers $s'_1$ meeting the first inequality in the condition~\eqref{csubsm2} defining this regime   is $O\left(1+q' H_\xi(q,q')\right)$. Therefore, 
\begin{align}
&\sum_{Q_1\le q'\le Q_2} \sum_{\underset{\eqref{csubsm2}}{\bm{s}'\in\Z^2~: }} \ppi_\xi\left(\bm{\mathcal{T}}; \left(\frac{\bm{s}}{q}, \frac{\bm{s}'}{q'}\right)\right)\nonumber\\
&\underset{\eqref{ineqproofreg2}}{\ll}\; \ppi_\xi\left(\bm{\mathcal{T}}; \frac{\bm{s}}{q}\right)\cdot \sum_{Q_1\le q'\le Q_2} \left(\sqrt{\frac{\xi(q')/q'}{H_{\xi}(q,q')}}\cdot   \left(1+q'\sqrt{H_\xi(q,q')}\right)\times\right.\nonumber\\
&\left.\qquad\qquad\qquad\qquad\qquad \qquad\left(1+q' H_\xi(q,q')\right)\cdot \exp\left(-\frac{1}{48\cdot H_{\xi}(q,q')^{2\mu}}\right)\right).\label{finreg20}
\end{align}
The sum over $q'$ is then split into two subsums. In the first one, one assumes that 
\begin{equation}\label{regsubmax1}
\max\left\{\frac{\xi(q)}{q}, \frac{\xi(q')}{q'}\right\}\;=\; \frac{\xi(q')}{q'}
\end{equation}
so as to obtain from the definition of $H_\xi(q,q')$ in~\eqref{defhH} that 
\begin{align}
 \sum_{\underset{\eqref{regsubmax1}}{Q_1\le q'\le Q_2~: }} &\sqrt{\frac{\xi(q')/q'}{H_{\xi}\!(q,q')}}\cdot   \left(1+q'\sqrt{H_\xi\!(q,q')}\right)\cdot \left(1+q' H_\xi\!(q,q')\right)\cdot \exp\left(-\frac{1/48}{H_{\xi}\!(q,q')^{2\mu}}\right)\nonumber\\
  &\ll \sum_{\underset{\eqref{regsubmax1}}{Q_1\le q'\le Q_2~: }} \sqrt{q'}\cdot \exp\left(-\frac{\left(q'/2\right)^{2\mu}}{48}\right)\;\ll\; 1.\label{finreg2-1}
\end{align}
In the second subsum, one assumes that 
\begin{equation}\label{regsubmax2}
\max\left\{\frac{\xi(q)}{q}, \frac{\xi(q')}{q'}\right\}\;=\; \frac{\xi(q)}{q}\cdotp
\end{equation}
Then,  upon relying on the bound $\exp(-1/x)\ll x^A$ valid for all reals $x>0$ and all exponents $A>0$, one readily sees that 
\begin{align}
 \sum_{\underset{\eqref{regsubmax2}}{Q_1\le q'\le Q_2~: }} &\sqrt{\frac{\xi(q')/q'}{H_{\xi}\!(q,q')}}\cdot   \left(1+q'\sqrt{H_\xi\!(q,q')}\right)\cdot \left(1+q' H_\xi\!(q,q')\right)\cdot \exp\left(-\frac{1/48}{H_{\xi}\!(q,q')^{2\mu}}\right)\nonumber \\
  &\ll  \left(\frac{\xi(q)}{q}\right)^{2A\mu-1/2}\cdot\left( \sum_{ Q_1\le q'\le Q_2 } \sqrt{\xi(q')\cdot \left(q'\right)^3}\right), \label{finreg2-2}
  \end{align} 
 where the implicit constant depends on the choice of $A$ and $\mu$. Here,  the second factor compares with the quantity $\Delta_\xi\left(\bm{\mathcal{Q}}, \bm{\mathcal{T}}\right)$ since Proposition~\ref{diagestimterm} implies that 
\begin{equation}\label{upboundreg2fin}
 \sum_{ Q_1\le q'\le Q_2 } \sqrt{\xi(q')\cdot \left(q'\right)^3}\;\ll\;  V\!\left(\bm{\mathcal{T}}\right)^2\cdot\Delta_{\xi}\left(\bm{\mathcal{Q}}; \bm{\mathcal{T}}\right).
\end{equation}

\noindent The sought inequality~\eqref{estimsub2} then follows upon putting the relations~\eqref{finreg20}, \eqref{finreg2-1}, \eqref{finreg2-2} and~\eqref{upboundreg2fin} together and upon choosing $A>3/(4\mu)$. Indeed, one has then that  
\begin{align}
\Theta_{ \xi}^{(2)}\left(\bm{\mathcal{Q}}; \bm{\mathcal{T}}\right)\;&=\; \sum_{Q_1\le q\le Q_2}\;\sum_{\underset{\bm{s}/q\in\Q_\xi\!\left(\bm{\mathcal{T}}\right)}{\bm{s}\in\Z^2~: }}\;\sum_{Q_1\le q'\le Q_2} \;\sum_{\underset{\eqref{csubsm2}}{\bm{s}'\in\Z^2~: }} \ppi_\xi\left(\bm{\mathcal{T}}; \left(\frac{\bm{s}}{q}, \frac{\bm{s}'}{q'}\right)\right)\nonumber\\
&\ll\; \left(\sum_{Q_1\le q\le Q_2} \;\sum_{\underset{\bm{s}/q\in\Q_\xi\!\left(\bm{\mathcal{T}}\right)}{\bm{s}\in\Z^2~: }}\;  \ppi_{\xi}\left(\bm{\mathcal{T}}; \frac{\bm{s}}{q}\right)\right)\nonumber\\
&\qquad +\; V\!\left(\bm{\mathcal{T}}\right)^2\cdot\Delta_{\xi}\left(\bm{\mathcal{Q}}; \bm{\mathcal{T}}\right)\cdot\left(\sum_{Q_1\le q\le Q_2} \;\left(\frac{\xi(q)}{q}\right)^{2A\mu-1/2}\cdot\sum_{\underset{\bm{s}/q\in\Q_\xi\!\left(\bm{\mathcal{T}}\right)}{\bm{s}\in\Z^2~: }}\;  \ppi_{\xi}\left(\bm{\mathcal{T}}; \frac{\bm{s}}{q}\right)\right)\nonumber\\
&\ll\; V\!\left(\bm{\mathcal{T}}\right)^2\cdot\Delta_{\xi}\left(\bm{\mathcal{Q}}; \bm{\mathcal{T}}\right).\label{upbreg2}
\end{align}
This last simplification results on the one hand from the definition of $\Delta_{\xi}\left(\bm{\mathcal{Q}}; \bm{\mathcal{T}}\right)$ in~\eqref{sommeglob} and, on the other, from the choice of an exponent $A>0$ large enough so that the second sum over $q$ converges. (Proposition~\ref{asymptinhomg} --- which gives an asymptotic as $q$ grows to infinity for the inner sums over $\bm{s}\in\Z^2$ --- shows that this is indeed possible.) This completes the proof of the inequality~\eqref{estimsub2} in Regime~2. 
\end{proof}

\subsection{Regime 3~: Squares Close to one-another with Distinct Centers}  \label{reg3}

\noindent Given reals $\sigma$ and $\tau$ such that $\sigma\ge 0$, let 
\begin{equation}\label{qsigtauT}
\Q_{\bm{\mathcal{T}}}\!\left(\sigma, \tau)\right)= \Q^2\;\bigcap\;  \left(\left[\widehat{T}_1+\sigma;\; \widehat{T}_1+\sigma+1\right)\times  \left[\tau;\; \tau+1\right)\right), \; \textrm{where}\; \widehat{T}_1= \max\left\{\frac{T_1}{2}, \;T_1-\frac{1}{2}\right\},
\end{equation}
and define, in the notation introduced in~\eqref{defproR}, the probability 
\begin{equation}\label{pistarprob}
\pi^{*}_{\bm{\mathcal{T}}}\!\left(\left(\sigma, \tau\right); \delta\right)\;=\;  \sup_{\bm{x}\in \Q_{\bm{\mathcal{T}}}\!\left(\sigma, \tau\right)} \;\ppi\!\left(\mathcal{R}\left(\bm{x}, \delta\right)\right).
\end{equation}
Here, $\delta>0$ and  given a vector $\bm{x}=\left(x_1, x_2\right)$, 
\begin{equation*}
\mathcal{R}\left(\bm{x}, \delta\right)\;=\; \left[x_1-\frac{\delta}{2};\; x_1+\frac{\delta}{2}\right]\times \left[x_2-\frac{\delta}{2};\; x_2+\frac{\delta}{2}\right].
\end{equation*}
In other words, $\pi^{*}_{\bm{\mathcal{T}}}\!\left(\left(\sigma, \tau\right); \delta\right)$ measures the supremum of the probabilities that a Brownian motion should hit a square with side length $\delta$ centered at a point in $\Q_{\bm{\mathcal{T}}}\!\left(\sigma, \tau\right)$.\\

\noindent The main result in this section reads as follows~: 

\begin{prop}\label{progreg3}
Assume that $\sigma\ge 0$ and $ \tau$ are real numbers, and let $\Theta_{\xi}^{(3)}\left(\bm{\mathcal{Q}};\bm{\mathcal{T}};  \left(\sigma, \tau\right)\right)$ be the subsum $\Theta_{\xi}^{(3)}\left(\bm{\mathcal{Q}}; \bm{\mathcal{T}}\right)$  corresponding to this regime restricted to the set  $\Q_{\bm{\mathcal{T}}}\!\left(\sigma, \tau)\right)$; that is, 
\begin{align*}
\Theta_{\xi}^{(3)}\left(\bm{\mathcal{Q}};\bm{\mathcal{T}};  \left(\sigma, \tau\right)\right)\;=\; \sum_{Q_1\le q,q'\le Q_2}\; \sum_{\underset{\eqref{csubsm3}}{\frac{\bm{s}}{q}, \frac{\bm{s}'}{q'}\in \Q_{\bm{\mathcal{T}}}\!\left(\sigma, \tau\right)~: }} \ppi_\xi\left(\bm{\mathcal{T}}; \left(\frac{\bm{s}}{q}, \frac{\bm{s}'}{q'}\right)\right).
\end{align*}
Given denominators $Q_1\le q,q'\le Q_2$, set 
\begin{equation}\label{defdeltaiqq}
\delta_\xi\!\left(q,q'\right)\;=\; \min\left\{\frac{\xi(q)}{q}, \frac{\xi(q')}{q'}\right\}.
\end{equation}
Then, 
\begin{align}
\Theta_{\xi}^{(3)}\left(\bm{\mathcal{Q}};\bm{\mathcal{T}};  \left(\sigma, \tau\right)\right)\;\ll\; &\sum_{Q_1\le q,q'\le Q_2}\; \pi^{*}_{\bm{\mathcal{T}}}\!\left(\left(\sigma, \tau\right); \delta_\xi\!\left(q,q'\right) \right)\times  \left[   \left(qq'\right)^2\cdot H_{\xi}(q,q') \cdot  h_\xi^{(\mu)}(q,q')\right.\nonumber \\
&\qquad\qquad \left.+ \;  qq'\cdot \left(H_{\xi}(q,q')+h_\xi^{(\mu)}(q,q')\right)\cdot \gcd\!\left(q,q'\right)\right],\label{thetpropo3sigtau}
\end{align}
where the implicit constant is absolute. 
\end{prop}

\begin{proof}[Deduction of the bound~\eqref{estimsub3} in Regime 3 from Proposition~\ref{progreg3}] It should be clear that if $\bm{s}/q$ and $\bm{s}'/q'$ are two rational points in the strip $\Q_\xi\left(\bm{\mathcal{T}}\right)$ (defined in~\eqref{defstrip}) such that the corresponding squares $\mathcal{R}_\xi\left(\bm{\mathcal{T}}; \bm{s}/q\right)$ and $\mathcal{R}_\xi\left(\bm{\mathcal{T}}; \bm{s}/q\right)$ (defined in~\eqref{defract}) meet the conditions~\eqref{csubsm3} defining Regime~3, then there exist integers 
\begin{equation}\label{bornemn}
0\;\le\, m\; \underset{\eqref{qsigtauT}}{\ll}\;\max\left\{1, T_2-T_1\right\}\qquad\qquad\textrm{and}\qquad\qquad n\in\Z
\end{equation}
such that 
\begin{align*}
\mathcal{R}_\xi\left(\bm{\mathcal{T}}; \frac{\bm{s}}{q}\right),  \; &\mathcal{R}_\xi\left(\bm{\mathcal{T}}; \frac{\bm{s}'}{q'}\right)\\
& \underset{\eqref{qsigtauT}}{\subset}\; \bigcup_{\kappa_1, \kappa_2\in \left\{0,1\right\}} \left[\widehat{T}_1+m+\frac{\kappa_1}{2};\; \widehat{T}_1+m+\frac{\kappa_1}{2}+1\right)\times  \left[n+\frac{\kappa_2}{2};\; n+\frac{\kappa_2}{2}+1\right).
\end{align*}
Therefore, the subsum $\Theta_{\xi}^{(3)}\left(\bm{\mathcal{Q}}; \bm{\mathcal{T}}\right)$ corresponding to this regime is bounded as 
\begin{align*}
\Theta_{\xi}^{(3)}\left(\bm{\mathcal{Q}}; \bm{\mathcal{T}}\right)\;\le\; \sum_{\kappa_1, \kappa_2\in \left\{0,1\right\}} \; \sum_{\underset{\eqref{bornemn}}{(m,n)\in\Z^2~:}}\Theta_{\xi}^{(3)}\left(\bm{\mathcal{T}}; \bm{\mathcal{Q}}; \left(m+\frac{\kappa_1}{2}, \; n+\frac{\kappa_2}{2}\right)\right).
\end{align*}
Set $$x^-\;\underset{\eqref{qsigtauT}}{=}\;\widehat{T}_1+m+\frac{\kappa_1}{2}, \qquad x^+\;=\; x^-+1, \qquad y^-\;=\; n+\frac{\kappa_2}{2}, \qquad y^+\;=\;y^-+1$$ and$$\alpha\;=\; \beta\; =\;  \delta_\xi\!\left(q,q'\right). $$ Relying on the  notation introduced in Proposition~\ref{progreg3}, it follows from the upper bound~\eqref{mainupbprob} for the hitting probability stated in Corollary~\ref{coroutile} that  for any choice of $\kappa_1, \kappa_2\in\left\{0,1\right\}$, 
\begin{align*}
&\sum_{\underset{\eqref{bornemn}}{(m,n)\in\Z^2~:}} \pi^{*}_{\bm{\mathcal{T}}}\!\left(\left(m+\frac{\kappa_1}{2}, \; n+\frac{\kappa_2}{2}\right); \delta_\xi\!\left(q,q'\right) \right)\\
&\ll\;\; \sum_{0\le m\ll \max\left\{1, T_2-T_1\right\}}\sqrt{\frac{\delta_\xi\!\left(q,q'\right)}{\widehat{T}_1+m+\kappa_1/2}}\cdotp\sum_{n\in\Z}\exp\left(-\frac{\left(n+\kappa_2/2\right)^2}{2\cdot\left(\widehat{T}_1+1+m+\kappa_1/2\right)}\right)\\
&\ll\; \sqrt{\delta_\xi\!\left(q,q'\right) }\cdot\max\left\{1, \; T_2-T_1\right\}.
\end{align*}
This last relation follows from the fact that   the inner sum (over $n\in\Z$) is  bounded by $O\left(1+\sqrt{\widehat{T}_1+m+\kappa_1/2}\right)$ as can be seen from a comparison with an integral. Upon relying on the identity $\max\left\{x,y\right\}\cdot \min\left\{x,y\right\}= x\cdot y$ valid for all reals $x,y\ge 0$, one thus obtains from the inequality~\eqref{thetpropo3sigtau} that 
\begin{align*}
\Theta_{\xi}^{(3)}\left(\bm{\mathcal{Q}}; \bm{\mathcal{T}}\right)\;& \ll\; \max\left\{1, T_2-T_1\right\}\times\\
& \sum_{Q_1\le q,q'\le Q_2}\sqrt{\xi(q)\cdot \xi(q')}\times\left(\left(qq' \right)^{3/2}\cdot \sqrt{H_{\xi}(q,q')}\cdot  h_\xi^{(\mu)}(q,q')\right.\\
&\left.\qquad \qquad \qquad  \qquad + \left(1+\frac{h_\xi^{(\mu)}(q,q')}{H_{\xi}(q,q')}\right)\cdot \sqrt{qq'\cdot H_{\xi}(q,q')}\cdot  \gcd(q,q') \right).
\end{align*}
This yields the form of the stated inequality~\eqref{estimsub3} upon employing the trivial  upper  bound $\gcd(q,q')\le \min\left\{q, q'\right\}$.
\end{proof}

\noindent It remains to establish Proposition~\ref{progreg3}. Its proof relies on an auxiliary lemma of an arithmetic nature.

\begin{lem}[Arithmetic Lemma]\label{arithlem}  \hfill
\begin{itemize}
\item[$\bm{(1)}$] Fix  integers $g,h\ge 1$. Let also $A\ge 1$ and $\sigma', \tau'\in\R$. Then, the number of pairs $\left(c,d\right)\in\Z^2$ solutions to the system
\begin{align*}
\begin{cases}
1\;\le\; \left|gd-ch\right|\;\le\; A ;\\
\sigma' h\;\le\; d\;\le\; h(\sigma'+1), \; d\neq 0 ;\\
\tau' g\;\le\; c\;\le\; g(\tau'+1), \; c\neq 0
\end{cases}
\end{align*}
is at most $6A$.

\item[$\bm{(2)}$] Fix  integers $g,h\ge 1$. Let also $A, B\ge 1$ and $\sigma_1, \sigma_2, \tau_1, \tau_2 \in\R$. Then, the number of quadruples $\left(c_1, c_2,d_1, d_2\right)\in\Z^4$ solutions to the system
\begin{align*}
\begin{cases}
1\;\le\; \left|gd_1-c_1h\right|\;\le\; A \quad \&\quad  1\;\le\; \left|gd_2-c_2h\right|\;\le\; B;\\
\sigma_1 h\;\le\; d_1\;\le\; h(\sigma_1+1), \; d_1\neq 0  \quad \&\quad  \sigma_2 h\;\le\; d_2\;\le\; h(\sigma_2+1), \; d_2\neq 0 ;\\
\tau_1g\;\le\; c_1\;\le\; g(\tau_1+1), \; c_1\neq 0  \quad \&\quad  \tau_2g\;\le\; c_2\;\le\; g(\tau_2+1), \; c_2\neq 0
\end{cases}
\end{align*}
is at most $36AB$.

\item[$\bm{(3)}$]  Given reals $\sigma', \tau'\in\R$, the number of pairs $\left(c,d\right)\in\Z^2$  solutions to the system 
\begin{align*}
\begin{cases}
gd\;=\; ch;\\
\sigma' h\;\le\; d\;<\; h(\sigma'+1);\\
\tau' g\;\le\; c\;<\; g(\tau'+1) 
\end{cases}
\end{align*}
is at most $\gcd(g,h)$.
\end{itemize}
\end{lem}

\begin{proof} Set $$a=\gcd(g,h), \qquad g'=\frac{g}{a}\qquad \textrm{and}\qquad h'=\frac{h}{a}\cdotp$$

\noindent $\bm{(1)}$  Fix an integer $n$ such that $1\le \left| n\right|\le A$ and consider the linear Diophantine equation in two variables
\begin{equation*}
gd-ch\;=\; n.
\end{equation*}
The above equation admits solutions in integers $\left(c,d\right)\in\Z^2$ if, and only if, $a$ divides $n$, in which case, upon letting $n'$ denote the integer $n/a$, it is equivalent to 
\begin{equation*}
g'd-ch'\;=\; n', \qquad \textrm{where}\qquad 1\;\le\; \left|n'\right|\;\le\; \frac{A}{a}\cdotp
\end{equation*}
Furthermore, if $\left(c_1, d_1\right)$ and $\left(c_2, d_2\right)$ are two particular distinct solutions, then $\left(c_2, d_2\right)=\left(c_1-tg', d_1-th'\right)$ for some integer $t\neq 0$. The bounds imposed on the admissible solutions imply that $ \left|c_2-c_1\right|= \left|tg'\right|\le g$, meaning that $1\le \left|t\right|\le a$. \\

\noindent This shows that for a fixed integer $n'$, the number of admissible pairs $\left(c,d\right)$ is at most $2a+1$. Since $n'$ takes at most $2A/a$ values, the total number of admissible pairs turn out to be at most $2A\cdot (2a+1)/a\le 6A.$\\

\noindent $\bm{(2)}$ The second point follows from the first one upon applying it separately to the pairs $\left(c_1, d_1\right)$ and $\left(c_2, d_2\right)$ and upon multiplying the corresponding estimates.\\

\noindent $\bm{(3)}$  The system under consideration is equivalent to the existence of an integer $\lambda$ such that  
\begin{align*}
\begin{cases}
c\;=\; \lambda g'\quad \& \quad d\;=\; \lambda h';\\
\sigma' a\;\le\; \lambda\;<\; a(\sigma'+1);\\
\tau' a\;\le\; \lambda\;<\; a(\tau'+1). 
\end{cases}
\end{align*}
Clearly, the number of integers $\lambda$ meeting   the last two sets of inequalities is at most $a$, whence the claim.
\end{proof}

\begin{proof}[Proof of Proposition~\ref{progreg3}] For fixed denominators $Q_1\le q, q'\le Q_2$, let 
\begin{equation*}
\mathfrak{C}_{\bm{\mathcal{T}}}\!\left(\left(q, q'\right); \left(\sigma, \tau\right)\right)\;=\; \left\{\left(\frac{\bm{s}}{q}, \frac{\bm{s}'}{q'}\right)\in \left(\Q_{\bm{\mathcal{T}}}\!\left(\sigma, \tau\right)\right)^2\;:\; \eqref{csubsm3}\; \textrm{holds}\right\}
\end{equation*}
and  recall the definition of the quantity $\delta_\xi\!\left(q,q'\right)$ in~\eqref{defdeltaiqq}. It should be clear that, in the notation introduced in the statement of the proposition, 
\begin{align}\label{upreg3supar}
\Theta_{\xi}^{(3)}\left(\bm{\mathcal{Q}};\bm{\mathcal{T}};  \left(\sigma, \tau\right)\right)\;\underset{\eqref{pistarprob}}{\le}\; \sum_{Q_1\le q,q'\le Q_2}\; \pi^{*}_{\bm{\mathcal{T}}}\!\left(\left(\sigma, \tau\right); \delta_\xi\!\left(q,q'\right) \right)\cdot \# \mathfrak{C}_{\bm{\mathcal{T}}}\!\left(\left(q, q'\right); \left(\sigma, \tau\right)\right).
\end{align} 
Here, $\# \mathfrak{C}_{\bm{\mathcal{T}}}\!\left(\left(q, q'\right); \left(\sigma, \tau\right)\right)$ stands for the cardinality of the set $\mathfrak{C}_{\bm{\mathcal{T}}}\!\left(\left(q, q'\right); \left(\sigma, \tau\right)\right)$. From the conditions~\eqref{csubsm3} defining this regime, it is equal to the number of integer quadruples $\left(s_1, s'_1, s_2, s'_2\right)$ such that
\begin{align*}
\frac{\bm{s}}{q}=\left(\frac{s_1}{q}, \frac{s_2}{q}\right)\in \Q_{\bm{\mathcal{T}}}\!\left(\sigma, \tau\right),\qquad  \frac{\bm{s}'}{q'}\;=\;\left(\frac{s'_1}{q'}, \frac{s'_2}{q'}\right)\in \Q_{\bm{\mathcal{T}}}\!\left(\sigma, \tau\right), \qquad \frac{\bm{s}}{q}\neq\frac{\bm{s}'}{q'}\;
\end{align*}
and
\begin{align*}
\begin{cases}
&\left|q's_1-qs'_1\right|\;\le\;  qq'\cdot H_{\xi}(q,q') ;\\
&\left|q's_2-qs'_2\right|\;\le\;   qq'\cdot h_\xi^{(\mu)}(q,q'). 
\end{cases}
\end{align*}
This   cardinality is estimated upon making a distinction of cases~:

\begin{itemize}
\item assume first that $\left|q's_1-qs'_1\right|\ge 1$ and that $\left|q's_2-qs'_2\right|\ge 1$. From Point~(2) in the Arithmetic Lemma~\ref{arithlem} (applied with obvious choices of parameters), the number of admissible quadruples is at most
\begin{equation*}
36\cdot \left(qq'\right)^2\cdot H_{\xi}(q,q') \cdot h_\xi^{(\mu)}(q,q');\\
\end{equation*}

\item assume then that $\left|q's_1-qs'_1\right|\ge 1$ and that $q's_2=qs'_2$. From Points~(1) and~(3) in the Arithmetic Lemma~\ref{arithlem}, the number of admissible quadruples is at most 
\begin{equation*}
6\cdot qq'\cdot H_{\xi}(q,q')  \cdot \gcd\!\left(q,q'\right);\\
\end{equation*}

\item assume finally that $q's_1=qs'_1$ and that $\left|q's_2-qs'_2\right|\ge 1$. As in the previous case, from Points~(1) and~(3) in the Arithmetic Lemma~\ref{arithlem}, the number of admissible quadruples is at most 
\begin{equation*} 
6\cdot qq'\cdot h_\xi^{(\mu)}(q,q')\cdot \gcd\!\left(q,q'\right);
\end{equation*}
\end{itemize}
\noindent As a consequence, 
\begin{align*}
\# \mathfrak{C}_{\bm{\mathcal{T}}}\!\left(\left(q, q'\right); \left(\sigma, \tau\right)\right)\;\le\; &36\cdot \left(qq'\right)^2\cdot H_{\xi}(q,q')\cdot h_\xi^{(\mu)}(q,q') \\
&+6\cdot qq'\cdot \left(H_{\xi}(q,q') +h_\xi^{(\mu)}(q,q')\right)\cdot \gcd\!\left(q,q'\right).
\end{align*}
The statement then follows from the inequality~\eqref{upreg3supar}.
\end{proof}

\subsection{Regime 4~: Squares with the same Centers}  \label{reg4}

The first step in the proof of  the upper bound~\eqref{estimsub4}  is to express the inner sum intervening in the definition of the sum of interest in this regime, namely $\Theta_{\xi}^{(4)}\left(\bm{\mathcal{Q}}; \bm{\mathcal{T}}\right)$ (defined in~\eqref{subsm4}), as a function of the global inhomogeneous sum $\mathcal{S}_{\xi}^{(\bm{\gamma})}\left(\bm{\mathcal{T}}; q\right)$ (defined  in~\eqref{defaxilfixq}). This is to then employ the sharp estimates established in Proposition~\ref{asymptinhomg} for $\mathcal{S}_{\xi}^{(\bm{\gamma})}\left(\bm{\mathcal{T}}; q\right)$. \\

\noindent To this end, it is convenient, given a pair of denominators $\bm{q}=\left(q, q'\right)\in \left(\Z_{\ge 1}\right)^2$, to set throughout this section
\begin{equation}\label{defggcdqq'}
g(\bm{q})=\gcd(q,q')\qquad \textrm{and} \qquad  \delta_\xi\!\left(q,q'\right)\;=\; \min\left\{\frac{\xi(q)}{q}, \frac{\xi(q')}{q'}\right\}.
\end{equation}

\begin{lem}[Reduction step]\label{lemreg4as}
Given integers $q,q'\ge 1$, let 
$$\mathcal{S}_{\xi(\bm{q}, \:\cdot\:)}\left(\bm{\mathcal{T}}; g(\bm{q})\right)\;=\; \sum_{\underset{\frac{\bm{s}}{g(\bm{q})}\in   \Q_{\xi(\bm{q}, \:\cdot\:)} ( \bm{\mathcal{T}})}{\bm{s}\in\Z^2~:}}\ppi_{\xi(\bm{q}, \:\cdot\:)}\!\left(\bm{\mathcal{T}}; \frac{\bm{s}}{g(\bm{q})}\right),$$
 where 
 \begin{equation}\label{defxibmq}
 \xi(\bm{q}, \:\cdot\:)~: g\;\mapsto\; g\cdot \xi(\bm{q}, g)\;=\; g\cdot \delta_\xi\!\left(q,q'\right)
 \end{equation}
(this corresponds to the  global inhomogeneous sum $\mathcal{S}_{\xi}^{(\bm{\gamma})}\left(\bm{\mathcal{T}}; q\right)$ defined  in~\eqref{defaxilfixq} specialised to the   homogeneous case $\bm{\gamma}=\left(0,0\right)$ and to the choice of the above approximation function $\xi(\bm{q}, \:\cdot\:)$). Then, the inner sum defining the subsum $\Theta_{\xi}^{(4)}\left(\bm{\mathcal{Q}}; \bm{\mathcal{T}}\right)$ defined in~\eqref{subsm4} equals  $\mathcal{S}_{\xi(\bm{q}, \:\cdot\:)}\left(\bm{\mathcal{T}}; g\right)$. This is saying that
\begin{equation*}
\mathcal{S}_{\xi(\bm{q}, \:\cdot\:)}\left(\bm{\mathcal{T}}; g(\bm{q})\right)\;=\;  \sum \ppi_\xi\left(\bm{\mathcal{T}}; \left(\frac{\bm{s}}{q}, \frac{\bm{s}'}{q'}\right)\right),
\end{equation*}
where the sum  is over the set of integers $\bm{s}, \bm{s}'$ such that $\bm{s}/q, \bm{s'}/q'\in \Q_{\xi}(\bm{\mathcal{T}})$ and  $\bm{s}/q=\bm{s'}/q'$.
\end{lem}

\begin{proof}
For fixed values of $q, q'$, the equality $\bm{s}/q=\bm{s'}/q'$ holds if, and only if, there exists an integer vector $\bm{u}\in\Z^2$ such that  $\bm{s}/q=\bm{s'}/q'= \bm{u}/g(\bm{q})$. Furthermore, the square centered at the rational point $\bm{u}/g(\bm{q})$ with side length $$\frac{\xi(\bm{q}, g(\bm{q}))}{g(\bm{q})}\;\underset{\eqref{defxibmq}}{=}\;  \delta_\xi\!\left(q,q'\right)$$ coincides with the intersection of the squares centered at the rationals $\bm{s}/q$  and $\bm{s'}/q'$  with respective side lengths $\xi(q)/q$ and $\xi(q')/q'$. The statement follows. 
\end{proof}

 \begin{prop}[Asymptotic expansion of the subsum $\Theta_{\xi}^{(4)}\left(\bm{\mathcal{Q}}; \bm{\mathcal{T}}\right)$]\label{asympreg4par} Assume that 
$T_1\;\ge\; 4/Q_1$. Then, 
\begin{align*}
\Theta_{\xi}^{(4)}\left(\bm{\mathcal{Q}}; \bm{\mathcal{T}}\right)\;&=\; \frac{2\sqrt{2}}{\sqrt{\pi}}\cdot \left(T_2-T_1\right)\cdot \sum_{Q_1\le q,q'\le Q_2} g(\bm{q})^2\cdot \sqrt{\delta_\xi\!\left(q,q'\right)} \\
&\qquad  +\; O\!\left(V\!\left(\bm{\mathcal{T}}\right)\cdot \sum_{Q_1\le q,q'\le Q_2} g(\bm{q}) \sqrt{\delta_\xi\!\left(q,q'\right)}\cdot \left(1+g(\bm{q})\cdot  \sqrt{\delta_\xi\!\left(q,q'\right)}\right)\right).
\end{align*}
\end{prop}

\begin{proof}
From the definition in~\eqref{subsm4} of the subsum $\Theta_{\xi}^{(4)}\left(\bm{\mathcal{Q}}; \bm{\mathcal{T}}\right)$ of interest in this regime and from the above Lemma~\ref{lemreg4as}, one has, in the notation of the latter statement,
\begin{align}\label{sumreg4qq'}
\Theta_{\xi}^{(4)}\left(\bm{\mathcal{Q}}; \bm{\mathcal{T}}\right)\;=\; \sum_{Q_1\le q,q'\le Q_2}\; \mathcal{S}_{\xi(\bm{q}, \:\cdot\:)}\left(\bm{\mathcal{T}}; g(\bm{q})\right).
\end{align} 
Recall here that $\mathcal{S}_{\xi(\bm{q}, \:\cdot\:)}\left(\bm{\mathcal{T}}; g(\bm{q})\right) = \mathcal{S}_{\xi}^{(\bm{\gamma})}\left(\bm{\mathcal{T}}; g(\bm{q})\right)$, where $\mathcal{S}_{\xi}^{(\bm{\gamma})}\left(\bm{\mathcal{T}}; \;\cdot\; \right)$ is the global inhomogeneous sum defined in~\eqref{defaxilfixq}, in the particular  case where   $\bm{\gamma}=\left(0,0\right)$  and where $\xi=\xi(\bm{q}, \:\cdot\:)$.  From Proposition~\ref{asymptinhomg}, under the assumption that $T_1\;\ge\; 4/Q_1$, and thus that 
\begin{equation}\label{b}
T_1\;\ge\; \frac{4}{Q_1}\;\ge\; 4\cdot \frac{\xi(\bm{q}, g(\bm{q}))}{g(\bm{q})}\;\underset{\eqref{defxibmq}}{=}\; 4\cdot\delta_\xi\!\left(q,q'\right),
\end{equation}
it admits the asymptotic expansion
\begin{align*}
\mathcal{S}_{\xi(\bm{q},\,\cdot)}
\left(\bm{\mathcal{T}}; g(\bm{q})\right)
&\underset{\eqref{defVT}}{=}
\frac{2\sqrt{2}}{\sqrt{\pi}}\,
\left(T_2-T_1\right)\,
g(\bm{q})^2\,
\sqrt{\delta_\xi\!\left(q,q'\right)}
\;+\; \\
&
O\!\left(
V\!\left(\bm{\mathcal{T}}\right)\cdot\left(
g(\bm{q})\sqrt{\delta_\xi\!\left(q,q'\right)}
\left(
1
+g(\bm{q})\sqrt{\delta_\xi\!\left(q,q'\right)}\right)\right)
+\Lambda_{\xi(\bm{q}, \:\cdot\:)}^{\left(\bm{0}, 2\right)}\left(\bm{\mathcal{T}}; g(\bm{q})\right)
\right). \\
\end{align*}
Here, $\Lambda_{\xi(\bm{q},\,\cdot)}^{\left(\bm{0},2\right)}
\left(\bm{\mathcal{T}}; g(\bm{q})\right)$ is a term which vanishes unless 
\begin{equation}\label{T_2-T_1last}
T_2-T_1\;\ge\; \frac{\xi(\bm{q}, g(\bm{q}))}{g(\bm{q})}\;\underset{\eqref{b}}{=}\; \delta_\xi\!\left(q,q'\right),
\end{equation}
in which case it is bounded above as
\begin{align*}
\Lambda_{\xi(\bm{q},\,\cdot)}^{\left(\bm{0},2\right)}
\left(
\bm{\mathcal{T}};\,
g(\bm{q})
\right)
&\;\le\;
g(\bm{q})\sqrt{\delta_\xi\!\left(q,q'\right)}
+
\sqrt{
\frac{
\xi\!\left(\bm{q},g(\bm{q})\right)
}{
g(\bm{q})\cdot T_2-\xi\!\left(\bm{q},g(\bm{q})\right)
}
}
\\
&\underset{\eqref{T_2-T_1last}}{\le}\;
g(\bm{q})\sqrt{\delta_\xi\!\left(q,q'\right)}
+
\sqrt{
\frac{
\xi\!\left(\bm{q},g(\bm{q})\right)
}{
g(\bm{q})\cdot T_1
}
}\\
&\underset{\eqref{defVT}\&\eqref{T_2-T_1last}}{\le}\; 2 V\!\left(\bm{\mathcal{T}}\right)\cdot g(\bm{q})\sqrt{\delta_\xi\!\left(q,q'\right)}
\end{align*}
As a consequence,
\begin{align*}
\mathcal{S}_{\xi(\bm{q},\,\cdot)}
\left(\bm{\mathcal{T}}; g(\bm{q})\right) &=\; \frac{2\sqrt{2}}{\sqrt{\pi}}\cdot \left(T_2-T_1\right)\cdot g(\bm{q})^2\cdot \sqrt{\delta_\xi\!\left(q,q'\right)}  \\
& \qquad \qquad +\; O\!\left(V\!\left(\bm{\mathcal{T}}\right)\cdot g(\bm{q}) \sqrt{\delta_\xi\!\left(q,q'\right)}\cdot \left(1+g(\bm{q})\cdot  \sqrt{\delta_\xi\!\left(q,q'\right)}\right)\right).
\end{align*}

The identity~\eqref{sumreg4qq'} then establishes the claim.
\end{proof}

\noindent The upper bound \eqref{estimsub4} in Regime 4 follows from the above statement~:

\begin{coro}[Upper Bound for the Subsum $\Theta_{\xi}^{(4)}\left(\bm{\mathcal{Q}}; \bm{\mathcal{T}}\right)$ under the assumption of monotonicity] Assume that $T_1\;\ge\; 3/Q_1$ and, furthermore, that the map $q\mapsto\xi(q)/q$ is nonincreasing. Then, 
\begin{align}\label{upbreg4concl}
\Theta_{\xi}^{(4)}\left(\bm{\mathcal{Q}}; \bm{\mathcal{T}}\right)\;&\ll\;  V\!\left(\bm{\mathcal{T}}\right)\cdot\sum_{q=Q_1}^{Q_2}\sqrt{q\cdot \xi(q)}\cdot \sigma\left(q\right),
\end{align}
where $\sigma(q)$ is the sum of the divisors of the integer $q$ defined in~\eqref{sumdiv}.
\end{coro}

\begin{proof}
Given integers $q\ge Q\ge 1$, set for the sake of this proof 
\begin{equation*}
D_2\left(Q, q\right)\;=\; \sum_{k=Q}^{q} \gcd(q,k)^2.
\end{equation*}
Splitting the sums intervening in the conclusion of Proposition~\ref{asympreg4par} into the two complementary ranges $Q_1\le q\le q'\le Q_2$ and $Q_1\le q'<q\le Q_2$,  it follows from the assumption of monotonicity under consideration that 
\begin{align*}
\Theta_{\xi}^{(4)}\left(\bm{\mathcal{Q}}; \bm{\mathcal{T}}\right)\;&\ll\;  V\!\left(\bm{\mathcal{T}}\right)\cdot\sum_{q=Q_1}^{Q_2}\sqrt{\frac{\xi(q)}{q}}\cdot D_2\left(Q_1, q\right).
\end{align*}
The conclusion is then a consequence of the (crude) estimate
\begin{equation*}
D_2\left(Q, q\right)\;\le\; \sum_{k=1}^q \gcd(q,k)^2\;\le\; \sum_{d|q}  d^2\cdot \#\left\{1\le m\le q\;:\; d|m\right\}\;\le\; q\cdot\sigma(q)
\end{equation*}
(obtained by bounding the cardinality by $q/d$).
\end{proof}

\begin{proof}[Completion of the proof of the estimates in Regime 4] It remains to establish the asymptotic expansion~\eqref{defMxibis} for the partial sums appearing on the right-hand side of the relation~\eqref{upbreg4concl}. To this end, given $Q\ge 1$, let $$T(Q)\;=\; \sum_{q=1}^{Q}q\cdot\sigma(q).$$ By partial summation, 
\begin{align}\label{partdumfun}
\sum_{q=1}^{Q}\sqrt{q\cdot \xi(q)}\cdot \sigma\left(q\right)\;&=\; \sqrt{\frac{\xi(Q)}{Q}}\cdot T(Q)\;+\; \sum_{q=1}^{Q-1}\left(\sqrt{\frac{\xi(q)}{q}}-\sqrt{\frac{\xi(q+1)}{q+1}}\right)\cdot T(q).
\end{align}
Furthermore, 
\begin{align*}
T(Q)\;&=\; \sum_{1\le ds\le Q}d^2s\;=\; \sum_{1\le s\le Q}s \sum_{d\le Q/s}d^2 \;=\; \sum_{1\le s\le Q}s \cdot \left(\frac{1}{3}\cdot \left(\frac{Q}{s}\right)^3+O\!\left(\left(\frac{Q}{s}\right)^2\right)\right)\\
&=\; \frac{\pi^2}{18}\cdot Q^3\;+\; O\!\left(Q^2\cdot \log Q\right).
\end{align*}
Substituting this relation back into~\eqref{partdumfun}, a short simplification based on reverse partial summations shows that 
\begin{align*}\label{partdumfunbis}
\sum_{q=1}^{Q}\sqrt{q\cdot \xi(q)}\cdot \sigma\left(q\right)\;&=\; \frac{\pi^2}{6}\cdot \sum_{q=1}^{Q}\sqrt{q^3\cdot\xi(q)}\;+\; O\!\left(\sum_{q=1}^Q \sqrt{q\cdot\xi(q)}\right)\\
&\qquad \qquad +\; O\!\left(\sum_{q=2}^{Q}\sqrt{\frac{\xi(q)}{q}}\cdot\left(q^2\log q - \left(q-1\right)^2\cdot \log (q-1)\right)\right)\\
&=\; \frac{\pi^2}{6}\cdot \sum_{q=1}^{Q}\sqrt{q^3\cdot\xi(q)}\;+\; O\!\left(\sum_{q=1}^Q \sqrt{q\cdot\xi(q)}\cdot \log q\right).
\end{align*}
This completes the proof of the relation~\eqref{defMxibis}.
\end{proof}

\section{Moment Estimates for Uniform Counting}\label{secfin}

\noindent This final section is devoted to the proof of Proposition~\ref{smeuf} (First and Second Moment Estimates for Uniform Counting). \\

\noindent In fact, the First Moment Estimate~\eqref{exp0} has already been established as part of the investigations in Subsection~\ref{completionproofbis}, and more precisely in~\eqref{asympexpbis}. It thus remains to prove the Second Moment Estimate~\eqref{var0}. \\

\begin{proof}[Proof {of} the Second Moment Estimate~\eqref{var0}] The proof relies on the distinction of cases induced by the four regimes introduced 
in Proposition~\ref{propestimregi}. The approximation function is chosen here as  the one defined in~\eqref{defpsideltaq}, namely $\xi_Q^{(\delta)}$. Consequently, the horizontal shift $H_{\xi_Q^{(\delta)}}(q,q')$ defined in~\eqref{defhH} admits the simpler form 
\begin{equation}\label{HQdelta}
H_Q^{\left(\delta\right)}\; :=\; H_{\xi_Q^{(\delta)}}(q,q')\;=\; {4}\cdot\frac{\delta(Q)}{Q} \qquad \textrm{ when }\qquad 1\le q, q'\le Q.
\end{equation}
The vertical shift $h_{\xi_Q^{(\delta)}}^{(\mu)}(q,q')$ also introduced in~\eqref{defhH} and intervening in the definition of Regimes~2 and~3 is, however, defined differently in this proof, namely as 
\begin{equation}\label{hQdelta}
h_Q^{\left(\delta\right)}\;:=\; h_{\xi_Q^{(\delta)}}(q,q')\;=\; C_0\cdot\sqrt{H_Q^{\left(\delta\right)}}\cdot \log\left(2Q\right) \qquad \textrm{ when }\qquad 1\le q, q'\le Q.
\end{equation}
Here, $C_0$ is a sufficiently large absolute constant chosen so that  
\begin{align*}
H_Q^{\left(\delta\right)}\le\min\left\{h_Q^{\left(\delta\right)}, \left(\frac{h_Q^{\left(\delta\right)}}{4}\right)^2\right\}\qquad \textrm{ and }\qquad h_Q^{\left(\delta\right)}<\frac{1}{2}
\end{align*}
for all integers 
\begin{equation}\label{QgeQ_0}
Q\; \ge\; Q_0, 
\end{equation} 
where $Q_0\ge 1$ is an auxiliary integer. The real $C_0$ and the integer $Q_0$ are introduced to ensure that the proofs based on the conditional probability decompositions  in Regimes~1 and~2 are still valid (in particular, for the chain of inequalities~\eqref{defe} to hold). \\

\noindent This new choice of parameters induces adjustments in the proof of the estimates in Proposition~\ref{propestimregi} in the various regimes. These adjustments, and the resulting conclusions,  are recorded hereafter. To this end, it is convenient to set 
\begin{equation*}
\Theta_{ \delta}^{*}\left(Q; \bm{\mathcal{T}}\right) \;=\; \Theta_{\xi}\left(\bm{\mathcal{Q}}; \bm{\mathcal{T}}\right)\qquad \textrm{and }\qquad \Theta_{ \delta}^{(i, *)}\left(Q; \bm{\mathcal{T}}\right)\; =\;  \Theta_{\xi}^{(i)}\left(\bm{\mathcal{Q}}; \bm{\mathcal{T}}\right)
\end{equation*}
whenever 
\begin{equation*}
\xi=\xi_Q^{(\delta)}, \qquad i=1, \dots, 4 \qquad \textrm{and} \qquad \bm{\mathcal{Q}}=(1,Q).
\end{equation*}
Here, $ \Theta_{\xi}\left(\bm{\mathcal{Q}}; \bm{\mathcal{T}}\right)$ is the quantity defined in~\eqref{ODPT} and  the subsums $\Theta_{\xi}^{(i)}\left(\bm{\mathcal{Q}}; \bm{\mathcal{T}}\right)$, where $1\le i\le 4$, are those introduced in the relations~\eqref{subsm1}, \eqref{subsm2}, \eqref{subsm3} and~\eqref{subsm4}. In particular, the decomposition~\eqref{decomthetabig} becomes
\begin{align}\label{decomthetabigbisdd}
\Theta_{ \delta}^{*}\left(Q; \bm{\mathcal{T}}\right) \;=\; \Theta_{ \delta}^{(1, *)}\left(Q; \bm{\mathcal{T}}\right) \;+\; \Theta_{ \delta}^{(2, *)}\left(Q; \bm{\mathcal{T}}\right) \;+\; \Theta_{ \delta}^{(3, *)}\left(Q; \bm{\mathcal{T}}\right) \;+\; \Theta_{ \delta}^{(4, *)}\left(Q; \bm{\mathcal{T}}\right) .
\end{align}
Taking into account the fact that the proof of Proposition~\ref{propestimregi} imposes that one deals with denominators larger than some integer  depending only on the  boundary para\-me\-ter $ \bm{\mathcal{T}}=\left(T_1, T_2\right)$  (this is the assumption~\eqref{asymptlageT_1T_2bis}), let  also
\begin{equation}\label{restricquan} 
\widehat{\Theta}_{ \delta}^{*}\left(Q; \bm{\mathcal{T}}\right) \qquad \textrm{and }\qquad  \widehat{\Theta}_{ \delta}^{(i, *)}\left(Q; \bm{\mathcal{T}}\right),
\end{equation}
where $1\le i\le 4$, denote the above sums restricted to those denominators $q\ge Q_{ \bm{\mathcal{T}}}$, where $Q_{ \bm{\mathcal{T}}}\ge 2$ is a free parameter. Similarly, define 
\begin{align*}
\Delta_{ \delta}^{*}\left(Q; \bm{\mathcal{T}}\right) \qquad \textrm{and} \qquad \widehat{\Delta}_{ \delta}^{*}\left(Q; \bm{\mathcal{T}}\right)
\end{align*}
from the Diagonal Probability Term $\Delta_{\xi}\left(\bm{\mathcal{Q}} ; \bm{\mathcal{T}}  \right)$ introduced in~\eqref{DPT} with the same meaning as above. \\

\noindent Adopting the same notation as in Proposition~\ref{propestimregi}, the above--mention adjustments in its proof and   conclusions now read as follows~: 

\begin{itemize}
\item \textbf{Regime~1}~: the proof in Subsection~\ref{reg1} runs smoothly when working with the map $\xi_Q^{(\delta)}$ so as to yield the same conclusion as in~\eqref{estimsub1}; namely that, provided that $Q_{ \bm{\mathcal{T}}}$ is large enough, for any $Q\ge Q_{ \bm{\mathcal{T}}}$,  
\begin{align}\label{theta*1}
& {\widehat{\Theta}_{ \delta}^{(1, *)}\left(Q; \bm{\mathcal{T}}\right)}\;=\; \left(\widehat{\Delta}_{ \delta}^{*}\left(Q; \bm{\mathcal{T}}\right)\right)^2+ \widehat{\Gamma}_{ \delta}^{*}\left(Q; \bm{\mathcal{T}}\right),
\end{align}
where the error term $\widehat{\Gamma}_{ \delta}^{*}\left(Q; \bm{\mathcal{T}}\right)$ meets the bound 
\begin{align}
\left|\widehat{\Gamma}_{ \delta}^{*}\left(Q; \bm{\mathcal{T}}\right)\right|\;&\ll\; \widehat{\Delta}_{ \delta}^{*}\left(Q; \bm{\mathcal{T}}\right)\cdot\left(\sum_{q=Q_{ \bm{\mathcal{T}}}}^{Q} \left(\sqrt{q\xi_Q^{(\delta)}}\cdot\left(1+\sqrt{q\xi_Q^{(\delta)}}+\log q\right)\right)\right)\nonumber\\
&\underset{\eqref{defpsideltaq}, \eqref{expdeldtqdel} \& \eqref{asympexp}}{\ll}\; \delta\cdot Q^4\cdot \log Q+\delta^{3/2}\cdot Q^{9/2}.\label{theta*1bis}
\end{align}
Here, and throughout, the implicit constants are allowed to depend on the boundary parameter $\bm{\mathcal{T}}$.\\

\item \textbf{Regime~2}~: this is the regime the adaptation of which requires the most attention. Since the condition~\eqref{defe} is met from the above choices of the parameters $C_0$ and $Q_0$, the proof in Subsection~\ref{reg2} runs smoothly up to the bound~\eqref{majreg2codn0} when replacing the quantities $\xi$, $H_{\xi}(q,q')$ and $h^{(\mu)}_{\xi}(q,q')$ with $\xi_Q^{(\delta)}$, $H_Q^{\left(\delta\right)}$ and $h_Q^{\left(\delta\right)}$, respectively. This {bound} thus becomes, adopting the remaining notations of Subsection~\ref{reg2}, 
\begin{align}
\ppi_{\xi}\left(\bm{\mathcal{T}}; \left( \frac{\bm{s}'}{q'} :   \bm{p}'\middle|\; \bm{p}\right)\right) \;&\ll \;  \sqrt{\beta}\cdot \left(\frac{1}{h_Q^{\left(\delta\right)}}+\frac{1}{\sqrt{x^+}}\right)\cdot  \exp\left(-\frac{\left(y^-\right)^2}{2x^+}\right).\nonumber
\end{align}
In the present case, $\beta=H_Q^{\left(\delta\right)}/2$ (see the definition of $\beta$ in~\eqref{boundreg2}). The monotonicity argument yielding the upper bound~\eqref{majreg2codn} in  the inequality following~\eqref{majreg2codn0} thus provides a simplified bound of the form
\begin{align}
\ppi_{\xi}\left(\bm{\mathcal{T}}; \left( \frac{\bm{s}'}{q'} :   \bm{p}'\middle|\; \bm{p}\right)\right)
\;&\underset{\eqref{boundreg2a}, \eqref{HQdelta}, \eqref{hQdelta}}{\ll}\;  \exp\left(-\frac{\left(h_Q^{\left(\delta\right)}/2 +\left|s'_2/q'-p_2\right|/4\right)^2}{6\cdot H_Q^{\left(\delta\right)}}\right)\nonumber\\
&\ll\; \exp\left(-\frac{\left(h_Q^{\left(\delta\right)}\right)^2/4 +\left(s'_2/q'-p_2\right)^2/16}{6\cdot H_Q^{\left(\delta\right)}}\right)\nonumber\\
\;&\underset{\eqref{HQdelta}\& \eqref{hQdelta}}{\ll}\;  \exp\left(-\frac{\left(C_0\cdot \log\left(2Q\right)\right)^2}{24}- \frac{\left(s'_2/q'-p_2\right)^2}{96\cdot H_Q^{\left(\delta\right)}}\right).\nonumber
\end{align}
The integer $s'_2$ is here restricted by the condition~\eqref{csubsm2} defining this regime. In the present case, it reads
\begin{equation}\label{reg2unifbisb}
\left|\frac{s'_2}{q'}-\frac{s_2}{q}\right|\;\ge\; h_Q^{\left(\delta\right)}.
\end{equation}
Continuing the calculations as in~\eqref{ineqproofreg2}, one is thus led to 
\begin{align}
&\sum_{\underset{\eqref{reg2unifbisb}}{s_2'\in\Z~: }} \ppi_\xi\left(\bm{\mathcal{T}}; \left(\frac{\bm{s}}{q}, \frac{\bm{s}'}{q'}\right)\right)\nonumber\\
&\ll  \ppi_{\xi}\left(\bm{\mathcal{T}}; \frac{\bm{s}}{q}\right)\cdot \exp\left(-\frac{\left(C_0\cdot \log\left(2Q\right)\right)^2}{24}\right)\cdot  \left(\sum_{s'_2\in\Z}  \exp\left(- \frac{\left(s'_2/q'-p_2\right)^2}{96\cdot H_Q^{\left(\delta\right)}}\right)\right) \nonumber \\
&\ll  \ppi_{\xi}\left(\bm{\mathcal{T}}; \frac{\bm{s}}{q}\right)\cdot \exp\left(-\frac{\left(C_0\cdot \log\left(2Q\right)\right)^2}{24}\right)\cdot  \left(1+q'\sqrt{H_Q^{\left(\delta\right)}}\right).\label{ineqproofreg2bisb}
\end{align}
The proof is concluded in the {same} vein as in Subsection~\ref{reg2}, with slight modifications~:  for fixed values of $s_1,q$ and $q'$, the number of integers $s'_1$ meeting the analogue of the first inequality in~\eqref{csubsm2} defining this regime, namely the inequality
\begin{equation}\label{defehjzd}
\left|\frac{s_1}{q}-\frac{s'_1}{q'}\right|\;<\;  H_Q^{\left(\delta\right)},
\end{equation}   
is $O\left(1+q' H_Q^{\left(\delta\right)}\right)$. Therefore, provided that $Q_{ \bm{\mathcal{T}}}$ is large enough, for any $Q\ge Q_{ \bm{\mathcal{T}}}$,   
\begin{align}
\sum_{Q_{ \bm{\mathcal{T}}}\le q'\le Q} \sum_{\underset{\eqref{reg2unifbisb}\& \eqref{defehjzd}}{\bm{s}'\in\Z^2~: }} \ppi_\xi\left(\bm{\mathcal{T}}; \left(\frac{\bm{s}}{q}, \frac{\bm{s}'}{q'}\right)\right)\;&\underset{\eqref{ineqproofreg2bisb}}{\ll}\; \ppi_\xi\left(\bm{\mathcal{T}}; \frac{\bm{s}}{q}\right)\cdot \exp\left(-\frac{\left(C_0\cdot \log\left(2Q\right)\right)^2}{24}\right)\times\nonumber \\
&\qquad  \sum_{Q_{ \bm{\mathcal{T}}}\le q'\le Q} \left(  \left(1+q'\sqrt{H_Q^{\left(\delta\right)}}\right)\cdot \left(1+q' H_Q^{\left(\delta\right)}\right)\right)\nonumber\\
&\ll\; \ppi_\xi\left(\bm{\mathcal{T}}; \frac{\bm{s}}{q}\right)\cdot \exp\left(-\frac{\left(C_0\cdot \log\left(2Q\right)\right)^2}{24}\right)\cdot Q^3\nonumber\\
&\ll\; \ppi_\xi\left(\bm{\mathcal{T}}; \frac{\bm{s}}{q}\right).\nonumber
\end{align}
Summing over the admissible values of the fraction $\bm{s}/q$, one obtains as in~\eqref{upbreg2} that 
\begin{align}\label{theta*2}
{ \widehat{\Theta}_{ \delta}^{(2, *)}\left(Q; \bm{\mathcal{T}}\right)}
\;\ll\; \widehat{\Delta}_{ \delta}^{*}\left(Q; \bm{\mathcal{T}}\right)\; \underset{\eqref{expdeldtqdel} \& \eqref{asympexp}}{\ll}\; \sqrt{\delta\cdot Q^5}.
\end{align}

\item \textbf{Regime~3}~: the proof in Subsection~\ref{reg3} makes no use of the specific expression for the vertical and horizontal shifts. Its conclusion, namely the bound~\eqref{estimsub3}, thus still holds in the case of the above--defined quantities $H_Q^{\left(\delta\right)}$ and $h_Q^{\left(\delta\right)}$. It claims that, provided that $Q_{ \bm{\mathcal{T}}}$ is large enough, for any $Q\ge Q_{ \bm{\mathcal{T}}}$,   
\begin{align}
{ \widehat{\Theta}_{ \delta}^{(3, *)}\left(Q; \bm{\mathcal{T}}\right)}
&\ll
\sum_{Q_{ \bm{\mathcal{T}}}\le q,q'\le Q}
\sqrt{\xi_Q^{(\delta)}(q)\cdot\xi_Q^{(\delta)}(q')}
\times
\Biggl(
(qq')^{3/2}\cdot\sqrt{H_Q^{(\delta)}}\cdot h_Q^{(\delta)}
\nonumber \\
&\qquad\qquad
+
\left(1+\frac{h_Q^{(\delta)}}{H_Q^{(\delta)}}\right)
\sqrt{\left(qq'\right)\cdot H_Q^{(\delta)}}{\cdot \min\{q,q'\}}
\Biggr).\nonumber\\
&\underset{\eqref{defpsideltaq}, \eqref{HQdelta} \& \eqref{hQdelta}}{\ll} \delta^2\cdot Q^4\cdot \log\left(2Q\right)+\sqrt{\delta^{3}\cdot Q^{7}} + \delta\cdot Q^4\cdot \log\left(2Q\right)\nonumber \\
&\ll \delta\cdot Q^4\cdot \log\left(2Q\right). \label{theta*3}
\end{align}

\item \textbf{Regime~4}~: since the map $q\mapsto {\xi_Q^{(\delta)}(q)/q}$ is nonincreasing, the proof in Subsection~\ref{reg4} runs smoothly when working with the approximation function $\xi_Q^{(\delta)}$ so as to yield the same conclusion as in~\eqref{estimsub4}; namely that, provided that $Q_{ \bm{\mathcal{T}}}$ is large enough, for any $Q\ge Q_{ \bm{\mathcal{T}}}$,  
\begin{align*}
{ \widehat{\Theta}_{ \delta}^{(4, *)}\left(Q; \bm{\mathcal{T}}\right)}\;&\ll\; \sum_{q= Q_{ \bm{\mathcal{T}}}}^{Q}\sqrt{q\cdot \xi_Q^{(\delta)}}\cdot\sigma(q) \; \underset{\eqref{defpsideltaq}}{\ll}\;\sqrt{\frac{\delta}{Q}}\cdot  \left(\sum_{q= Q_{ \bm{\mathcal{T}}}}^{Q}q\cdot \sigma(q)\right),
\end{align*}
where $\sigma$ denotes the sum of divisors function. Since the latter  meets the well-known bound $\sum_{1\le q\le Q}q\cdot \sigma(q)\ll Q^3$, this yields 
\begin{align}\label{theta*4}
{ \widehat{\Theta}_{ \delta}^{(4, *)}\left(Q; \bm{\mathcal{T}}\right)}\;&\ll\; \sqrt{\delta Q^5}.
\end{align}
\end{itemize}

\noindent From the decomposition~\eqref{decomthetabigbisdd} and from the definition of the restricted quantities in~\eqref{restricquan}, one infers from the above--established relations~\eqref{theta*1}, \eqref{theta*1bis}, \eqref{theta*2}, \eqref{theta*3}  and~\eqref{theta*4} that, provided that   the integer $Q$ is larger than the threshold $Q_0\ge 1$ and than some integer $Q_{ \bm{\mathcal{T}}}\ge 1$ depending only on the  boundary parameter $ \bm{\mathcal{T}}=\left(T_1, T_2\right)$, 
\begin{align}\label{finalvarrestc}
\widehat{\Theta}_{ \delta}^{*}\left(Q; \bm{\mathcal{T}}\right) \;=\; \left(\widehat{\Delta}_{ \delta}^{*}\left(Q; \bm{\mathcal{T}}\right)\right)^2+O\left(\delta\cdot Q^4\cdot \log Q+\sqrt{\delta^{3}\cdot Q^{9}}+\sqrt{\delta Q^5}\right).
\end{align}

\noindent Furthermore, letting $$Q_{ \bm{\mathcal{T}}}^{(0)}\;=\; \max\left\{Q_{ \bm{\mathcal{T}}}, Q_0\right\}-1,$$ 
equations~\eqref{deltaexpec} and~\eqref{thetafomeg}, make it clear that 
\begin{align*}
\E\left[\mathcal{N}^{\,\flat}\!\left(\bm{\mathcal{T}};\left(\delta, Q\right)\right)-\mathcal{N}^{\,\flat}\!\left(\bm{\mathcal{T}};\left({\frac{\delta\cdot Q_{ \bm{\mathcal{T}}}^{(0)}}{Q}}, Q_{ \bm{\mathcal{T}}}^{(0)}\right)\right)\right]\;=\; \widehat{\Delta}_{ \delta}^{*}\left(Q; \bm{\mathcal{T}}\right) 
\end{align*}
and that 
\begin{align*}
\E\left[\left(\mathcal{N}^{\,\flat}\!\left(\bm{\mathcal{T}};\left(\delta, Q\right)\right)-\mathcal{N}^{\,\flat}\!\left(\bm{\mathcal{T}};\left({\frac{\delta\cdot Q_{ \bm{\mathcal{T}}}^{(0)}}{Q}}, Q_{ \bm{\mathcal{T}}}^{(0)}\right)\right)\right)^2\right]\;=\; \widehat{\Theta}_{ \delta}^{*}\left(Q; \bm{\mathcal{T}}\right) .
\end{align*}
As a consequence, one deduces from~\eqref{finalvarrestc} that 
\begin{align*}
&\V\left[\mathcal{N}^{\,\flat}\!\left(\bm{\mathcal{T}};\left(\delta, Q\right)\right)\right]\\
&=\; \V\left[\left(\mathcal{N}^{\,\flat}\!\left(\bm{\mathcal{T}};\left(\delta, Q\right)\right)-\mathcal{N}^{\,\flat}\!\left(\bm{\mathcal{T}};\left({\frac{\delta\cdot Q_{ \bm{\mathcal{T}}}^{(0)}}{Q}}, Q_{ \bm{\mathcal{T}}}^{(0)}\right)\right)\right)+  \mathcal{N}^{\,\flat}\!\left(\bm{\mathcal{T}};\left({\frac{\delta\cdot Q_{ \bm{\mathcal{T}}}^{(0)}}{Q}}, Q_{ \bm{\mathcal{T}}}^{(0)}\right)\right) \right]\\
&\le 2\cdot \V\left[\mathcal{N}^{\,\flat}\!\left(\bm{\mathcal{T}};\left(\delta, Q\right)\right)-\mathcal{N}^{\,\flat}\!\left(\bm{\mathcal{T}};\left({\frac{\delta\cdot Q_{ \bm{\mathcal{T}}}^{(0)}}{Q}}, Q_{ \bm{\mathcal{T}}}^{(0)}\right)\right)\right]+ 2\cdot \V\left[ \mathcal{N}^{\,\flat}\!\left(\bm{\mathcal{T}};\left({\frac{\delta\cdot Q_{ \bm{\mathcal{T}}}^{(0)}}{Q}}, Q_{ \bm{\mathcal{T}}}^{(0)}\right)\right)\right]\\
&\ll\; \delta\cdot Q^4\cdot \log Q+\sqrt{\delta^{3}\cdot Q^{9}}+\sqrt{\delta Q^5}+ \V\left[ \mathcal{N}^{\,\flat}\!\left(\bm{\mathcal{T}};\left({\frac{\delta\cdot Q_{ \bm{\mathcal{T}}}^{(0)}}{Q}}, Q_{ \bm{\mathcal{T}}}^{(0)}\right)\right)\right]\\
&\ll\; \delta\cdot Q^4\cdot \log Q+\sqrt{\delta^{3}\cdot Q^{9}}+\sqrt{\delta Q^5}+ \E\left[ \mathcal{N}^{\,\flat}\!\left(\bm{\mathcal{T}};\left({\frac{\delta\cdot Q_{ \bm{\mathcal{T}}}^{(0)}}{Q}}, Q_{ \bm{\mathcal{T}}}^{(0)}\right)\right)^2\right].
\end{align*} 
To conclude, it remains to notice that when a rational $\bm{s}/q=\left(s_1/q, s_2/q\right)$ intervenes in the counting function $\mathcal{N}^{\,\flat}\!\left(\bm{\mathcal{T}};\left({\delta\cdot Q_{ \bm{\mathcal{T}}}^{(0)}/Q}, Q_{ \bm{\mathcal{T}}}^{(0)}\right)\right)$, then $q$, and therefore $s_1$ are bounded as functions of the boundary times $\bm{\mathcal{T}}$ and of the integer $Q_0$. As for the integer $s_2$, it is clearly bounded by $q\cdot \left(1+\max_{T_1\le t\le T_2}\left|B(t)\right|\right)$ in such a way that 
\begin{align*}
 \E\left[ \mathcal{N}^{\,\flat}\!\left(\bm{\mathcal{T}};\left(\delta, Q_{ \bm{\mathcal{T}}}^{(0)}\right)\right)^2\right]\;&\ll\; 1+ \E\left[ \left(\max_{T_1\le t\le T_2}{\left|B(t)\right|}\right)^2\right]\\
 &=\; 1+2\cdot \int_{0}^{\infty}t\cdot \bP\left(\max_{T_1\le t\le T_2}{\left|B(t)\right|}\ge t\right)\cdot\textrm{d}t,
\end{align*}
where the implicit constant is allowed to depend on $\bm{\mathcal{T}}$ and  on $Q_0$ and where the last equation follows from Fubini's Theorem. It is now an easy consequence of the Law of Maximum~\eqref{Mt} that the above integral is finite, implying that 
\begin{align*}
\V\left[\mathcal{N}^{\,\flat}\!\left(\bm{\mathcal{T}};\left(\delta, Q\right)\right)\right]\;\ll\; 1+\delta Q^4\cdot \log Q+\sqrt{\delta^{3}\cdot Q^{9}}+\sqrt{\delta Q^5}.
\end{align*} 
This reduces to the Second Moment Estimate stated in~\eqref{var0} under the assumption~\eqref{assumpsecmom} that $\delta Q^5\ge 1$.
\end{proof}

\bibliographystyle{unsrt}
\addcontentsline{toc}{section}{References}

\end{document}